\documentclass[11pt, a4paper]{article}

\usepackage{inputenc} 
\usepackage[english]{babel}
\usepackage{dsfont}

\usepackage[totalheight=24 true cm, totalwidth=17 true cm]{geometry}
\usepackage{enumitem}

\setlist[itemize]{noitemsep}

\usepackage{bbding}
\usepackage{amsmath,multirow}
\usepackage{amsthm}
\usepackage{amssymb}
\usepackage{comment}
\usepackage{mathrsfs}
\usepackage{mathtools}
\usepackage{ stmaryrd }
\usepackage{url}
\usepackage{wrapfig}
\usepackage{starfont}
\usepackage{pifont}
\usepackage{eurosym}
\usepackage{subcaption}

\usepackage{imakeidx}
\makeindex[]

\usepackage{multicol}
\usepackage{relsize}
\usepackage{float}
\usepackage{tikz,pgf}
\usepackage{tikz-cd}
\usepackage{wasysym}
\usepackage{graphicx}
\usepackage{fancyhdr}
\usepackage{verbatim}

\usepackage{pgfplots}
\pgfplotsset{compat=1.15}
\usepackage{mathrsfs}
\usetikzlibrary{arrows}

\usepackage[all,dvips,knot,web,arc,curve,color,frame]{xy}
\usepackage[all,knot,arc,color,web]{xy}
\xyoption{arc}
\xyoption{web}
\xyoption{curve}

\numberwithin{equation}{section}

\theoremstyle{plain}
\newtheorem{theorem}{Theorem}[section]
\newtheorem{proposition}[theorem]{Proposition}
\newtheorem{corollary}[theorem]{Corollary}
\newtheorem{lemma}[theorem]{Lemma}
\newtheorem{notation}[theorem]{Notation}
\newtheorem{setup}[theorem]{Setup}

\theoremstyle{definition}
\newtheorem{definition}[theorem]{Definition}

\newtheorem{remark}[theorem]{Remark}

\newtheorem{example}[theorem]{Example}
\newtheorem{question}[theorem]{Question}
\usepackage[bookmarksnumbered=true]{hyperref} 
\hypersetup{
     colorlinks = true,
     linkcolor = blue,
     anchorcolor = blue,
     citecolor = teal,
     filecolor = blue,
     urlcolor = blue
     }
\newcommand\restr[2]{{% we make the whole thing an ordinary symbol
  \left.\kern-\nulldelimiterspace % automatically resize the bar with \right
  #1 % the function
  \vphantom{\big|} % pretend it's a little taller at normal size
  \right|_{#2} % this is the delimiter
  }}

\everymath{\displaystyle}
\allowdisplaybreaks

\makeatletter
\def\mathcenterto#1#2{\mathclap{\phantom{#1}\mathclap{#2}}\phantom{#1}}
\let\old@widetilde\widetilde
\def\widetildeto#1#2{\mathcenterto{#2}{\old@widetilde{\mathcenterto{#1}{#2\,}}}}
\let\old@widehat\widehat
\def\widehatto#1#2{\mathcenterto{#2}{\old@widehat{\mathcenterto{#1}{#2\,}}}}
\makeatother

\newcommand{\ip}[2]{\langle  #1,#2 \rangle} % \ip a b le pone < >
\newcommand{\is}[1]{\langle  #1 \rangle} % \ip a b le pone < >
\newcommand{\size}[1]{\left| #1 \right|} % \size a b le pone | |
\newcommand{\pare}[1]{\left (#1 \right)} % \pare a b le pone  ()
\newcommand{\set}[1]{{\left\{ #1 \right\}}} % \set a b le pone { }

\newcommand{\corch}[1]{\left[ #1 \right]} % \pare a b le pone [ ]

\newcommand*\closure[1]{\overline{#1}}

\DeclareMathOperator{\rank}{rk}

\DeclareMathOperator{\lift}{lift}
\DeclareMathOperator{\gr}{graph}
\def\dim{\operatorname{dim}}

\DeclareMathOperator{\sign}{sign}

\DeclareMathOperator{\CC}{\mathbb{C}}

\hypersetup{
    colorlinks=true,
    linkcolor=blue,
    filecolor=magenta,      
    urlcolor=cyan,
    pdftitle={Overleaf Example},
    pdfpagemode=FullScreen,
    }
    
\def\Id{\operatorname{Id}}

\setlist[itemize]{noitemsep}

\normalfont

\usepackage{bbm}
\usepackage{dsfont}

\usepackage{url}
\usepackage{amsmath,blkarray,booktabs, bigstrut}
\usepackage{tikz}
\usepackage{multirow}
\usepackage{xcolor}
\usepackage{color}
\usepackage[maxbibnames=99]{biblatex}
\title{Paving Matroids: Defining Equations and Associated Varieties
}
\author{Emiliano Liwski and Fatemeh Mohammadi}
\begin{document}
\maketitle

\begin{abstract}
\noindent We study paving matroids, their realization spaces, and their closures, along with matroid varieties and circuit varieties. Within this context, we introduce %three 
two distinct methods for generating polynomials within the associated ideals of these varieties across any dimension. Additionally, we explain the relationship between polynomials constructed using these different methods. We then compute a comprehensive and finite set of defining equations for matroid varieties associated with specific classes of paving matroids. 
Furthermore, we provide a decomposition for the associated circuit variety of 
paving matroids, where all points have a degree less than $3$. Finally, we present several examples applying our results and compare them with known cases in the literature.

\end{abstract}

{\hypersetup{linkcolor=black}
{\tableofcontents}}

%\vspace{50pt}

\section{Introduction}
We study paving matroids from an algebro-geometric viewpoint. A paving matroid $M$ is a matroid in which every circuit has a size equal to either the rank $r(M)$ or $r(M)+1$, where $r(M)$ represents the rank of $M$. We call a paving matroid an $n$-paving matroid if its rank $r(M)$ equals $n$. It is conjectured in \cite{mayhew2011asymptotic} that asymptotically almost all matroids are paving; see also~\cite{lowrance2013properties}. 
~Our focus lies on the matroid variety $V_{M}$, defined as the Zariski closure of the matroid's realization space, which was originally introduced in \cite{gelfand1987combinatorial} and has since been extensively studied; see~e.g.~\cite{clarke2021matroid,sidman2021geometric,Fatemeh3}. 
\smallskip

The primary question is to compute the defining equations, or equivalently, the ideal of the matroid variety $I_{M}=I(V_{M})$ and to represent such polynomials geometrically or combinatorially. Equally significant is the task of determining an irredundant irreducible decomposition of the matroid variety $V_{M}$. 
In \cite{mnev1988universality}, the Mnëv-Sturmfels Universality Theorem~demonstrates that matroid varieties adhere to the ``Murphy's Law in Algebraic Geometry," indicating that any singularity in a semi-algebraic set corresponds to a matroid variety with the same singularity. Additionally, %Knutson, Lam, and Speyer noted 
it is noted in \cite{knutson2013positroid} that the irreducible decomposition of $V_{M}$ remains a difficult task in general and is only known for specific examples of matroids.
Moreover, it is known that the circuit ideal $I_{\mathcal{C}(M)}$ (see Definition~\ref{cir}), generated by the polynomials associated with the circuits of $M$, is always a subset of $I_{M}$. However, the reverse is rarely true. In the case of a positroid, it was demonstrated in \cite{knutson2013positroid} that the two ideals $I_{\mathcal{C}(M)}$ and $I_{M}$ coincide. Additionally, positroid varieties are also irreducible. Indeed, positroids are in bijection with various combinatorial objects that encode their algebraic invariants, offering profound insights into their structure and properties; see,~e.g.~\cite{knutson2013positroid, postnikov2006total,mohammadi2022computing,mohammadi2024combinatorics}. %However, this is not true in most cases; 
In \cite{computationalgorithms}, Sturmfels introduced a minimal example where $I_{\mathcal{C}(M)}$ is strictly contained in the ideal $I_{M}$, and $I_{M}$ is reducible.

In \cite{sidman2021geometric}, the authors computed some polynomials in $I_{M}\backslash I_{\mathcal{C}(M)}$ using the Grassmann-Cayley algebra, for a specific family of point-line configurations. A classical method for computing such polynomials within $I_M$ applies the Grassmann-Cayley algebra, which offers a combinatorial approach. However, fully generating all polynomials in $I_M$, or verifying whether those generated by the Grassmann-Cayley method are sufficient to generate $I_{M}$, remains challenging. In \cite{clarke2021matroid}, the ideal $I_{M}$ is shown to coincide with the saturation of $I_{\mathcal{C}(M)}$ with respect to the product of all minors corresponding to the bases of the matroid; see also~\cite{white2017geometric, sturmfels1989matroid, Stu4, sitharam2017handbook, sidman2021geometric, caminata2021pascal}. However, performing this saturation can be computationally challenging, which motivates the search for alternative descriptions of the defining ideal $I_{M}$.
%This challenge arises due to the saturation step included in the matroid ideal's construction, where the circuit ideal is saturated with the product of polynomials corresponding to the bases
%basis elements 
%of the matroid; see e.g.~\cite{white2017geometric, sturmfels1989matroid, Stu4, sitharam2017handbook, sidman2021geometric, caminata2021pascal}. 

Existing algorithms for computing the saturation of ideals are characterized by high complexity, thereby limiting their practicality to small-scale matroids. Furthermore, the resulting polynomials often lack geometric insight, being lengthy and challenging to interpret. This limitation is illustrated by examples in \cite{pfister2019primary, Fatemeh3}. 
Given the complexity associated with computing a generating set of $I_M$, there has been considerable interest in generating certain polynomials, albeit not necessarily the entire generating set, for the ideal $I_M$, particularly for specific families of matroids. 
For instance, in \cite{sidman2021geometric}, the authors applied the Grassmann-Cayley algebra in combination with geometric tools from \cite{caminata2021pascal, eisenbud1996cayley} to generate polynomials for a family of examples of point-line configurations.
Another geometric approach, not relying on the Grassmann-Cayley algebra, was applied in \cite{Fatemeh3} for rank-3 matroids, leading to a minimal generating set for the matroid ideals of quadrilateral sets and the $3\times 4$ grid.  Such ideals also appear in the context of conditional independence statements in algebraic statistics; see, for example, \cite{clarke2022conditional, clarke2020conditional, caines2022lattice, mohammadi2018prime, ene2013determinantal, alexandersson2026decomposing, Mohammadi5, Mohammadi4}.
Our work extends these geometric techniques to matroids of arbitrary rank. 

%In \cite{clarke2021matroid},  a decomposition strategy for computing the irreducible decomposition of matroid varieties of point-line configurations has been introduced, which can be viewed as an analogue of the Grassmannian stratification by matroids from \cite{gelfand1987combinatorial}.

\medskip
\noindent{\bf Our contributions.}
We will now summarize the main results of this paper. %, and the rank $r(M)$ is $n$. %or $r(M)+1$, where $r(M)$ represents the rank of $M$.
In Section~\ref{constructing}, we introduce the concept of liftability, which serves as a geometric technique for constructing polynomials within the ideal $I_{M}$. This idea is motivated by the following question. 
\begin{question}\label{question_intro}
Consider an $n$-paving matroid $M$, i.e.~$M$ is of rank $n$, and every circuit of $M$ is either of size  $n$ or $n+1$. Let $\gamma$ be a collection of vectors in $\mathbb{C}^{n}$ indexed by the ground set of $M$, lying within a hyperplane $H$ of $\mathbb{C}^{n}$, with $q \in \mathbb{C}^{n} \setminus H$. What condition is both necessary and sufficient for the existence of a lifting of the vectors from $\gamma$ at point $q$ to a non-degenerate collection $\widetilde{\gamma} \in V_{\mathcal{C}(M)}$? 
\end{question}

We address this question in Lemma~\ref{lift}, where we establish that the desired lifting exists if and only if certain minors of a so-called liftability matrix vanish on the vectors of $\gamma$. 
More precisely, for any vector $q \in \mathbb{C}^{n}$ and any $n$-paving matroid $M$, we define the \textit{liftability matrix} $\mathcal{M}_{q}(M)$ with columns indexed by the elements of $M$ and rows indexed by the $n$-circuits of $M$. For each $n$-circuit $c=\{c_{1},c_{2},\ldots,c_{n}\}$, the $c_{i}^{\rm th}$ coordinate of the corresponding row is given by the polynomial 
$(-1)^{i-1}\left[c_{1},c_{2},\ldots,c_{i-1},\hat{c_{i}},c_{i+1},\ldots, c_{n},q\right]$,
while the other entries of the row are set to $0$.

\medskip
Notably, projecting the vectors from any realization of an $n$-paving matroid $M$ onto a hyperplane in $\mathbb{C}^{n}$, from a  vector outside of it, ensures that the projected vectors satisfy the condition for non-degenerate lifting to exist. This observation holds true for any submatroid of full rank within $M$, as well. Using this geometric insight, we can construct polynomials within $I_{M}$.

\begin{theorem}[Theorem~\ref{sub}] Let $M$ be an $n$-paving matroid and $N$ be a submatroid of $M$ with full rank. % (rank $n$). 
Then, for any $q\in \mathbb{C}^{n}$, the $(|N|-n+1)$ minors of the matrix $\mathcal{M}_{q}(N)$ are polynomials in~$I_{M}$.
\end{theorem}

\noindent
The polynomials constructed in this manner are called {\em lifting polynomials}, and we denote by $I_{M}^{\lift}\subset I_{M}$ the ideal generated by all these lifting polynomials.

\medskip

In Section~\ref{sec 4}, we focus on $n$-paving matroids with no points of degree greater than two, where each point is incident to at most two different dependent hyperplanes. 
For rank $3$ matroids, this includes point-line configurations with no three concurrent lines. {We present a set of defining equations for the matroid varieties of these paving matroids.}
%We present a complete generating set, not necessarily finite, for the ideals of such paving matroids.

\begin{theorem}[Theorem~\ref{theo lift}]
Let $M$ be an $n$-paving matroid with no points of degree greater than two. Then, the circuit and lifting polynomials generate the matroid ideal, up to radical, as follows:
\[I_{M} = \sqrt{I_{\mathcal{C}(M)}+I_{M}^{\lift}}.\]
\end{theorem} 
Applying the above theorem, which provides a comprehensive description of the defining equations of the matroid varieties of this type, we proceed to compare the matroid variety and the circuit variety of $M$. 
In the following theorem, we summarize our main results from Section~\ref{sec 4}.
Below, we use the notation $V_{\mathcal{C}(M)}$ and $V_M$ for the circuit variety and the matroid variety, respectively.  We denote $V_{0}$ for the matroid variety of the uniform matroid $U_{n-1
,|M|}$ of rank $n-1$ on $|M|$ elements.
\begin{theorem}\label{thm:intro} Let $M$ be an $n$-paving matroid with no points of degree greater than two. Then:
\begin{itemize}
    \item  If all submatroids of %dependent hyperplanes of $M$ 
    are liftable, then $V_{\mathcal{C}(M)}=V_{M}.$\hfill{\rm(Proposition~\ref{ico})}
   \item[] \item If all proper submatroids of %dependent hyperplanes of 
   $M$ are liftable, then $V_{\mathcal{C}(M)}=V_{M}\cup V_{0}$.  %\textcolor{blue}
    \hfill{\rm(Proposition~\ref{des 2})}
\end{itemize}
\end{theorem}

\medskip

In Section~\ref{sec 5}, we introduce a combinatorial approach for generating polynomials within $I_{M}$. Each polynomial produced by this method is uniquely linked to a specific graph. We demonstrate that these polynomials, along with the circuit polynomials, collectively form a generating set for the matroid ideal, up to radical. Importantly, we establish that this generating set is finite, unlike the set derived from the lifting polynomials discussed in Section~\ref{sec 4}.

\medskip

Let $M$ be an $n$-paving matroid. Consider two subsets of elements of $M$, $J$ and $P=\{p_{1},\ldots,p_{k}\}$, such that $P\cap \overline{J}=\emptyset$, where $\closure{J}$ denotes the matroid closure. Suppose $C=\{c_{1},\ldots,c_{k}\}$ is a set of circuits of size $n$ whose points belong to $J\cup P$, and $p_{i}\in c_{i}$ for every $i\in [k]$. 
%For any realization $\gamma \in \Gamma_{M}$,
 We construct a directed weighted 
graph $G$, with vertex set $V(G)=P$ and edge set $E(G)=\{(p_{i},p_{j}):p_{j}\in c_{i}\}$. Each edge $(p_{i},p_{j})\in E(G)$ is assigned a weight $\alpha_{i,j}$, an indeterminate. %where the weights are the coefficients of some linear dependencies on the vectors $\{\gamma_{p}:p\in P\}$ arising from the circuits $c_{i}$. 
%In Theorem \ref{thm: first way}, we establish that the weights of the graph $G$ satisfy a polynomial relation, as expressed in Equation~\eqref{eq r}. 
In Remark~\ref{rem 5}, we illustrate that these weights can be substituted with quotients of polynomial brackets by selecting an additional set of vectors $\{q_{p}:p\in P\}\subset \mathbb{C}^{n}$. Theorem~\ref{thm ci} provides a polynomial in these brackets, which %vanishes at $\gamma$. Since $\gamma$ is a realization of $M$, we obtain a polynomial within 
belongs to $I_{M}$. We refer to such polynomials as {\em graph polynomials}, and we denote the ideal generated by all graph polynomials by $I_M^{\gr}$.

\begin{theorem}[Theorem~\ref{inc} and Corollary~\ref{fin}]
Let $M$ be an $n$-paving matroid. Then $I_{M}^{\lift}\subset I_{M}^{\gr}$. If in addition $M$ has no points of degree greater than two, 
then the circuit polynomials and graph polynomials together form a {\em finite} generating set for the matroid ideal, up to radical, as follows:
\[I_{M} = \sqrt{I_{\mathcal{C}(M)}+I_{M}^{\gr}}.\]
%Furthermore, this provides a finite generating set for $I_M$.
\end{theorem}

\medskip\noindent{\bf Related work.}
Before reviewing the contents of the paper, we would like to comment on some related works. In \cite{computationalgorithms}, Sturmfels introduced the concept of projection, focusing specifically on the quadrilateral set. He provided an explicit polynomial condition for the liftability of six collinear points to a quadrilateral set and algorithmically derived their standard bracket representation. In \cite{richter2011perspectives},  Richter-Gebert applied a similar projection technique to characterize the liftability of six points on a line to a quadrilateral set using determinant equations involving the center of projection. A generalization of this liftability technique was also presented in \cite{Fatemeh3} for 3-rank matroids and was applied to explicitly compute a minimal generating set for the associated ideals of a family of matroids of rank $3$. %the quadrilateral sets and the matroid of the $3 \times 4$ grid.
In our work, we extend these techniques to any paving matroid of any dimension by projecting onto a hyperplane instead of a line.

In \cite{clarke2021matroid}, it was established that forest-type point-line configurations are realizable. Proposition~\ref{real 2} demonstrates that $n$-paving matroids with no points of degree greater than two are also realizable. This expansion broadens the scope of known realizable matroids.
We also note that the polynomials constructed using our methods differ from those constructed using the Grassmann-Cayley algebra. For example, concerning the Pascal matroid studied in \cite%[Theorem 3.0.1]
{sidman2021geometric}, Proposition~\ref{notin} illustrates that the polynomials constructed in Theorem \ref{thm ci} are not contained in the ideal generated by the polynomials constructed in \cite{sidman2021geometric} using the Grassmann-Cayley algebra. Additionally, in \cite[Theorem 4.2.2]{sidman2021geometric}, the authors constructed a matroid from $d+4$ points on a rational normal curve of degree $d$ via the Grassmann-Cayley algebra. Applying Theorem~5.1 of \cite{caminata2021pascal}, they demonstrated the existence of non-trivial quartics in the associated matroid ideal using the Grassmann-Cayley algebra. In Example~\ref{ext}, we construct additional polynomials for such matroids. 
%%%%%%%%%%%%%%% 

\medskip\noindent{\bf Outline.}
%We now provide an overview of the paper. 
In Section~\ref{2}, we introduce key concepts such as paving matroids, point-line configurations, and matroid varieties. Section~\ref{constructing} outlines a geometric method to generate lifting polynomials within the matroid ideals. Applying this in Section~\ref{sec 4}, we derive defining equations for paving matroids, albeit infinite. Section~\ref{sec 5} introduces a combinatorial method for polynomial generation. In Section~\ref{rel 6}, we discuss the relationship between these methods, yielding a finite %generating 
set of defining equations for the associated %ideals 
matroid varieties of paving matroids. Section~\ref{sec 7} features various examples and applications, including constructing polynomials for the matroid ideals studied in \cite{sidman2021geometric, clarke2021matroid}.

\medskip\noindent{\bf Acknowledgement.} The authors would like to thank Oliver Clarke and Giacomo Masiero for helpful discussions. E.L. is supported by the FWO grant G0F5921N. F.M. was partially supported by the grants G0F5921N (Odysseus programme) and G023721N from the Research Foundation - Flanders (FWO), the UiT Aurora project MASCOT and the grant iBOF/23/064 from the KU Leuven.

\section{Preliminaries}\label{2}

We first fix our notation throughout the note. Let $K$ be an algebraically closed field. We will mainly work
over $K = \mathbb{C}$, but most of the results also work for $\mathbb{R}$. We denote $[n]$ as the set $\{1,\ldots,n\}$. For a $n\times d$ matrix $X=\pare{x_{i,j}}$ of indeterminates, we denote $R=K[X]$ for the polynomial ring on the variables $x_{ij}$. Given the subsets $A\subset[n]$ and $B\subset[d]$ with $\size{A}=\size{B}$, we denote  $[A|B]_{X}$ for the minor of $X$ with rows indexed by $A$ and columns indexed by $B$. %We will only work over $\CC$, but most of the results also work for $\mathbb{R}$.

\begin{notation}\label{initial notation}
 To any matroid $M$ of rank $n$ over the ground set $E$, we associate $X$, an $n\times \size{E}$ matrix of indeterminates, whose columns are indexed by the elements of $E$. For any subset $\{p_{1},\ldots,p_{k}\}\subset E$ with $k\leq n$, and vectors $v_{1},\ldots,v_{n-k}\in \CC^{n}$, 
we denote by $[p_{1},\ldots,p_{k},v_{1},\ldots,v_{n-k}]\in \CC[X]$ the determinant of the matrix formed by taking the columns of $X$ associated with $p_{1},\ldots,p_{k}$ 
and the vectors $v_{1},\ldots,v_{n-k}$ as additional columns. Observe that $[p_{1},\ldots,p_{k},v_{1},\ldots,v_{n-k}]$ is a polynomial in $\CC[X]$, and not a number.
\end{notation}

\subsection{Realization spaces of matroids and associated varieties}
We first review some results about matroids and their algebraic counterparts to fix our notation and keep the paper self-contained.  We refer to \cite{Oxley, piff1970vector} for an introduction to matroids, and to \cite{mnev1985manifolds, mnev1988universality, sturmfels1989matroid} for their realization spaces and associated varieties. We use the notation from \cite{bruns2003determinantal} for determinantal varieties and \cite{richter2011perspectives,lee2013mnev} for projective geometry, and recall some necessary results from \textup{\cite{clarke2021matroid, sidman2021geometric}}.

%A matroid is a combinatorial abstraction of linear dependence within a vector space. When presented with a finite set of vectors within a fixed vector space, the collection of linearly dependent sets of vectors defines a matroid. 
%Should this process be reversible, meaning that for a given matroid $M$, we can identify such a collection of vectors, we refer to these vectors as a realization of $M$.  %We denote  $\Gamma_M$ as the {\em space of realizations} of $M$ in $\CC^{\ranc(M)}$, containing all such collections of vectors.
\begin{definition}
 Let $M$ be a matroid of rank $n$ on the ground set $[d]$ and let $r\geq n$. A realization of a matroid $M$ in $\CC^{r}$ is a collection of vectors $Y=\{v_{1},\ldots,v_{d}\}\subset \CC^{r}$ such that
\[\{v_{i_{1}},\ldots,v_{i_{p}}\}\ \text{is linearly dependent} \Longleftrightarrow \{i_{1},\ldots,i_{p}\} \ \text{is a dependent set of $M$.}\]
The realization space of $M$ in $\CC^{r}$ is 
$\Gamma_{M,r}=\{Y\subset \CC^{r}: Y \ \text{is a realization of $M$}\}.$
Each element of $\Gamma_{M,r}$ is identified with an $r\times d$ matrix over $\CC$. In this work, we only consider the case $r=n$, and we simply denote the realization space $\Gamma_{M,n}\subset \CC^{nd}$ by $\Gamma_{M}$.
\end{definition}

\begin{definition}[Matroid variety]
    The {\em matroid variety} $V_M$ %of $M$
    is defined as the Zariski closure  of %its  realization space 
    $\Gamma_M$  in $\CC^{nd}$. We denote as $I_{M}=I\pare{V_{M}}$ the corresponding  {\em matroid ideal}.%ideal of the matroid variety. 
\end{definition}

\begin{remark}\label{relation between varieties}
 Given a rank $n$ matroid $M$ on %over the ground set 
$[d]$, both $\Gamma_{M}$ and $V_{M}$ are subsets of $\CC^{nd}$. Observe that this definition of the matroid variety coincides with those given in \cite{clarke2021matroid, Fatemeh3}. However, it differs from the one given in \cite{gelfand1987combinatorial, sidman2021geometric},  where the realization space of $M$ is defined as the set of all points in $\text{Gr}(n,d)$ that realize $M$, with the matroid variety being the Zariski closure of this realization space within the Grassmannian. Under this definition, we denote the corresponding ideal of the matroid variety by $I_{M}\rq$.
Any polynomial in $I_{M}\rq$ %living inside the homogeneous coordinate ring of the Grassmannian
gives rise to a polynomial living within $I_{M}$, by evaluating its Pl\"ucker coordinates on the associated Pl\"ucker minors; see~Example~\ref{gc3}. However, the converse is not necessarily true.
\end{remark}

One of the main problems that we study is to find a finite set of generators for the matroid ideal $I_M$. We now introduce the \textit{circuit ideals} and \textit{basis ideals} of matroids.

\begin{definition}[Circuit and basis ideals]\normalfont\label{cir}
Let $M$ be a matroid of rank $n$ on $[d]$. We denote $\mathcal{D}(M)$ as the set of %minimal 
dependencies of $M$, $\mathcal{C}(M)$ as the set of circuits of $M$, and $\mathcal{B}(M)$ as the set of all bases of $M$.
%Let $M$ be a matroid in the ground set $\corch{d}$. We denote $\mathcal{D}\pare{M}$ as the set of all dependencies of $M$,  $\mathcal{C}\pare{M}$ to refer to the circuits of $M$ and $\mathcal{B}\pare{M}$ as the set of all the basis of $M$. 
Consider the $n\times d$ matrix $X=\pare{x_{i,j}}$ of indeterminates. We define the \textit{circuit ideal} as:
$$ I_{\mathcal{C}(M)} = \{ [A|B]_X:\ B \in \mathcal{C}(M),\ A \subset [n],\ \text{and}\ |A| = |B| \}. $$
Note that the polynomials in $I_{\mathcal{C}(M)}$ vanish in any realization of $M$, implying $I_{\mathcal{C}(M)} \subset I_M$.
 %We define the circuit variety $V_{\mathcal{C}(M)}$ as $V(I_{\mathcal{C}(M)})$.
We say that $\gamma$, a collection of vectors of $\CC^{n}$ indexed by $[d]$, includes the dependencies of $M$ if it satisfies:
\[\set{i_{1},\ldots,i_{k}}\  \text{is a dependent set of $M$} \Longrightarrow \set{\gamma_{i_{1}},\ldots,\gamma_{i_{k}}}\ \text{is linearly dependent}. \] 
The {\em circuit variety} of $M$ is $V_{\mathcal{C}(M)}=V(I_{\mathcal{C}(M)})=\{\gamma:\text{$\gamma$ includes the dependencies of $M$}\}$.
Moreover, we define the \textit{basis ideal} of $M$ as:
\[
J_{M}=\sqrt{
\prod_{B\in \mathcal{B}\pare{M}}\ip{[A|B]_{X}:\ A\subset [n]}{\size{A}=\size{B}}.}
\]
Note that a collection of vectors belongs to $V(J_{M})$ if and only if it has a dependent subset not in $\mathcal{D}({M})$, where $\mathcal{D}(M)$ is the set of all dependent sets of $M$. Consequently, $\Gamma_{M} \cap V(J_{M}) = \emptyset$.
%This variety consists of collections of vectors in $\CC^{n}$, indexed by $[d]$, whose dependencies include those of $M$, i.e.~}
%\[V_{\mathcal{C}(M)}=\set{\{v_{1},\ldots,v_{d}\}\in \CC^{d\times n}:\rank(X_{J})<\size{J}\  \text{for each $J\in \mathcal{C}(M)$}}.\]

%Furthermore, observe that the elements of $V(I_{\mathcal{C}(M)})$ precisely consist of collections of vectors whose dependencies include those of $M$. 
\end{definition}
If the collection of dependent sets of matroid $M$ is a subset of the collection of dependent sets of matroid $N$, then we write $M\leq N$. We also recall that for $I,J\subset R=\CC[x_{1},\ldots,x_{n}]$ we denote 
$I:J^\infty=\{f\in R:\ fJ^i\subset I\ \text{for some $i$}\}$.

\medskip
We now recall the Proposition 3.9 from \cite{clarke2021matroid}.
\begin{proposition}\label{sat}%[\cite{clarke2021matroid}]\label{sat}
The matroid ideal can be obtained 
as $I_{M}=I_{\mathcal{C}\pare{M}}:J_{M}^\infty.$ 
\end{proposition}

Computing the saturation of ideals can be very challenging. Our aim is to explore alternative methods for computing $I_{M}$. In particular, we will focus on the family of paving matroids; 
see \cite{Oxley,welsh2010matroid}.

%\textcolor{red}{I changed the notation of ''regular hyperplanes of $\gamma$'' to ''$\gamma$-regular''}
\begin{definition}[Paving matroid]\normalfont \label{pav}
A rank $n$ matroid $M$ is called \textit{paving} if every circuit of $M$ has a size of either $n$ or $n+1$. In this case, we refer to it as an \textit{$n$-paving matroid}. 
Every paving matroid is entirely characterized by its circuits of size $n$. 
We fix the following notions for an $n$-paving matroid: 
\begin{itemize}
\item We define a \textit{dependent hyperplane} as a maximal subset of elements  in the ground set of $M$, with a size of at least $n$, where every subset of $n$ elements forms a circuit. For the special case where $n=3$, we simply call them \textit{lines}.  We denote the set of all dependent hyperplanes of $M$ by $\mathcal{L}_M$. %Observe that this definition of dependent hyperplanes differs from the usual definition in the literature, where dependent hyperplanes are typically defined as flats of rank equal to $n-1$.

\item We refer to elements  in the ground set of $M$ as {\em points} and we denote the set of points by $\mathcal{P}_M$. For a point $p \in \mathcal{P}_M$, we denote the set of dependent hyperplanes that contain $p$ as $\mathcal{L}_{p}$, and we define the {\em degree} of $p$ as the number of dependent hyperplanes in $\mathcal{L}_{p}$, denoted by $|\mathcal{L}_{p}|$.

\item 
Let $\gamma$ be a collection of vectors of $\CC^{n}$ indexed by $\mathcal{P}_M$. We denote by $\gamma_{p}$ the vector of $\gamma$ associated to $p\in \mathcal{P}_{M}$.
%We consider the vector $\gamma_{p}$  in $\mathbb{C}^{n}$ for any point $p$. We denote the collection of vectors $\gamma_{p}$ associated to the points of $\mathcal{P}_M$ by $\gamma$.
For a dependent hyperplane $l$, we denote the subspace $\text{span}(\gamma_{p}:p \in l)$ by $\gamma_{l}$. 

\item Let $\gamma\in V_{\mathcal{C}(M)}$. A dependent hyperplane $l \in \mathcal{L}_M$ is called \textit{$\gamma$-regular} %hyperplane} %of $\gamma\in V_{\mathcal{C}(M)}$ 
if $\text{dim}(\gamma_{l}) = n-1$. We denote by $R_{\gamma} \subset \mathcal{L}_M$ the set of  $\gamma$-regular hyperplanes. %of $\gamma$. 
If $R_{\gamma} = \mathcal{L}_M$, we say that $\gamma$ is \textit{regular}. 

\item For $\gamma \in V_{\mathcal{C}(M)}$, we denote by $\overline{R_{\gamma}}$ the partition of $R_{\gamma}$ in which two dependent hyperplanes $l_{1}, l_{2} \in R_{\gamma}$ belong to the same subset of the partition if and only if $\gamma_{l_{1}} = \gamma_{l_{2}}$. Note that this condition is equivalent to $\text{span}(\gamma_{p} : p \in l_1) = \text{span}(\gamma_{p} : p \in l_2)$, meaning that the vectors in $\gamma$ associated to $l_1$ and those associated to $l_2$ span the same hyperplane. This may happen even when $l_1 \neq l_2$.

\item A vector $q \in \mathbb{C}^{n}$ is called \textit{outside} $\gamma$ if it does not belong to any subspace $\gamma_{l}$ for any $l \in \mathcal{L}_M$.
\end{itemize}
For ease of notation, when it is clear, we drop the index $M$ and 
simply write $\mathcal{L}$ and $\mathcal{P}$. % instead of  $\mathcal{L}_M$ and $\mathcal{P}_M$. 
Note that $\gamma \in V_{\mathcal{C}(M)}$ if and only if $\text{dim}(\gamma_{l}) \leq n-1$ for every $l \in \mathcal{L}$, and if $\gamma \in \Gamma_{M}$, then $\gamma_{l}$ is a hyperplane in $\mathbb{C}^{n}$. %, and if $\gamma \in \Gamma_{M}$, then $\gamma_{l}$ is a dependent hyperplane in $\mathbb{C}^{n}$. 

\end{definition}

We now recall the general notion of submatroids. 

\begin{definition}[Submatroid]\normalfont \label{subm}
Let $M$ be a matroid over the ground set $[d]$, and let $S \subset [d]$. % be a subset. 
Then, there exists a matroid structure over the ground set $S$, where the dependent sets are those of $M$ %that are 
contained in $S$, and the rank function on $S$ is the rank function of $[d]$ restricted to $S$. Unless otherwise stated, we always assume that the subsets of $[d]$ have this structure, and we refer to them as \textit{submatroids} of $M$. %Note that, 
For an $n$-paving matroid, all of its submatroids of full rank are also $n$-paving matroids. 
\end{definition}

We have the following lemma for submatroids, which is useful for computing the matroid~ideals.

\begin{lemma}\label{lem sub}
Let $N\subset M$ be a submatroid  of full rank of $M$. Then $I_{N}\subset I_{M}$.
\end{lemma}
\begin{proof}
Let $n = \text{rank}(M) = \text{rank}(N)$, and let $[d]$ be the ground set of $M$. The realization space $\Gamma_{M}$ is defined as the set of all collections of vectors $\{\gamma_{1}, \ldots, \gamma_{d}\} \subset \mathbb{C}^{n}$ that realize $M$. The restriction of any such collection to the indices corresponding to $N$ yields a collection realizing $N$, which belongs to $\Gamma_{N}$ since $\text{rank}(N) = n$. Thus, the restriction of $\Gamma_{M}$ to $N$ is contained in $\Gamma_{N}$. If $p \in I_{N}$, then $p(\Gamma_{N}) = 0$. Since the restriction of $\Gamma_{M}$ to $N$ is a subset of $\Gamma_{N}$, it follows that $p(\Gamma_{M}) = 0$, and therefore $p \in I_{M}$.
\end{proof}

%The following lemma provides a characterization of paving matroids in terms of their dependent hyperplanes.
The following lemma from \cite{hartmanis1959lattice} gives a hypergraph characterization of paving matroids. See also \textup{\cite[Proposition~2.1.24]{Oxley}}.

\begin{lemma}\label{sub h}
Consider two integers $d, n$ such that $d\geq n+1$. Let $\mathcal{L}$ be a collection of subsets of $[d]$ such that $\size{l}\geq n$ for every $l\in \mathcal{L}$. Then the subsets of $\mathcal{L}$ determine the dependent hyperplanes of an $n$-paving matroid $M$ over the ground set $[d]$ if and only if $\size{l_{1}\cap l_{2}}\leq n-2$ for every $l_{1}\neq l_{2}\in \mathcal{L}$.
\end{lemma}

%\begin{proof}
%The result follows by Lemma $2$ and Theorem~$1$ from \cite{mederos2022method}.
%\end{proof}

%\textcolor{red}{I added the notation $M^{L}$ to avoid repeating ''submatroid of hyperplanes of $L$'' multiple times. I also swapped the order of Definitions~\ref{submatroid of hyperplanes} and~\ref{subm}.}

\begin{definition}[Submatroid of hyperplanes]\normalfont\label{submatroid of hyperplanes}
Let $M$ be an $n$-paving matroid, and let $L \subset \mathcal{L}_M$ be a subset containing at least two dependent hyperplanes. 
 By Lemma~\ref{sub h}, we can endow the set of points $\cup_{l\in L}l$ with a structure of an $n$-paving matroid, with $L$ as its set of dependent hyperplanes. We call this matroid the \textit{submatroid of hyperplanes} of $L$, and we denote it by $M^{L}$. Note that this matroid structure on $\mathcal{P}_{M^{L}}=\cup_{l\in L}l$ differs from the one given as a submatroid of $M$, in Definition~\ref{subm}.
%Consider $N$ as the subset of points belonging to some hyperplane in $L$. 
%Using Lemma~\ref{sub h}, we can endow $N$ with a structure of an $n$-paving matroid, with $L$ as its set of dependent hyperplanes. We denote this matroid as the \textit{submatroid of hyperplanes} of $L$.  Observe that this matroid structure on $N$ differs from the one given as a submatroid of $M$, of Definition~\ref{subm}.}
\end{definition}

\begin{definition}[Uniform matroid]\normalfont\label{uniform 3}
A uniform matroid $U_{n, d}$ on the ground set $[d]$ of rank $n$ is defined as follows: every subset $S$ of $[d]$ of size $|S| \leq n$ is independent, and every subset $S$ with $|S| > n$ is dependent. %The only matroid on $[d]$ with no circuits of size $<n+1$ is the uniform matroid. 
In particular, an $n$-paving matroid is \textit{uniform} if all its circuits have size $n+1$.
\end{definition}

%\begin{remark}\label{uniform}
If $M$ is $n$-paving, then each of its dependent hyperplanes is a uniform matroid of rank $n-1$.
%\end{remark}
A particular class of paving matroids consists of those of rank $3$, which are matroids of rank $3$ without double points or loops, also known as simple matroids. This class of matroids is also referred to as \textit{point-line configurations}; see e.g. \cite{clarke2021matroid}.

%\textcolor{red}{I changed the subscripts $i$ to $p_{i}$ throughtout the file.}

\begin{example}[Three concurrent lines]\label{three lines}
Consider %$3$-paving matroid
 the point-line configuration in Figure~\ref{fig:combined}~(Left). In this case, $\mathcal{P}=\{  {p_{1}},\ldots,  {p_7}\}$, and lines $\mathcal{L}=\{   \{{p_1},  {p_2}, {p_7}\} , \{{p_3},  {p_4},   {p_7}\}  ,   \{{p_5},  {p_6}, {p_7}\} \}$. A line $l$ is  $\gamma$-regular for a collection of vectors $\gamma\in V_{\mathcal{C}(M)}$ if and only if $\dim(\gamma_{l})=2$. For instance, $  \{{p_1},{p_2},{p_7}\}  $ is  $\gamma$-regular %for $\gamma$ 
when $\rank \set{\gamma_{  {p_7}},\gamma_{  {p_1}},\gamma_{  {p_2}}}=2$.
\end{example}

\begin{figure}[h]
    \centering
    \begin{subfigure}[b]{0.3\textwidth}
        \centering
        \begin{tikzpicture}[x=0.75pt,y=0.75pt,yscale=-1,xscale=1]

\tikzset{every picture/.style={line width=0.75pt}} %set default line width to 0.75pt        

%Straight Lines [id:da06101063982657062] 
\draw    (81.69,116.61) -- (191.16,174.3) ;
%Straight Lines [id:da7678393526214486] 
\draw    (77,131.88) -- (224,131.88) ;
%Straight Lines [id:da3419905968139849] 
\draw    (80.13,150.55) -- (191.16,79.27) ;
%Shape: Ellipse [id:dp43020942362366354] 
\draw [fill={rgb, 255:red, 173; green, 216; blue, 230}, fill opacity=1]
(107.34,131.2) .. controls (107.34,133.07) and (108.63,134.58) .. (110.23,134.58) .. controls (111.83,134.58) and (113.12,133.07) .. (113.12,131.2) .. controls (113.12,129.33) and (111.83,127.81) .. (110.23,127.81) .. controls (108.63,127.81) and (107.34,129.33) .. (107.34,131.2) -- cycle ;
%Shape: Ellipse [id:dp7909486902309077] 
\draw [fill={rgb, 255:red, 173; green, 216; blue, 230}, fill opacity=1]
(142.62,149.93) .. controls (142.62,151.8) and (143.91,153.32) .. (145.51,153.32) .. controls (147.11,153.32) and (148.4,151.8) .. (148.4,149.93) .. controls (148.4,148.06) and (147.11,146.55) .. (145.51,146.55) .. controls (143.91,146.55) and (142.62,148.06) .. (142.62,149.93) -- cycle ;
%Shape: Ellipse [id:dp6735689153424713] 
\draw [fill={rgb, 255:red, 173; green, 216; blue, 230}, fill opacity=1]
(173.7,166.79) .. controls (173.7,168.66) and (174.99,170.18) .. (176.59,170.18) .. controls (178.19,170.18) and (179.48,168.66) .. (179.48,166.79) .. controls (179.48,164.92) and (178.19,163.41) .. (176.59,163.41) .. controls (174.99,163.41) and (173.7,164.92) .. (173.7,166.79) -- cycle ;
%Shape: Ellipse [id:dp9808673721534766] 
\draw [fill={rgb, 255:red, 173; green, 216; blue, 230}, fill opacity=1] (201.42,131.2) .. controls (201.42,133.07) and (202.71,134.58) .. (204.31,134.58) .. controls (205.91,134.58) and (207.2,133.07) .. (207.2,131.2) .. controls (207.2,129.33) and (205.91,127.81) .. (204.31,127.81) .. controls (202.71,127.81) and (201.42,129.33) .. (201.42,131.2) -- cycle ;
%Shape: Ellipse [id:dp8560955724937005] 
\draw [fill={rgb, 255:red, 173; green, 216; blue, 230}, fill opacity=1] (167.82,131.2) .. controls (167.82,133.07) and (169.11,134.58) .. (170.71,134.58) .. controls (172.31,134.58) and (173.6,133.07) .. (173.6,131.2) .. controls (173.6,129.33) and (172.31,127.81) .. (170.71,127.81) .. controls (169.11,127.81) and (167.82,129.33) .. (167.82,131.2) -- cycle ;
%Shape: Ellipse [id:dp8522323991300782] 
\draw [fill={rgb, 255:red, 173; green, 216; blue, 230}, fill opacity=1] (177.06,87.18) .. controls (177.06,89.05) and (178.35,90.56) .. (179.95,90.56) .. controls (181.55,90.56) and (182.84,89.05) .. (182.84,87.18) .. controls (182.84,85.31) and (181.55,83.79) .. (179.95,83.79) .. controls (178.35,83.79) and (177.06,85.31) .. (177.06,87.18) -- cycle ;
%Shape: Ellipse [id:dp38553240573402947] 
\draw [fill={rgb, 255:red, 173; green, 216; blue, 230}, fill opacity=1] (145.14,105.91) .. controls (145.14,107.78) and (146.43,109.3) .. (148.03,109.3) .. controls (149.63,109.3) and (150.92,107.78) .. (150.92,105.91) .. controls (150.92,104.04) and (149.63,102.52) .. (148.03,102.52) .. controls (146.43,102.52) and (145.14,104.04) .. (145.14,105.91) -- cycle ;

% Text Node
\draw (166.25,71.57) node [anchor=north west][inner sep=0.75pt]   [align=left] {{\scriptsize $p_1$}};
% Text Node
\draw (206.28,115.51) node [anchor=north west][inner sep=0.75pt]   [align=left] {{\scriptsize $p_{3}$}};
% Text Node
\draw (138.91,153.35) node [anchor=north west][inner sep=0.75pt]   [align=left] {{\scriptsize $p_6$}};
% Text Node
\draw (170.08,173.62) node [anchor=north west][inner sep=0.75pt]   [align=left] {{\scriptsize $p_5$}};
% Text Node
\draw (173.44,116.45) node [anchor=north west][inner sep=0.75pt]   [align=left] {{\scriptsize $p_4$}};
% Text Node
\draw (106.8,111.36) node [anchor=north west][inner sep=0.75pt]   [align=left] {{\scriptsize $p_7$}};
% Text Node
\draw (136.66,90.59) node [anchor=north west][inner sep=0.75pt]   [align=left] {{\scriptsize $p_2$}};
% Text Node
%\draw (96.73,190.03) node [anchor=north west][inner sep=0.75pt]   [align=left] {{\footnotesize Pencil of three lines}};

   \end{tikzpicture}
        %\caption{Quadrilateral set}
        \label{fig:quadrilateral 2}
    \end{subfigure}
    \hfill
    \begin{subfigure}[b]{0.3\textwidth}
        \centering

\tikzset{every picture/.style={line width=0.75pt}} %set default line width to 0.75pt        

\begin{tikzpicture}[x=0.75pt,y=0.75pt,yscale=-1,xscale=1]

%Straight Lines [id:da9613256122302601] 
\draw [line width=0.75]    (247.65,123.92+60) -- (201.96,207.53+60) ;
%Straight Lines [id:da3259117734461473] 
\draw [line width=0.75]    (247.65,123.92+60) -- (291.68,207.53+60) ;
%Straight Lines [id:da2624677482334885] 
\draw [line width=0.75]    (219.23,174.98+60) -- (291.68,207.53+60) ;
%Straight Lines [id:da26466266747028633] 
\draw [line width=0.75]    (274.41,174.98+60) -- (201.96,207.53+60) ;
%Shape: Ellipse [id:dp10345993152561672] 
\draw [fill={rgb, 255:red, 173; green, 216; blue, 230}, fill opacity=1] (244.87,123.92+60) .. controls (244.87,125.68+60) and (246.12,127.11+60) .. (247.65,127.11+60) .. controls (249.19,127.11+60) and (250.44,125.68+60) .. (250.44,123.92+60) .. controls (250.44,122.15+60) and (249.19,120.73+60) .. (247.65,120.73+60) .. controls (246.12,120.73+60) and (244.87,122.15+60) .. (244.87,123.92+60) -- cycle ;
%Shape: Ellipse [id:dp21945184574290333] 
\draw [fill={rgb, 255:red, 173; green, 216; blue, 230}, fill opacity=1] (244.31,187.1+60) .. controls (244.31,188.87+60) and (245.56,190.29+60) .. (247.1,190.29+60) .. controls (248.64,190.29+60) and (249.88,188.87+60) .. (249.88,187.1+60) .. controls (249.88,185.34+60) and (248.64,183.91+60) .. (247.1,183.91+60) .. controls (245.56,183.91+60) and (244.31,185.34+60) .. (244.31,187.1+60) -- cycle ;
%Shape: Ellipse [id:dp033802625652762264] 
\draw [fill={rgb, 255:red, 173; green, 216; blue, 230}, fill opacity=1] (216.45,174.98+60) .. controls (216.45,176.74+60) and (217.69,178.17+60) .. (219.23,178.17+60) .. controls (220.77,178.17+60) and (222.02,176.74+60) .. (222.02,174.98+60) .. controls (222.02,173.21+60) and (220.77,171.79+60) .. (219.23,171.79+60) .. controls (217.69,171.79+60) and (216.45,173.21+60) .. (216.45,174.98+60) -- cycle ;
%Shape: Ellipse [id:dp395767307047196] 
\draw [fill={rgb, 255:red, 173; green, 216; blue, 230}, fill opacity=1] (271.62,174.98+60) .. controls (271.62,176.74+60) and (272.87,178.17+60) .. (274.41,178.17+60) .. controls (275.94,178.17+60) and (277.19,176.74+60) .. (277.19,174.98+60) .. controls (277.19,173.21+60) and (275.94,171.79+60) .. (274.41,171.79+60) .. controls (272.87,171.79+60) and (271.62,173.21+60) .. (271.62,174.98+60) -- cycle ;
%Shape: Ellipse [id:dp5069700831588411] 
\draw [fill={rgb, 255:red, 173; green, 216; blue, 230}, fill opacity=1] (288.9,207.53+60) .. controls (288.9,209.29+60) and (290.14,210.72+60) .. (291.68,210.72+60) .. controls (293.22,210.72+60) and (294.47,209.29+60) .. (294.47,207.53+60) .. controls (294.47,205.77+60) and (293.22,204.34+60) .. (291.68,204.34+60) .. controls (290.14,204.34+60) and (288.9,205.77+60) .. (288.9,207.53+60) -- cycle ;
%Shape: Ellipse [id:dp3796249357557302] 
\draw [fill={rgb, 255:red, 173; green, 216; blue, 230}, fill opacity=1] (199.17,207.53+60) .. controls (199.17,209.29+60) and (200.42,210.72+60) .. (201.96,210.72+60) .. controls (203.49,210.72+60) and (204.74,209.29+60) .. (204.74,207.53+60) .. controls (204.74,205.77+60) and (203.49,204.34+60) .. (201.96,204.34+60) .. controls (200.42,204.34+60) and (199.17,205.77+60) .. (199.17,207.53+60) -- cycle ;

% Text Node
\draw (243.29,107.07+60) node [anchor=north west][inner sep=0.75pt]   [align=left] {{\scriptsize $p_1$}};
% Text Node
\draw (200.77,165.67+60) node [anchor=north west][inner sep=0.75pt]   [align=left] {{\scriptsize $p_2$}};
% Text Node
\draw (184.94,205.77+60) node [anchor=north west][inner sep=0.75pt]   [align=left] {{\scriptsize $p_3$}};
% Text Node
\draw (241.88,169.35+60) node [anchor=north west][inner sep=0.75pt]   [align=left] {{\scriptsize $p_4$}};
% Text Node
\draw (277.98,164.77+60) node [anchor=north west][inner sep=0.75pt]   [align=left] {{\scriptsize $p_5$}};
% Text Node
\draw (292.84,206.21+60) node [anchor=north west][inner sep=0.75pt]   [align=left] {{\scriptsize $p_6$}};
% Text Node
%\draw (203.1,213.83) node [anchor=north west][inner sep=0.75pt]   [align=left] {{\footnotesize Quadrilateral set}};

\end{tikzpicture}
       % \caption{Pascal configuration}
        \label{fig:pascal}
    \end{subfigure}
    \hfill
    \begin{subfigure}[b]{0.3\textwidth}
        \centering

\tikzset{every picture/.style={line width=0.75pt}} %set default line width to 0.75pt        

\begin{tikzpicture}[x=0.75pt,y=0.75pt,yscale=-1,xscale=1]

%Shape: Ellipse [id:dp6893614752841091] 
\draw  [dash pattern={on 4.5pt off 4.5pt}] (171,191.4) .. controls (171,167.83) and (209.5,148.72) .. (257,148.72) .. controls (304.5,148.72) and (343,167.83) .. (343,191.4) .. controls (343,214.98) and (304.5,234.09) .. (257,234.09) .. controls (209.5,234.09) and (171,214.98) .. (171,191.4) -- cycle ;
%Straight Lines [id:da17320322934700716] 
\draw    (193.78,161.68) -- (247.97,234.09) ;
%Straight Lines [id:da7615768712472442] 
\draw    (190.63,182.26) -- (313.94,208.17) ;
%Straight Lines [id:da09145367358212564] 
\draw    (302.95,154.82) -- (247.97,234.09) ;
%Straight Lines [id:da5228730695592458] 
\draw    (250.32,148.72) -- (185.92,215.03) ;
%Straight Lines [id:da9999142489139483] 
\draw    (193.78,161.68) -- (285.67,232.56) ;
%Straight Lines [id:da2959348461531355] 
\draw    (302.95,154.82) -- (185.92,215.03) ;
%Straight Lines [id:da5997219702791894] 
\draw    (250.32,148.72) -- (285.67,232.56) ;
%Shape: Ellipse [id:dp527194292233689] 
\draw [fill={rgb, 255:red, 173; green, 216; blue, 230}, fill opacity=1] (191.42,161.68) .. controls (191.42,162.94) and (192.47,163.96) .. (193.78,163.96) .. controls (195.08,163.96) and (196.13,162.94) .. (196.13,161.68) .. controls (196.13,160.41) and (195.08,159.39) .. (193.78,159.39) .. controls (192.47,159.39) and (191.42,160.41) .. (191.42,161.68) -- cycle ;
%Shape: Ellipse [id:dp714577908207096] 
\draw [fill={rgb, 255:red, 173; green, 216; blue, 230}, fill opacity=1] (229.9,191.4) .. controls (229.9,192.67) and (230.96,193.69) .. (232.26,193.69) .. controls (233.56,193.69) and (234.62,192.67) .. (234.62,191.4) .. controls (234.62,190.14) and (233.56,189.12) .. (232.26,189.12) .. controls (230.96,189.12) and (229.9,190.14) .. (229.9,191.4) -- cycle ;
%Shape: Ellipse [id:dp04607294248496796] 
\draw [fill={rgb, 255:red, 173; green, 216; blue, 230}, fill opacity=1] (283.31,232.56) .. controls (283.31,233.82) and (284.37,234.85) .. (285.67,234.85) .. controls (286.97,234.85) and (288.02,233.82) .. (288.02,232.56) .. controls (288.02,231.3) and (286.97,230.27) .. (285.67,230.27) .. controls (284.37,230.27) and (283.31,231.3) .. (283.31,232.56) -- cycle ;
%Shape: Ellipse [id:dp4727254048800096] 
\draw [fill={rgb, 255:red, 173; green, 216; blue, 230}, fill opacity=1] (245.61,234.09) .. controls (245.61,235.35) and (246.67,236.37) .. (247.97,236.37) .. controls (249.27,236.37) and (250.32,235.35) .. (250.32,234.09) .. controls (250.32,232.82) and (249.27,231.8) .. (247.97,231.8) .. controls (246.67,231.8) and (245.61,232.82) .. (245.61,234.09) -- cycle ;
%Shape: Ellipse [id:dp4090579290974339] 
\draw [fill={rgb, 255:red, 173; green, 216; blue, 230}, fill opacity=1] (210.27,186.83) .. controls (210.27,188.09) and (211.32,189.12) .. (212.63,189.12) .. controls (213.93,189.12) and (214.98,188.09) .. (214.98,186.83) .. controls (214.98,185.57) and (213.93,184.54) .. (212.63,184.54) .. controls (211.32,184.54) and (210.27,185.57) .. (210.27,186.83) -- cycle ;
%Shape: Ellipse [id:dp3007543571571194] 
\draw [fill={rgb, 255:red, 173; green, 216; blue, 230}, fill opacity=1] (183.57,215.03) .. controls (183.57,216.29) and (184.62,217.32) .. (185.92,217.32) .. controls (187.22,217.32) and (188.28,216.29) .. (188.28,215.03) .. controls (188.28,213.77) and (187.22,212.74) .. (185.92,212.74) .. controls (184.62,212.74) and (183.57,213.77) .. (183.57,215.03) -- cycle ;
%Shape: Ellipse [id:dp177612614991548] 
\draw [fill={rgb, 255:red, 173; green, 216; blue, 230}, fill opacity=1] (300.59,154.82) .. controls (300.59,156.08) and (301.64,157.1) .. (302.95,157.1) .. controls (304.25,157.1) and (305.3,156.08) .. (305.3,154.82) .. controls (305.3,153.55) and (304.25,152.53) .. (302.95,152.53) .. controls (301.64,152.53) and (300.59,153.55) .. (300.59,154.82) -- cycle ;
%Shape: Ellipse [id:dp6002333029432028] 
\draw [fill={rgb, 255:red, 173; green, 216; blue, 230}, fill opacity=1] (247.97,148.72) .. controls (247.97,149.98) and (249.02,151.01) .. (250.32,151.01) .. controls (251.63,151.01) and (252.68,149.98) .. (252.68,148.72) .. controls (252.68,147.46) and (251.63,146.43) .. (250.32,146.43) .. controls (249.02,146.43) and (247.97,147.46) .. (247.97,148.72) -- cycle ;
%Shape: Ellipse [id:dp5616508853630788] 
\draw [fill={rgb, 255:red, 173; green, 216; blue, 230}, fill opacity=1] (269.17,199.02) .. controls (269.17,200.29) and (270.23,201.31) .. (271.53,201.31) .. controls (272.83,201.31) and (273.89,200.29) .. (273.89,199.02) .. controls (273.89,197.76) and (272.83,196.74) .. (271.53,196.74) .. controls (270.23,196.74) and (269.17,197.76) .. (269.17,199.02) -- cycle ;

% Text Node
\draw (183.96,150.51) node [anchor=north west][inner sep=0.75pt]   [align=left] {{\scriptsize $p_1$}};
% Text Node
\draw (277.42,188.62) node [anchor=north west][inner sep=0.75pt]   [align=left] {{\scriptsize $p_8$}};
% Text Node
\draw (226.37,193.19) node [anchor=north west][inner sep=0.75pt]   [align=left] {{\scriptsize $p_9$}};
% Text Node
\draw (192.88,185.57) node [anchor=north west][inner sep=0.75pt]   [align=left] {{\scriptsize $p_7$}};
% Text Node
\draw (291.55,228.25) node [anchor=north west][inner sep=0.75pt]   [align=left] {{\scriptsize $p_6$}};
% Text Node
\draw (250.71,134.5) node [anchor=north west][inner sep=0.75pt]   [align=left] {{\scriptsize $p_5$}};
% Text Node
\draw (168.1,210.72) node [anchor=north west][inner sep=0.75pt]   [align=left] {{\scriptsize $p_4$}};
% Text Node
\draw (305.69,142.89) node [anchor=north west][inner sep=0.75pt]   [align=left] {{\scriptsize $p_3$}};
% Text Node
\draw (248.36,234.35) node [anchor=north west][inner sep=0.75pt]   [align=left] {{\scriptsize $p_2$}};
% Text Node
%\draw (204.18,247.5) node [anchor=north west][inner sep=0.75pt]   [align=left] {{\footnotesize Pascal configuration}};

\end{tikzpicture}
      %  \caption{Pencil of three lines}
        \label{fig:pencil}
    \end{subfigure}
    \caption{(Left) Three concurrent lines; (Center) Quadrilateral set; (Right) Pascal configuration.}
    \label{fig:combined}
\end{figure}

\subsection{Grassmann-Cayley algebra}
In this subsection, we adopt the notation from \textup{\cite[Section~3.3]{computationalgorithms}}. The Grassmann-Cayley algebra is the usual exterior algebra $\bigwedge (\mathbb{C}^{d})$ equipped with two operations: the \textit{join} and the \textit{meet}, which are denoted by $\vee$ and $\wedge$, respectively.
The join operation corresponds to the exterior product, though the symbol $\vee$ is used in the Grassmann-Cayley algebra instead of the usual $\wedge$. If $\{e_{1},\ldots,e_{d}\}$ denotes the standard basis of $\CC^{d}$, then
$v_1 \vee \cdots \vee v_{d} = [v_{1} \cdots v_{d}] \, e_{1}\vee \cdots \vee e_{d}$,
where $[v_{1} \cdots v_{d}]$ denotes the determinant of the matrix with columns $v_{1}, \ldots, v_{d}$.
We refer to a join $v_{1}\vee \cdots \vee v_{k}$ of vectors $v_{1},\ldots, v_{k}\in \CC^{d}$ as an \emph{extensor} of length $k$, which we also denote by $v_{1}\cdots v_{k}$. The Grassmann-Cayley algebra is generated by such extensors, with the join operation acting on two of them as follows:
\[
(v_{1}\vee \cdots \vee v_{k})\vee (w_{1}\vee \cdots \vee w_{j}) = v_{1}\vee \cdots \vee v_{k}\vee w_{1}\vee \cdots \vee w_{j}.
\]
%We denote $v_{1}\cdots v_{k}$, \textcolor{blue}{or equivalently by $v_{1}\vee \cdots \vee v_{k}$}, the join (or {\em extensor}) of length $k$ associated with the vectors $v_{1},\ldots ,v_{k}$. 
The meet operation  %denoted by $\wedge$, 
is defined on two extensors $v = v_{1}\cdots v_{k}$ and $w = w_{1}\cdots w_{j}$ of lengths $k$ and $j$,~respectively,~with~$j+k\geq d$,~as:
$$v\wedge w=\sum_{\sigma\in \mathcal{S}(k,j,d)}
\corch{v_{\sigma(1)}\cdots v_{\sigma(d-j)}w_{1}\cdots w_{j}}\cdot v_{\sigma(d-j+1)}\cdots v_{\sigma(k)}$$
where $\mathcal{S}(k,j,d)$ is the set of all permutations of ${1,\ldots,k}$ that satisfy $\sigma(1)<\cdots <\sigma(d-j)$ and $\sigma(d-j+1)<\cdots <\sigma(d)$. If $j+k<d$, the meet of the two extensors is defined as $0$.
Both the join and the meet operations are extended linearly to any elements of the exterior algebra.
Every extensor $v = v_{1}\cdots v_{k}$ is associated with the subspace $\overline{v} = \text{span}\{v_{1},\ldots,v_{k}\}$, and the following properties hold.

\begin{lemma}
Consider the extensors $v=v_{1}\cdots v_{k}$ and $w=w_{1}\cdots w_{j}$ with $j+k\geq d$. Then:
\begin{itemize}
\item The vectors $v_{1},\ldots,v_{k}$ are linearly dependent if and only if $v=0$.
\item Any extensor $v$ is uniquely determined by $\overline{v}$ up to a scalar multiple.
\item The meet of two extensors is an extensor.
\item We have that $v\wedge w\neq 0$ if and only if $\overline{v}+\overline{w}=\CC^{d}$. In this case, we have $\overline{v}\cap\overline{w}=\overline{v\wedge w}.$
\end{itemize}
\end{lemma}

The proof of the last two items is not immediate. We refer the reader to~\textup{\cite[Theorem~3.3.2]{computationalgorithms}} for a complete proof.

\begin{example}\label{gc3}
Consider the %$3$-paving matroid
point-line configuration in Figure~\ref{fig:combined}~(Left). Since,  in any realization, the lines $\{p_{1}p_{2},p_{3}p_{4},p_{5}p_{6}\}$ are concurrent, the intersection point of the lines $p_{1}p_{2}$ and $p_{3}p_{4}$ lies on the line $p_{5}p_6$. Thus, the condition for %the lines 
$\{p_1 p_2,p_3 p_4,p_5 p_6\}$ to be concurrent is the vanishing of the polynomial:
$$(p_3 p_4) \wedge (p_1 p_2) \vee (p_5 p_6) = (\corch{p_3 p_1 p_2} p_4 - \corch{p_4 p_1 p_2} p_3) \vee p_5 p_6 = \corch{p_1 p_2 p_3} \corch{p_4 p_5 p_6} - \corch{p_1 p_2 p_4} \corch{p_3 p_5 p_6}.$$
Consequently, we can infer that this polynomial lies within the associated matroid ideal.
\end{example}

%%%%%%%%%%%%%%%%%%%%

\section{Computing lifting polynomials %in %the matroid ideal $I_{M}$
}\label{constructing}
In this section, we consider Question~\ref{question_intro}, exploring methods for constructing polynomials within the matroid ideals $I_{M}$, where $M$ represents an $n$-paving matroid. We introduce a method rooted in projective geometry, based on the projection and lifting of points within the matroid with respect to a hyperplane. Applying this technique, we provide a %finite %generating 
set of defining equations for the matroid %ideal
variety of paving matroids, in Theorem~\ref{theo lift}. To begin, we provide an overview of the underlying framework.

\medskip

Throughout this section, we use the following notation.  Let $M$ be an $n$-paving matroid with the set of points $\mathcal{P}$ and the set of dependent hyperplanes $\mathcal{L}$. We denote a vector in $\mathbb{C}^{n}$ as $\mathbf{x} = (x_{1}, %x_{2},
\ldots,
x_{n})$.

\begin{definition}
 Consider a collection of vectors $\gamma = \{\gamma_{p} : p \in \mathcal{P}\}\subset \CC^{n}$ and let $q\in \CC^{n}$. We say that  $\widetilde{\gamma}=\{\widetilde{\gamma}_{p}:p\in \mathcal{P}\}\subset \CC^{n}$ is a {\em lifting} of $\gamma$ from the vector $q$, if for each $p\in \mathcal{P}$ there exists $\lambda_{p}\in \CC$ such that $\widetilde{\gamma}_{p}=\gamma_{p}+\lambda_{p}q$.
\end{definition}

%\textcolor{red}{I divided Definition~\ref{def_lift} into Definition~\ref{def_lift} and Notation~\ref{notation liftings}, because the liftability condition was only defined for the case $q=(0,\ldots,0,1)$}

\begin{definition}[Liftability condition]\label{def_lift}
   Let $M$ be an $n$-paving matroid with the set of points $\mathcal{P}$. Consider a collection of vectors $\gamma = \{\gamma_{p} : p \in \mathcal{P}\}\subset \CC^{n}$ of rank $n-1$. Assume these vectors span the hyperplane $H$ of $\CC^{n}$, and let $q\not \in H$. 
    Then % collection of vectors 
    $\gamma$ is called {\em liftable} for $M$ from the vector $q$, if 
     the vectors $\gamma_p$ can be lifted  from the point $q$ to a collection of vectors $\widetilde{\gamma} \in V_{\mathcal{C}(M)}$ in a non-degenerate manner. By \emph{non-degenerate} we mean that the lifted vectors $\widetilde{\gamma}$ do not all lie in a common hyperplane. We refer to this condition as the {\em liftability condition}~of~$\gamma$.
\end{definition}

\begin{example}
 Consider the matroid depicted in Figure~\ref{fig:combined} (Center), known as the quadrilateral set $QS$. In Figure~\ref{new figure 2}, we present an example of a liftable collection of vectors for %the quadrilateral set
$QS$, represented in the projective plane $\mathbb{P}^{2}$. More precisely, the collection of points $\gamma=\{\gamma_{p_{1}},\ldots,\gamma_{p_6}\}\subset \mathbb{P}^{2}$ can be lifted in a non-degenerate manner from the point $q$ to a collection of points within $V_{\mathcal{C}(QS)}$.%the circuit variety of the quadrilateral set.}
\end{example}

\begin{figure}[H]
    \centering
    \includegraphics[width=0.5\textwidth, trim=0 0 0 0, clip]{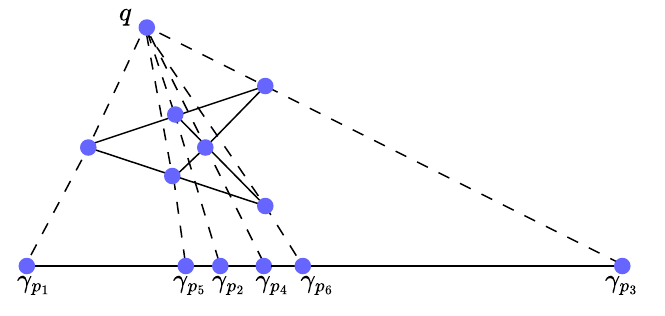}
    \caption{Example of a liftable collection of vectors.}
    \label{new figure 2}
\end{figure}

The setup below is used throughout this subsection.

% 
%The setup below is used throughout this subsection.
%We also present the following setting.}

%\textcolor{red}{I changed the notation of $\gamma_{p}=(y_{p},1,0)$ with $y_{p}\in \CC^{n-2}$ to $(y_{p},0)$ with $y_{p}\in \CC^{n-1}$, since we were not using that the $(n-1)^{\text{th}}$ entry was $1$.}

\begin{setup}\label{notation liftings}
 Let $M$ be an $n$-paving matroid with the set of points $\mathcal{P}$. Consider a collection of vectors $\gamma = \{\gamma_{p} : p \in \mathcal{P}\}\subset \mathbb{C}^n$ of rank $n-1$. Assume these vectors span the hyperplane $x_{n} = 0$. Denote these points as $\gamma_{p} = (y_{p}, %1,
0)$, where $y_{p} \in \mathbb{C}^{ n-1}$ for $p \in \mathcal{P}$, and consider the canonical vector $q = (0, %0,
\ldots, 0, 1)\in \mathbb{C}^n$.
%Following Notation~\ref{notation liftings}, 
We write $\corch{ p_{1}, p_{2}, \ldots, p_{n-1}}_{\gamma}$ to represent the determinant of the matrix with column vectors %$\left(y_{p_{1}}, 1\right), \left(y_{p_{2}}, 1\right), \ldots, \left(y_{p_{n-1}}, 1\right)$,
 $y_{p_{1}},\ldots,y_{p_{n-1}}$,
for any $p_{1},\ldots,p_{n-1}\in \mathcal{P}$. Next, we define the matrix $\mathcal{M}^{\gamma}(M)$, where the columns are indexed by  $\mathcal{P}$ %the points of $M$ 
and the rows are indexed by the circuits of size $n$ of $M$, as follows. For each $n$-circuit $c=\{c_{1},c_{2},\ldots,c_{n}\}$, the $c_{i}^{\rm th}$ coordinate of the corresponding row is given by: 
\[ (-1)^{i-1}\left[c_{1},c_{2},\ldots,c_{i-1},\hat{c_{i}},c_{i+1},\ldots, c_{n}\right]_{\gamma}, \]
while the other entries of the row are set to $0$.
Note that the non-zero entries of the row corresponding to $c$ are precisely the $(n-1)$-minors of the matrix with column vectors %$\left(y_{c_{1}}, 1\right), \left(y_{c_{2}}, 1\right), \ldots, \left(y_{c_{n}}, 1\right)$. 
 $y_{c_{1}},\ldots,y_{c_{n}}$.
\end{setup}

\begin{notation}\label{size matroid}
 For a matroid $M$, let $\lvert M \rvert$ denote the size of its ground set.
\end{notation}

The following lemma illustrates the main idea behind constructing lifting polynomials of $M$.

\begin{lemma}\label{lift}
Following Setup~\ref{notation liftings}, $\gamma$ is liftable if all $(\size{M}-n+1)$-minors of %the matrix 
$\mathcal{M}^{\gamma}(M)$ vanish.
\end{lemma}

\begin{proof}
Let $\widetilde{\gamma}=\{\widetilde{\gamma}_{p}= (y_{p},z_{p}): p\in \mathcal{P}\}$ be a lifting of $\gamma$ from the vector $q$. For $\widetilde{\gamma}$ to be on $V_{\mathcal{C}(M)}$, the vectors $\set{ \pare{y_{c_{i}},z_{c_{i}}}},1\leq i \leq n$ must be dependent for any circuit 
$c=\{c_{1},\ldots,c_{n}\}\in \mathcal{C}(M)$, which is equivalent to the determinant of the matrix with them as column vectors being $0$. If we expand this determinant in the last row (the row with the numbers $z_{c_{i}}$), we obtain that 
\[\sum_{i=1}^{n}z_{c_{i}}(-1)^{i-1}\corch{c_{1},c_{2},\ldots,c_{i-1},\hat{c_{i}},c_{i+1},\ldots, c_{n}}_{\gamma}=0.
\]
%has to be zero. 
However, the coefficients multiplying $z_{c_{i}}$ in this expression are precisely the non-zero entries of the row corresponding to $c$ in the matrix $\mathcal{M}^{\gamma}(M)$. Consequently, we can deduce that the collection of vectors $\widetilde{\gamma}$ belongs to $V_{\mathcal{C}(M)}$ if and only if the vector $z=\{z_{p}:p\in \mathcal{P}\}$ is in the kernel of the matrix $\mathcal{M}^{\gamma}(M)$.

\medskip
 
{\bf Claim.} $\dim(\ker \mathcal{M}^{\gamma}(M)) \geq n-1$.
%We now prove that  

\medskip
Since the vectors of $\gamma$ have rank $n-1$, we can select points $p_{1},\ldots,p_{n-1}\in \mathcal{P}$ such that $\gamma_{p_{1}},\ldots,\gamma_{p_{n-1}}$ are linearly independent.  To prove the claim, we will see that for any collection $w_{1},\ldots,w_{n-1}$ of numbers in $\CC$, there exists $\widetilde{\gamma}=\{\widetilde{\gamma}_{p}= (y_{p},z_{p}): p\in \mathcal{P}\}$, a lifting of $\gamma$, such that $z_{p_{i}}=w_{i}$ for every $i\in [n-1]$.
%For any choice of the numbers $\{z_{p_{1}},z_{p_{2}},\ldots,z_{p_{n-1}}\}$,  
Consider $H$ as the hyperplane generated by the vectors 
$\{ \pare{y_{p_{i}},w_{i}}:i\in \corch{n-1}\}$, and note that $q\notin H$. Then, for each $p\in \mathcal{P}$, if we choose the number $z_{p}$ such that $ (y_{p},z_{p})$ belongs uniquely to $H$ (we project the point $\gamma_{p}$ onto the hyperplane $H$ from the vector $q$), it follows that the vectors $\{ (y_{p},z_{p}):p\in \mathcal{P}\}$ all belong to $H$. 
Thus, this collection of vectors specifically belongs~to~$V_{\mathcal{C}(M)}$.  Also note that $z_{p_{i}}=w_{i}$ for every $i\in [n-1]$.
Consequently, the vector $z$ belongs to the kernel of $\mathcal{M}^{\gamma}(M)$. 
However,  since $w_{1},\ldots,w_{n-1}$ are arbitrary, we conclude that the dimension of the kernel is at least $n-1$.

\medskip
From this argument, we easily conclude that the dimension of the kernel is $n-1$ if and only if every vector in the kernel can be represented in this form. This equivalence implies that all liftings of $\gamma$ to a collection of vectors in $V_{\mathcal{C}(M)}$ are degenerate.
Hence, the vectors of $\gamma$ can be lifted from $q$ in a non-degenerate manner if and only if $\dim(\ker \mathcal{M}^{\gamma}(M))\geq n$. This condition holds if and only if the rank of the matrix is at most $\lvert M\rvert-n$, which in turn occurs if and only if all $(\lvert M\rvert-n+1)$-minors of the matrix vanish. This completes the proof.
\end{proof}

\begin{definition}[Liftable matroid]\normalfont\label{liftable}
Let $M$ be an $n$-paving matroid. We say that $M$ is $\textit{liftable}$ if every collection of vectors $\gamma=\{\gamma_{p}:p\in \mathcal{P}\}\subset \mathbb{C}^{n}$ of rank $n-1$ within a hyperplane $H$ can be lifted in a non-degenerate manner to a collection of vectors in $V_{\mathcal{C}(M)}$ from any vector $q\notin H$.
\end{definition}

\begin{remark}\label{different definition}
 It is important to note that our definition of a  liftable matroid differs from the definition of liftable point-line configuration in \cite{Fatemeh3}. In \cite{Fatemeh3}, a point-line configuration is considered liftable if every rank-two collection of vectors $\gamma=\{\gamma_{p}:p\in \mathcal{P}\}\subset \mathbb{C}^{3}$ can be lifted to a collection of vectors in $\Gamma_{M}$. However, for our purposes, the definition provided in Definition~\ref{liftable} is more suitable.
\end{remark}

\begin{example}
 The quadrilateral set, in Figure~\ref{fig:combined} (Center), was shown to be non-liftable in \cite{Fatemeh3}. On the other hand, the non-liftability of the matroid in Figure~\ref{fig:combined} (Right) follows easily from Example~\ref{conic}.
\end{example}

For the following definition, recall Notation~\ref{initial notation}.

\begin{definition}[Liftability matrix]\label{lif mat}
Let $M$ be an $n$-paving matroid, and let $q\in \mathbb{C}^{n}$. We define the \textit{liftability matrix} $\mathcal{M}_{q}(M)$ of $M$ and $q$, as the matrix with columns indexed by the points $\mathcal{P}$ and rows indexed by the $n$-circuits of $M$. The entries of this matrix are defined as follows: for each $n$-circuit $c=\{c_{1},c_{2},\ldots,c_{n}\}$, the $c_{i}^{\text{th}}$ coordinate of the corresponding row is given by the polynomial 
\[ 
(-1)^{i-1}[c_{1},c_{2},\ldots,c_{i-1},\hat{c_{i}},c_{i+1},\ldots, c_{n},q],
\]
while the other entries of the row are set to $0$.
Note that the entries of the matrix are polynomials, not numbers. For a collection of vectors $\gamma=\{\gamma_{p}:p\in \mathcal{P}\}\subset \CC^{n}$, we denote $\mathcal{M}_{q}^{\gamma}(M)$ as the matrix obtained by evaluating the entries of $\mathcal{M}_{q}(M)$ at the vectors of $\gamma$,   i.e.~substituting the vector of $n$ indeterminates associated with each $p\in \mathcal{P}$ by the vector $\gamma_{p}$.
\end{definition}

\begin{example}
 Consider the quadrilateral set, depicted in Figure~\ref{fig:combined} (Center). The quadrilateral set $QS$ has points $\mathcal{P}=\set{p_{1},p_2,p_3,p_4,p_5,p_6}$ and its circuits of size $3$ are \[\set{\{p_1,p_2,p_3\},\{p_1,p_5,p_6\},\{p_2,p_4,p_6\},\{p_3,p_4,p_5\}}.\] Therefore, for any $q\in \CC^{3}$, its liftability matrix is the following matrix:
\begin{equation}\label{matrix quad}\mathcal{M}_{q}(QS)=\begin{pmatrix}
\corch{p_2p_3q} & -\corch{p_1p_3q} & \corch{p_1p_2q} &0 &0 &0 \\
\corch{p_5p_6q} & 0 &0 &0 &-\corch{p_1p_6q} &\corch{p_1p_5q} \\ 
 0  &\corch{p_4p_6q} &0 &-\corch{p_2p_6q} & 0&\corch{p_2p_4q}    \\
 0& 0 & \corch{p_4p_5q} &-\corch{p_3p_5q}& \corch{p_3p_4q}& 0
\end{pmatrix}.\end{equation}
\end{example}

The~following proposition is crucial for constructing polynomials in $I_M$ from its liftability~matrices. %We use the following notation.

\begin{proposition} \label{lp}
%For every $q\in \mathbb{C}^{n}$, 
The $(\lvert M\rvert-n+1)$-minors of %the liftability matrix 
$\mathcal{M}_{q}(M)$ are polynomials within $I_{M}$, for every $q\in \mathbb{C}^{n}$.
\end{proposition}

\begin{proof}
Given that $\Gamma_{M}$ is preserved by projective transformations and the vanishing of the minors is also a projective invariant (due to their homogeneity), we can apply a projective transformation to assume, without loss of generality, that $q=(0,\ldots,0,1)$. Consider $\delta=\{\delta_{p}:p\in \mathcal{P}\}$, which is any realization of $M$. For each $p\in \mathcal{P}$, we express the vector $\delta_{p}$ as  $\delta_{p}=(y_{p},z_{p})$, where $y_{p}\in\mathbb{C}^{n-1}$. Then, we define the collection of vectors
$
\gamma:=\{\gamma_{p}= (y_{p},0):p\in \mathcal{P}\}
$
within the hyperplane $x_{n}=0$.
The vectors of $\gamma$ are projections of the vectors of $\delta$ onto that hyperplane. Consequently, it is evident that the points $\{\gamma_{p}:p\in \mathcal{P}\}$ can be lifted in a non-degenerate manner from the point $q$ to a collection of vectors in $V_{\mathcal{C}(M)}$ (specifically, that collection of vectors is $\delta$).

Applying Lemma~\ref{lift}, we establish that the $(\lvert M\rvert-n+1)$-minors of the matrix $\mathcal{M}_{q}^{\gamma}(M)$ vanish. However,  considering that ${[\gamma_{p_{1}},\ldots,\gamma_{p_{n-1}},q]=[\delta_{p_{1}},\ldots,\delta_{p_{n-1}},q]}$ for any $p_{1},\ldots,p_{n-1}\in \mathcal{P}$, we recognize that these minors are identical to those of the matrix $\mathcal{M}_{q}^{\delta}(M)$, which also vanish. Hence, we have shown that the $(\lvert M\rvert-n+1)$-minors of $\mathcal{M}_{q}(M)$ vanish at every element of $\Gamma_{M}$,  since $\delta \in \Gamma_{M}$ is arbitrary,
thereby confirming that they are polynomials within $I_{M}$.
\end{proof}

\begin{example}
 Applying Proposition~\ref{lp} for the quadrilateral set, we obtain that the $4\times 4$ minors of the matrix $\mathcal{M}_{q}(QS)$ of Equation~\eqref{matrix quad} are polynomials within the ideal $I_{QS}$, % (the quadrilateral set ideal), 
for any $q\in \CC^{3}$.
\end{example}

\begin{lemma}\label{lifting 3}
Let $\gamma=\set{\gamma_{p}:p\in \mathcal{P}}$ be a collection of vectors of rank $n-1$ within a hyperplane $H$. Then there exists a  non-degenerate lifting of the vectors of $\gamma$ from a vector $q\notin H$ to a collection of vectors in  $V_{\mathcal{C}(M)}$ if and only if the $(\size{M}-n+1)$-minors of the matrix $\mathcal{M}_{q}^{\gamma}\pare{M}$ vanish. Moreover, if this lifting exists, then we can make the lifted vectors arbitrarily close to the original ones.
\end{lemma}

\begin{proof}
Let $T$ be a projective transformation that sends $q$ to $\pare{0,\ldots,0,1}$ and the hyperplane $H$ to the hyperplane $x_{n}=0$. We know that the vectors of $\gamma$ are liftable in a non-degenerate way from $q$ if and only if the vectors $\set{T(\gamma_{p}):p\in \mathcal{P}}$ can be lifted in a non-degenerate way from  $T(q)=\pare{0,\ldots,0,1}$. Using Lemma~\ref{lift}, we know that this happens if and only if the $\lvert M\rvert-n+1$ minors of the matrix $\mathcal{M}_{T(q)}(M)$ vanish on the vectors $\{T(\gamma_{p}):p\in \mathcal{P}\}$. However, this occurs if and only if the $(\lvert M\rvert-n+1)$-minors of the matrix $\mathcal{M}_{q}(M)$ vanish on the vectors of $\gamma$, since we are using $T(q)=(0,\ldots,0,1)$ and the vanishing of the minors is invariant under projective transformations. This proves the first statement.

Now, suppose that the non-degenerate lifting exists. Since the vectors of $\gamma$ can be lifted in a non-degenerate way from $q$, we have that the vectors $\{T(\gamma_{p}):p\in \mathcal{P}\}$ can also be lifted in a non-degenerate way from the point $T(q)=(0,\ldots,0,1)$. Consequently, we can make this lifting arbitrarily close to the original vectors by choosing the vector $z$ in the kernel with small coordinates. Therefore, the same holds true for the vectors of $\gamma$ and $q$, which completes the proof of the second statement.
\end{proof}

The following lemma provides a sufficient condition for a paving matroid to be liftable.

\begin{lemma}\label{lifts}
Let $M$ be an $n$-paving matroid. Suppose that the number of circuits of size $n$ of $M$ is $k$, and $\size{M}\geq k+n$. Then, $M$ is liftable. 
\end{lemma}
\begin{proof}
We need to prove that any collection of vectors $\gamma=\{\gamma_{p}:p\in \mathcal{P}\}$ of rank $n-1$ within a hyperplane $H$ %can be lifted in a non-degenerate way to a collection of vectors in $V_{\mathcal{C}(M)}$ 
is liftable from any $q\notin H$. Since $\text{rank} \mathcal{M}_{q}^{\gamma}(M)\leq k\leq \lvert M\rvert-n$, the dimension of the kernel of $\mathcal{M}_{q}^{\gamma}(M)$ is at least $n$. This implies that the $(\lvert M\rvert-n+1)$-minors of this matrix vanish. Then the claim follows by applying Lemma~\ref{lifting 3}.
\end{proof}

\begin{figure}[H]
    \centering
    \includegraphics[width=0.5\textwidth, trim=0 0 0 0, clip]{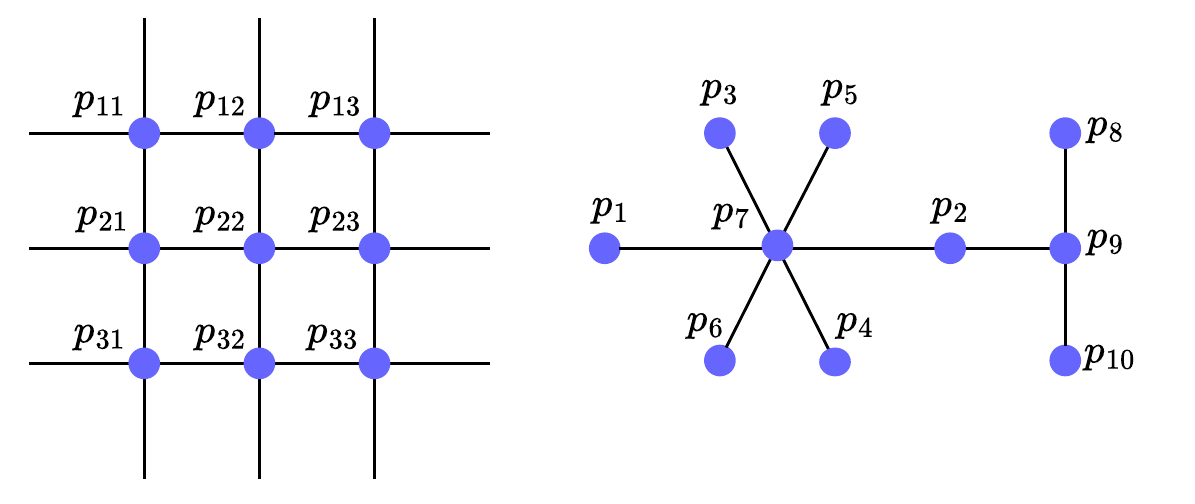}
    \caption{(Left) $3\times 3$ grid; (Right) Forest configuration.}
    \label{new figure}
\end{figure}

\begin{example}\label{example 3 grid}
 Consider the $3 \times 3$ grid depicted in Figure~\ref{new figure} (Left). This point-line configuration has 6 circuits of size 3, corresponding to the 6 lines of the grid, and 9 elements in its ground set. Therefore, by applying Lemma~\ref{lifts}, we conclude that this matroid is liftable. On the other hand, the configuration illustrated in Figure~\ref{new figure} (Right) is a forest-type configuration, as defined in \textup{\cite[Definition~5.1]{clarke2021matroid}}. These types of configurations were shown to be liftable in \textup{\cite[Lemma~3.2]{Fatemeh3}}. In particular, the point-line configurations in Figures~\ref{fig:combined} (Left) and Figure~\ref{new figure} (Right) are liftable, since both are forest configurations.
\end{example}

%\textcolor{red}{I removed the definition of the ideal $L_{q}(M)$ from Definition~\ref{cl 2} since it is not necessary.}

\begin{definition}\label{cl 2}
 We define $I_{M}^{\text{lift}}$ %$L_{q}(M)$ 
as the ideal generated by all the $(\lvert N\rvert-n+1)$-minors of the liftability matrices $\mathcal{M}_{q}(N)$, where $N$ varies over all the submatroids of $M$ of full rank and $q$ varies over $\mathbb{C}^{n}$. %Additionally, we define $I_{M}^{\text{lift}}$ as the ideal generated by all the ideals $L_{q}(M)$, where $q$ varies over $\mathbb{C}^{n}$.
\end{definition}

%\begin{remark}\label{obs}
%Note that for a fixed $q\in \mathbb{C}^{n}$, we can explicitly provide a finite set of generators for the ideal $L_{q}(M)$. This is because there are only a finite number of submatroids of $M$, and for each of these submatroids, we consider a finite number of minors of its liftability matrix. 
\noindent Note that providing a finite generating set for the ideal $I_{M}^{\text{lift}}$ is not straightforward, as $q$ varies over~$\mathbb{C}^{n}$.
%\end{remark}

\medskip
Now, we present the main result of this section, which generalizes Proposition~\ref{lp}.

\begin{theorem}\label{sub}
Let~$M$ be an~$n$-paving matroid and $N$ be a submatroid of~$M$ with full rank. Then, for any $q\in \mathbb{C}^{n}$, the $(|N|-n+1)$-minors of $\mathcal{M}_{q}(N)$ are polynomials within $I_{M}$. Moreover,~$I_{M}^{\text{lift}}\subseteq I_M$.
\end{theorem}

\begin{proof}
Since $N$ has full rank, we know that it is also an $n$-paving matroid. By Proposition~\ref{lp}, we deduce that the $(\lvert N\rvert-n+1)$-minors of $\mathcal{M}_{q}(N)$ are within $I_{N}$. Then, the result holds by Lemma~\ref{lem sub}.
\end{proof}

\section{%Equations defining matroid varieties of 
Paving matroids
with no points of degree greater than $2$
}\label{sec 4}
%\textcolor{red}{I am not sure if we should make this assumption here. It could create confusion if this paragraph is not read. If someone reads the statements of the results of this section before reading this paragraph, the statements would be false.}
Throughout this section, we assume that $M$ is a paving matroid with no points of degree greater than~$2$. We will avoid repeating this assumption in the statements of lemmas. We apply the construction outlined in Section~\ref{constructing} to provide a complete generating set,  up to radical, for the associated matroid ideal in Theorem~\ref{theo lift}. Additionally, we present a decomposition of the circuit variety %$V_{\mathcal{C}(M)}$ 
in~Proposition~\ref{dec}.

\subsection{Infinitesimal motion}
\begin{definition}%[{Infinitesimal motion}]
 Consider $\gamma$ a collection of vectors in $\CC^{n}$. %$\gamma\in\mathbb{C}^{n}$.
Let $X$ denote a specific property. We say that it is possible to apply an {\em infinitesimal motion} to $\gamma$ to yield a new collection of vectors satisfying property $X$, if for every $\epsilon>0$, it is feasible to select, for each vector $v\in \gamma$, a vector $\widetilde{v}$ such that $\lVert v-\widetilde{v} \rVert <\epsilon$, ensuring that the new collection of vectors adheres to property $X$. Generally, when we use the term {\em infinitesimal}, %or {\em infinitesimally}, 
we are referring to a motion that can be chosen to be arbitrarily close to $0$.
When discussing an {\em infinitesimal motion} of a $k$-dimensional subspace $S\subset \mathbb{C}^{n}$, we are considering $S$ as an extensor $v_{1}%ldots
 \cdots v_{k}$, and applying an infinitesimal motion to $\{v_{i}:i\in [k]\}$ to obtain another $k$-dimensional subspace corresponding to the extensor with the new vectors.
\end{definition}

%Applying the above lemmas, we can establish the following proposition. 
%\textcolor{red}{I removed the notion of $L_{\gamma}$ from this definition since we were only using the size of this set.}

\begin{definition}[Lifting number]\normalfont \label{lif num}
 Recall the notion of $R_{\gamma}$ from Definition~\ref{pav}. Let $M$ be an $n$-paving matroid, and let $\gamma$ be a collection of vectors in $V_{\mathcal{C}(M)}$. We define the \textit{lifting number} of $\gamma$ as:
\[
c(M,\gamma) = \size{\{(l_{1},l_{2}) \in R_{\gamma} \times R_{\gamma} \ \text{such that} \ l_{1} \neq l_{2} \ \text{and} \ \gamma_{l_{1}} = \gamma_{l_{2}}\}}.
\]
\end{definition}
%\vspace{-1cm}
With the notation established, our first main result in this section is the following proposition. % I am not sure if it is ok that we state this proposition before the definition of $M_{\gamma}$, which is first defined in Definition~\ref{lif num}}
\begin{proposition}\label{coincide}
Let $M$ be an $n$-paving matroid, % with no points of degree greater than two, 
and let $\gamma$ be a collection of vectors in $V_{\mathcal{C}(M)}$ such that $c(M,\gamma)=0$. Then, it is possible to perform an infinitesimal motion on the vectors of $\gamma$ %such that they belong to 
 and obtain a collection of vectors in $\Gamma_{M}$.
\end{proposition}

 To prove Proposition~\ref{coincide}, we require several lemmas, which we will first establish.

%\textcolor{red}{I removed an unnecessary paragraph from this lemma .}
\begin{lemma}\label{move}
Let $H_{1}\neq H_{2}$ be two distinct hyperplanes of $\CC^{n}$ and let %$p$ be a vector in $H_{1}\cap H_{2}$. 
 $p\in H_{1}\cap H_{2}$. Assume that $\widetilde{H_{1}}$ and $\widetilde{H_{2}}$ are two hyperplanes obtained by an infinitesimal motion of $H_{1}$ and $H_{2}$, respectively. Then, it is possible to select a vector $\widetilde{p}\in \widetilde{H_{1}}\cap \widetilde{H_{2}}$ that is an infinitesimal motion of $p$.
\end{lemma}

\begin{proof}
Let $\{v_{1},\ldots,v_{n-2}\}$ be a basis of $H_{1}\cap H_{2}$. Then, $p$ can be written as $p = \textstyle \sum_{i=1}^{n-2}\alpha_{i}v_{i}$ for some $\alpha_{i}$'s. Let $w_{1}$ and $w_{2}$ be two vectors that complete a basis with $\{v_{1},\ldots,v_{n-2}\}$, and define the subspaces $S_{j} = \text{span}\{v_{j},w_{1},w_{2}\}$ for $j \in [n-2]$.
For $i=1,2$, we fix bases $\{h_{i_{1}},\ldots,h_{i_{n-1}}\}$ of $H_{i}$ and $\{\widetilde{h}_{i_{1}},\ldots,\widetilde{h}_{i_{n-1}}\}$ of $\widetilde{H_{i}}$, with $\widetilde{h}_{i_{j}}$ being an infinitesimal motion of $h_{i_{j}}$ for $i=1,2$ and $j\in [n-1]$.
For every $j\in [n-2]$, we have that $H_{1}\cap H_{2} \cap S_{j}=\{v_{j}\}$, and as a consequence we have that:
$$v_{j}=h_{1_{1}}\cdots h_{1_{n-1}}\wedge h_{2_{1}}\cdots h_{2_{n-1}} \wedge v_{j}w_{1}w_{2}.$$
%and as $\set{v_{1},\ldots,v_{n-2}}$ is a basis of $H_{1}\cap H_{2}$ we have that $v_{1}\vee\cdots\vee v_{n-2}=v_{1}\cdots v_{n-2}\neq 0.$
%$$\bigvee_{j=1}^{n-2} h_{1_{1}}\cdots h_{1_{n-2}}\wedge h_{1_{1}}\cdots h_{1_{n-2}} \wedge v_{j}w_{1}w_{2}=v_{1}\cdots v_{n-2}\neq 0.$$
As the vectors $\widetilde{h}_{i_{j}}$ are an infinitesimal motion of the vectors $h_{i_{j}}$, we have that the vectors 
$$\widetilde{v}_{j}=\widetilde{h}_{1_{1}}\cdots \widetilde{h}_{i_{n-1}} \wedge \widetilde{h}_{1_{1}}\cdots \widetilde{h}_{i_{n-1}} \wedge v_{j}w_{1}w_{2}$$
are inside $\widetilde{H_{1}}\cap \widetilde{H_{2}}$ and are an infinitesimal motion of the vectors $v_{j}$. %We also have that 
%$$\widetilde{v}_{1}\cdots \widetilde{v}_{n-2}=\bigvee_{j=1}^{n-2} \widetilde{h}_{1_{1}}\cdots \widetilde{h}_{i_{n-1}} \wedge \widetilde{h}_{1_{1}}\cdots \widetilde{h}_{i_{n-1}} \wedge v_{j}w_{1}w_{2}\neq 0.$$
%This implies that the vectors $\{\widetilde{v}_{1},\ldots ,\widetilde{v}_{n-2}\}$ form a basis for $\widetilde{H_{1}}\cap \widetilde{H_{2}}$. 
Thus, the vector $\widetilde{p}=\textstyle\sum_{i=1}^{n-2}\alpha_{i}\widetilde{v}_{i}$ is inside $\widetilde{H_{1}}\cap \widetilde{H_{2}}$ and represents an infinitesimal motion of $p$.
\end{proof}

 Note that in Lemma~\ref{move}, the infinitesimal motion of $p$ is expressed in terms of the motions of the hyperplanes $H_{1}$ and $H_{2}$.

\begin{remark}
Note that the condition $H_{1} \neq H_{2}$ in Lemma~\ref{move} is necessary. For example, suppose that $H_{1} = H_{2}$ represents the same line $l$ in the projective plane (which corresponds to a hyperplane of $\CC^{3}$), and let $p$ be a point in $l = H_{1} \cap H_{2}$. Take $q$ to be any other point in $l$, and let us apply an infinitesimal %motion
 rotation of $H_{1}$ from the point $q$, while keeping $H_{2} = l$ fixed. In this case, the new intersection is $q$, indicating that it is not possible to select the point $\widetilde{p}$ as an infinitesimal motion of $p$.
\end{remark}

The objective now is to prove Proposition~\ref{coincide}, a necessary step in proving the main theorem of this section, Theorem~\ref{theo lift}. To prove this proposition, we first need to establish Lemmas~\ref{sets},~\ref{ñ},~\ref{sets 2}, and~\ref{lemm 2}. We will begin by introducing a definition.

\begin{definition}[{Choice property}]\normalfont
We say that the sets $X_{1},\ldots,X_{k}$ have the \textit{choice property} if it is possible to select elements $x_{i}\in X_{i}$ for each $i\in [k]$, such that $x_{i}\neq x_{j}$ for $i\neq j$. Using Hall's Marriage Theorem,~\cite[Theorem~1]{hall1948distinct}, we know that 
the sets have the \textit{choice property} if and only if for every $S\subset[k]$, we have 
$\lvert \textstyle \cup_{s\in S} X_{s} \rvert \geq \size{S}.$

We say that the subspaces $V_{1},\ldots,V_{k}\subset \mathbb{C}^{n}$ have the \textit{choice property} if it is possible to select vectors $v_{i}\in V_{i}$ for each $i\in [k]$, such that the vectors $\{v_{1},\ldots,v_{k}\}$ are linearly independent. Applying Hall's Theorem, % \cite[Theorem~1]{hall1948distinct}, % \cite[Theorem~2.1]{pym1970submodular}, 
the subspaces have the \textit{choice property} if and only if 
for every $S\subset[k]$, we have that: 
\[\dim(\textstyle \sum_{s\in S}V_{s})\geq \size{S}.\]
\end{definition}

\begin{remark}\label{choo 2}
Let ${v_{1},\ldots,v_{n}}$ be a basis of $\mathbb{C}^{n}$. Let $X_{1},\ldots,X_{k}$ be subsets of $[n]$, and consider the subspaces $V_{i}=\text{span}\{v_{x}:x\in X_{i}\}$. Suppose that the sets $X_{1},\ldots,X_{k}$ have the choice property. Then, for each $i\in [k]$, we can select elements $x_{i}\in X_{i}$ such that $x_{i}\neq x_{j}$, for $i\neq j$. Consequently, the vectors $v_{x_{i}}\in V_{i}$ are linearly independent. Therefore, the subspaces $V_{1},\ldots,V_{k}$ also have the choice property.
\end{remark}

\begin{lemma}\normalfont\label{sets}
Let $X, X_{1},\ldots,X_{n-1}$ be subsets of size $n-1$ of $[n]$, with not all of the sets $X_{i}$ being equal. Then, the sets $X\cap X_{i}$ have the choice property.
\end{lemma}
\begin{proof}
We denote by $C_{i}$ the set $X\cap X_{i}$. Applying Hall's Marriage Theorem \cite[Theorem~1]{hall1948distinct}, 
we need to prove that $\lvert \textstyle \cup_{s\in S}C_{s} \rvert \geq \lvert S \rvert$ for every $S\subset [n-1]$. If $\lvert S \rvert \leq n-2$, it is immediate because $\lvert C_{i} \rvert \geq n-2$ for every $i\in [n-1]$, so we only have to prove that $\lvert \textstyle \cup_{i=1}^{n-1} C_{i} \rvert \geq n-1$. We will prove this by showing that every element of $X$ belongs to some $C_{i}$, and since $\lvert X \rvert = n-1$, this would finish the proof. Let $x\in X$, and suppose that $x\notin C_{i}$ for every $i\in [n-1]$. Then $x\notin X_{i}$, which implies that $X_{i}=[n]\backslash {x}$ for every $i\in [n-1]$, but this contradicts our hypothesis, which completes the proof. 
\end{proof}

\begin{lemma}\label{ñ}
Let $H, H_{1},\ldots,H_{k}$ be hyperplanes in $\mathbb{C}^{n}$ %with $H \neq H_{i}$ for $i\in [k]$, 
and let ${c_{1},\ldots,c_{k}}$ be numbers such that $0<c_{i}<n-1$ for each $i\in [k]$ and $\textstyle\sum_{i=1}^{k} c_{i}=n-1$. Then it is possible to apply an infinitesimal motion to the hyperplanes $H, H_{1},\ldots,H_{k}$ to obtain distinct hyperplanes $\widetilde{H},\widetilde{H}_{1},\ldots,\widetilde{H}_{k}$ such that for each $i\in [k]$ %it is possible to select a set 
there exists a set $C_{i}$ of $c_{i}$ vectors from $\widetilde{H}\cap \widetilde{H}_{i}$ where the $n-1$ vectors in $\cup_{i=1}^{k} C_{i}$ are linearly independent.
\end{lemma}

\begin{proof}
First, we define the {\em modified choice property} for hyperplanes as the property described for $\widetilde{H},\widetilde{H}_{1},\ldots,\widetilde{H}_{k}$ at the end of the statement.
We begin by assuming the existence of distinct hyperplanes $L, L_{1},\ldots,L_{k}$ that possess the modified choice property. For each $i\in [k]$, let $E_{i}=\{v_{i_{1}},\ldots,v_{i_{c_{i}}}\}$ be the $c_{i}$ vectors of $L\cap L_{i}$ such that the $n-1$ vectors in $\cup_{i=1}^{k} E_{i}$ are linearly independent. For each $i\in [k]$, we select $D_{i}$ to be a subspace of dimension $2$ such that $L\cap L_{i}+D_{i}=\mathbb{C}^{n}$. For each $j\in [c_{i}]$, we define $F_{i_{j}}$ as the subspace $D_{i}+\text{span}\{v_{i_{j}}\}$, and we choose $f_{i_{j}}$ as an extensor of size $3$ corresponding to it. Now, we consider the following polynomial function:
\begin{equation}\label{eq1}
  \begin{gathered}
p:\underbrace{(\CC^{n})^{n-1}\times \cdots \times  (\CC^{n})^{n-1}}_{k+1 \ \text{times}}\longrightarrow \bigwedge \CC^{n}\\
(x,x_{1},\ldots,x_{k})\mapsto \bigvee_{i=1}^{k}\ \bigvee_{j=1}^{c_{i}}(x\wedge x_{i}\wedge f_{i_{j}}) 
 \end{gathered}
\end{equation}
where we are thinking of the elements $x$ and $x_{i}$ as extensors of size $n-1$ on the expression of the right-hand side. Now take $l,l_{1},\ldots,l_{k}$ extensors of size $n-1$ corresponding to the hyperplanes $L,L_{1},\ldots ,L_{k}$, respectively. As $L\cap L_{i}\cap F_{i_{j}}=\{v_{i_{j}}\}$ we have that 
$l\wedge l_{i}\wedge f_{i_{j}}=v_{i_{j}}$
for every $i\in [k]$ and $j\in [c_{i}]$. Since the vectors of $\textstyle \cup_{i=1}^{k} E_{i}$ are linearly independent, %as a consequence, 
we have that: 
$$p(l,l_{1},\ldots,l_{k})=\bigvee_{i=1}^{k}\ \bigvee_{j=1}^{c_{i}}v_{i_{j}}\neq 0.$$
%since the vectors of $\textstyle \cup_{i=1}^{k} E_{i}$ are linearly independent. 
This proves that the polynomial map $p$ is not identically zero. Therefore, $D(p)=\{x:p(x)\neq 0\}$ is a non-empty Zariski open set in $(\mathbb{C}^{n})^{(n-1)(k+1)}$ and, as a consequence, is dense.

Now take $h,h_{1},\ldots,h_{k}$ extensors of size $n-1$ corresponding to the hyperplanes $H,H_{1},\ldots,H_{k}$, respectively. Since $D(p)$ is dense, we know that there exist extensors $\widetilde{h},\widetilde{h}_{1},\ldots,\widetilde{h}_{k}$ of size $n-1$ that represent an infinitesimal motion of $h,h_{1},\ldots,h_{k}$ and satisfy $p(\widetilde{h},\widetilde{h}_{1},\ldots,\widetilde{h}_{k})\neq 0$. Let $\widetilde{H},\widetilde{H}_{1},\ldots,\widetilde{H}_{k}$ be the hyperplanes that correspond to ${\widetilde{h},\widetilde{h}_{1},\ldots,\widetilde{h}_{k}}$, respectively. For every $i\in [k]$  we know that the $c_{i}$ vectors $h_{i_{j}}=\widetilde{h}\wedge \widetilde{h}_{i}\wedge f_{i_{j}}$, for $j\in [c_{i}]$,  belong to $\widetilde{H}\cap{\widetilde{H}_{i}}$. Since $$\bigvee_{i=1}^{k}\ \bigvee_{j=1}^{c_{i}}h_{i_{j}}=p(\widetilde{h},\widetilde{h}_{1},\ldots,\widetilde{h}_{k})\neq 0,$$ we deduce that the vectors $h_{i_{j}}$,  with $i\in [k]$ and $j\in [c_{i}]$, are linearly independent. Therefore, the hyperplanes $\widetilde{H},\widetilde{H}_{1},\ldots,\widetilde{H}_{k}$ satisfy the modified choice property, as desired.

We still need to show the existence of hyperplanes $L,L_{1},\ldots,L_{k}$ that possess the modified choice property. Observe that it is equivalent to  the subspaces
\begin{equation}\label{eq 3}
\underbrace{L\cap L_{1}, \ldots  ,L\cap L_{1}}_{c_{1} \ \text{times}},\ldots \ldots,\underbrace{L\cap L_{k}, \ldots  ,L\cap L_{k}}_{c_{k} \ \text{times}}
\end{equation}
having the \textit{choice property}. To achieve this, we choose $L=\is{E\backslash \set{e_{n}}}$ and $L_{i}=\is{E\backslash \set{e_{i}}}$ for $i\in [k]$, where $E=\set{e_{1},\ldots ,e_{n}}$ is the canonical basis of $\CC^{n}$. Note that the sets 
$$E\backslash \set{e_{n}},\  \underbrace{ E\backslash \set{e_{1}}, \ldots  , E\backslash \set{e_{1}}}_{c_{1} \ \text{times}},\ldots \ldots\ldots \ldots, \underbrace{ E\backslash \set{e_{k}}, \ldots  , E\backslash \set{e_{k}}}_{c_{k} \ \text{times}}$$
are subsets of size $n-1$ of $[n]$ and satisfy the hypothesis of Lemma~\ref{sets}, so %we have that 
the sets 
$$\underbrace{E\backslash \set{e_{n}}\cap E\backslash \set{e_{1}}, \ldots  ,E\backslash \set{e_{n}}\cap E\backslash \set{e_{1}}}_{c_{1} \ \text{times}},\ldots \ldots\ldots \ldots, \underbrace{E\backslash \set{e_{n}}\cap E\backslash \set{e_{k}}, \ldots  ,E\backslash \set{e_{n}}\cap E\backslash \set{e_{k}}}_{c_{k} \ \text{times}}$$
have the choice property, and by Remark~\ref{choo 2}, the subspaces in~\eqref{eq 3} also have the choice property. %, as desired.
\end{proof}

\begin{lemma}\normalfont\label{sets 2}
Let $k\leq n$ and $Z=\set{X_{1},\ldots,X_{k},Y_{1},\ldots ,Y_{k}}$ be a multiset of $2k$ subsets of size $n-1$ of $[n]$. Suppose that there are no $X,Y,J\subset [n]$ of size $n-1$ with $X\neq Y$ such that $\set{X_{j},Y_{j}}=\set{X,Y}$ for every $j\in J$ and that there is no $W\subset [n]$ of size $n-1$ that has multiplicity at least $n$ in $Z$. Then the sets $\set{X_{i}\cap Y_{i}:i\in [k]}$ have the choice property.
\end{lemma}

\begin{proof}
Similarly to Lemma~\ref{sets}, we define $C_{i}=X_{i}\cap Y_{i}$, and by Hall's Marriage Theorem \cite[Theorem~1]{hall1948distinct}, we reduce the lemma to proving that $\lvert \textstyle \cup_{s\in S} C_{s} \rvert \geq \lvert S \rvert$ for every $S\subset [k]$. If $\lvert S \rvert \leq n-2$, this is immediate because $\lvert C_{i} \rvert \geq n-2$ for every $i\in [k]$. Now, suppose $\lvert S \rvert = n-1$. To prove this case, assume that $\lvert \textstyle \cup_{s\in S} C_{s} \rvert = n-2$, and we will arrive at a contradiction. Since every $C_{j}$ has size $n-2$, there must exist elements $x\neq y \in [n]$ such that $C_{s}=[n]\backslash \{x,y\}$ for every $s\in S$. This implies that $\{X_{s},Y_{s}\}= \{[n]\backslash \{x\},[n]\backslash \{y\}\}$ for every $s\in S$, but this contradicts the hypothesis of the lemma. Now we only have to look at the case $\size{S}=k=n$ and we must prove that every $x\in [n]$ belongs to some $C_{i}$. Suppose the contrary, that $x\notin C_{i}$ for every $i\in [n]$. Since the intersection of $X_{i}$ and $Y_{i}$ does not contain $x$, we have $[n]\backslash \{x\}\in \{X_{i},Y_{i}\}$ for every $i\in [n]$. Consequently, $[n]\backslash \{x\}$ has a multiplicity of at least $n$ in $Z$, which contradicts the hypothesis of the lemma. This proves the case $\lvert S \rvert = n$.
\end{proof}

\begin{lemma}\label{lemm 2}
Consider integers $d \leq n$. Let $H_1, \ldots, H_d$ be $d$ hyperplanes of $\mathbb{C}^n$, and let $D = (i_1, \ldots, i_{2k}) \in [d]^{2k}$ be a tuple of indices with $k \leq n$. Suppose that each $i \in [d]$ has a multiplicity of at most $n-1$ in $D$. Additionally, suppose there are no distinct $a, b \in [d]$ and $J \subset [n]$ of size $n-1$ such that $\{i_{2j-1}, i_{2j}\} = \{a, b\}$ for every $j \in J$. Then, it is possible to apply an infinitesimal motion to the hyperplanes $H_1, \ldots, H_d$ to obtain hyperplanes $\widetilde{H}_1, \ldots, \widetilde{H}_d$ such that the subspaces 
\[ \widetilde{H}_{i_1} \cap \widetilde{H}_{i_2}, \ldots, \widetilde{H}_{i_{2k-1}} \cap \widetilde{H}_{i_{2k}} \]
have the choice property. 
\end{lemma}

\begin{proof}
As in Lemma~\ref{ñ}, we begin by assuming the existence of hyperplanes $L_{1},\ldots,L_{d}$ in $\mathbb{C}^{n}$ such that the subspaces 
$$L_{i_{1}}\cap L_{i_{2}},\ldots, L_{i_{2k-1}}\cap L_{i_{2k}}$$
have the \textit{choice property}. By applying the same argument as in Lemma~\ref{ñ}, we conclude that we can find $\widetilde{H}_{1},\ldots,\widetilde{H}_{d}$ with the desired property.
So we only need to prove the existence of these hyperplanes $\set{L_{i}:i\in [d]}$. For that, we choose $L_{j}=\is{E\backslash \set{e_{j}}}$ for $j\in [d]$,  where $E=\set{e_{1},\ldots,e_{n}}$ is the canonical basis of $\CC^{n}$. Now we consider the multiset  $Z=\set{E\backslash \set{e_{i_{1}}},\ldots,E\backslash \set{e_{i_{2k}}}},$ where each element is a subset of size $n-1$ of $[n]$. As every $j\in [d]$ has multiplicity at most $n-1$ in $D$, we have that every subset of size $n-1$ of $[n]$ has multiplicity at most $n-1$ in $Z$. By our hypothesis, we also know that there are no $X,Y,J\subset [n]$ of size $n-1$, with $X\neq Y$, such that $\{E\backslash\{e_{i_{2j-1}}\},E\backslash\{e_{i_{2j}}\}\}=\set{X,Y}$ for every $j\in J$. Thus, the multiset $Z$ satisfies the hypothesis of Lemma~\ref{sets 2}, and so the sets 
\[E\backslash \{e_{i_{1}}\} \cap E\backslash \{e_{i_{2}}\}, \ldots, E\backslash \{e_{i_{2k-1}}\} \cap E\backslash \{e_{i_{2k}}\}\]
have the choice property. This implies, by Remark~\ref{choo 2}, that the subspaces 
\[\text{span}(E\backslash \{e_{i_{1}}\}) \cap \text{span}(E\backslash \{e_{i_{2}}\}), \ldots, \text{span}(E\backslash \{e_{i_{2k-1}}\}) \cap \text{span}(E\backslash \{e_{i_{2k}}\})\]
also have the choice property, as desired.
\end{proof}

The following lemma highlights a property %n important characteristic 
of sets of subspaces with the choice property.

\begin{lemma}\label{choo}
Let $V_{1},\ldots,V_{k}\subset \CC^{n}$ be subspaces that have the \textit{choice property} and let $\set{v_{1},\ldots,v_{k}}\subset \CC^{n}$ be a set of vectors with $v_{i}\in V_{i}$ for every $i$. Then it is possible to apply an infinitesimal motion to the vectors $\set{v_{1},\ldots,v_{k}}$ to obtain linearly independent vectors $\set{\widetilde{v_{1}},\ldots,\widetilde{v_{k}}}$ with $\widetilde{v_{i}}\in V_{i}$ for every $i$.
\end{lemma}

\begin{proof}
Consider the vectors $w_{i}\in V_{i}$ such that $w_{1},\ldots,w_{k}$ are linearly independent. Since they are independent, there exists a $k$-minor of the $n\times k$ matrix with columns ${w_{1},\ldots,w_{k}}$ that does not vanish. Without loss of generality, assume that it is the minor formed by the first $k$ coordinates of each $w_{i}$.

For each $i\in [k]$, let $z_{i}$ and $x_{i}$ denote the vectors containing the first $k$ coordinates of $w_{i}$ and $v_{i}$, respectively. We observe that $\det(z_{1},\ldots,z_{k})\neq 0$. Now, consider the following polynomial
$$p(\lambda)=\det(x_{1}+\lambda z_{1},\ldots,x_{k}+\lambda{z_{k}}).$$
By the linearity property of the determinant with respect to its columns, we can deduce that the leading coefficient of $p$ is $\det(z_{1},\ldots,z_{k})\neq 0$. Hence, $p$ is a non-zero polynomial, implying it has only a finite number of roots. As a result, we can select an infinitesimal $\lambda$ such that $p(\lambda)\neq 0$. Then, we define $\widetilde{v_{i}}=v_{i}+\lambda w_{i}\in V_{i}$ for every $i\in [k]$. Consequently, the minor of the submatrix formed by the first $k$ coordinates of each $\widetilde{v_{i}}$ has determinant 
$\det(x_{1}+\lambda z_{1},\ldots,x_{k}+\lambda z_{k})=p(\lambda)\neq 0.$
This implies that the vectors $\{\widetilde{v_{1}},\ldots,\widetilde{v_{k}}\}$ are linearly independent, leading us to the desired conclusion.
\end{proof}

We are now ready to prove Proposition~\ref{coincide}. Before stating the proof, we recall the notions of  $\gamma$-regular hyperplanes $R_\gamma$, %of $\gamma$, 
from Definition~\ref{pav}, and lifting number $M_\gamma$ from Definition~\ref{lif num}.

\begin{proof}[{\bf Proof of Proposition~\ref{coincide}}]
Given that no two hyperplanes $\gamma_{l_{1}},\gamma_{l_{2}}$ coincide for $l_{1}\neq l_{2}\in R_{\gamma}$, we can select a hyperplane $H_{l}$  of $\CC^{n}$ for each $l\in \mathcal{L}$ such that $\gamma_{l}\subset H_{l}$ and $H_{l_{1}}\neq H_{l_{2}}$ for $l_{1}\neq l_{2}\in \mathcal{L}$.

Let $C$ be a set of independent points of $M$ such that $
\{\gamma_{c}:c\in C\}$ forms a dependent set on $\gamma$. We will show that it is possible to enact an infinitesimal motion on $\gamma$ to obtain a collection of vectors $\tau\in V_{\mathcal{C}(M)}$, where the set $\{\tau_{c}:c\in C\}$ 
is independent. %, while still satisfying %the proposition's conditions.
%that $M_{\tau}=0$.
As the motion is infinitesimal, we can
ensure that there is no set of points $D\subset \mathcal{P}$ such that the vectors $\{\tau_{p}:p\in D\}$ are linearly dependent while the vectors $\{\gamma_{p}:p\in D\}$ are linearly independent. In other words, we can ensure that $\tau\leq \gamma$ as matroids.  Moreover, we can ensure that $c(M,\tau)=0$. %that the vectors $\{\tau_{c}:c\in C\}$ are linearly independent.
Thus, by repeating this procedure, we eventually reach a collection of points $\kappa\in \Gamma_{M}$ after a finite number of steps. Thus, we only need to prove that we can find such $\tau$ for every set $C$ with these conditions. We begin %the proof
by considering different cases.

\medskip
{\bf Case 1.} Assume that $C=\{p_{1},\ldots,p_{n-1}\}$ is %the
 a set of points lying within a dependent hyperplane $l\in \mathcal{L}$. Note that %these vectors should be independent, as 
the set $\{p_{1},\ldots,p_{n-1}\}$ is independent in $M$. Our objective now is to demonstrate that we can apply an infinitesimal motion to the vectors of $\gamma$  %ensure the vectors  become linearly independent while still maintaining the hypothesis of the proposition. 
 to obtain $\tau\in V_{\mathcal{C}(M)}$ where the set $\{\tau_{p_{i}}:i\in [n-1]\}$ is independent. For each $i\in [n-1]$, we consider the other dependent hyperplane in $\mathcal{L}_{p_{i}}$ distinct from $l$, denoted as $l_{p_{i}}$.

Note that not all hyperplanes $l_{p_{i}}$ are equal; otherwise, this dependent hyperplane would share $n-1$ points with $l$, which %contradicts the definition of a matroid. 
 is not possible for two distinct hyperplanes. Consequently, we have that the hyperplanes $H_{l}$ and $\{H_{l_{p_{i}}}\}_{i\in [n-1]}$ 
satisfy the hypothesis of Lemma~\ref{ñ}, where the numbers $c_{i}$ of Lemma~\ref{ñ} correspond to the multiplicity of the dependent hyperplanes $\{l_{p_{i}}\}_{i\in [n-1]}$ 
 within the multiset $\{l_{p_{1}},\ldots,l_{p_{n-1}}\}$. Therefore, by Lemma~\ref{ñ}, we can perform an infinitesimal motion on these hyperplanes to obtain hyperplanes $\widetilde{H}_{l}$ and $\{\widetilde{H}_{l_{p_{i}}}\}_{i\in [n-1]}$ such that the following subspaces have the choice property: 
%given by 
\begin{equation}\label{ex}
\widetilde{H}_{l}\cap \widetilde{H}_{l_{p_{1}}},\ldots,\widetilde{H}_{l}\cap \widetilde{H}_{l_{p_{n-1}}}
\end{equation}
%For the sake of simplicity in 
 To simplify notation, we set $\widetilde{H}_{l'}=H_{l'}$, for $l' \not \in \{l,l_{p_{1}},\ldots,l_{p_{n-1}}\}$. %will
%also denote the hyperplanes 
 %\mathcal{L}\backslash 
For each $i\in [n-1]$, since $H_{l}\neq H_{ l_{p_{i}}}$ by construction, we apply Lemma~\ref{move} to select a vector $\gamma_{p_{i}}' \in\widetilde{H}_{l}\cap \widetilde{H}_{l_{p_{i}}}$, which represents an infinitesimal motion of $\gamma_{p_{i}}$.  By applying Lemma~\ref{choo} and utilizing that  the subspaces in \eqref{ex} have the choice property, we select linearly independent vectors $\tau_{p_{i}} \in\widetilde{H}_{l}\cap \widetilde{H}_{l_{p_{i}}}$ that represent infinitesimal motions of the vectors $\gamma_{p_{i}}'$ for $i\in [n-1]$. %, ensuring linear independence among ${\tau_{p_{1}},\ldots,\tau_{p_{n-1}}}$ as a consequence.
%For each $i\in [n-1]$, we move the vector $\gamma_{p_{i}}$ to $\tau_{p_{i}}$. 
For  each $q \notin \{p_{1},\ldots,p_{n-1}\}$,  %and $H_{l_{1}}\neq H_{l_{2}}$, 
we apply Lemma~\ref{move} to select a vector $\tau_{q}  \in \widetilde{H}_{l_{1}}\cap \widetilde{H}_{l_{2}}$ which represents an infinitesimal motion of $\gamma_{q}$,  where $\mathcal{L}_{q}=\{l_{1},l_{2}\}$. %and we move $\gamma_{q}$ to $\tau_{q}$.

The new collection of vectors $\tau$ is obtained by applying an infinitesimal motion to the vectors of $\gamma$. We observe that $\tau_{l'}\subset \widetilde{H}_{l'}$ for every $l' \in \mathcal{L}$, and consequently, $\tau\in V_{\mathcal{C}(M)}$. Since the motion is infinitesimal, we can ensure that $\widetilde{H}_{l_{1}}\neq\widetilde{H}_{l_{2}}$ for $l_{1}\neq l_{2}$, implying $c(M,\tau)=0$. Since $\{\tau_{p_{1}},\ldots,\tau_{p_{n-1}}\}$ are independent and $c(M,\tau)=0$, we have found $\tau$ for this case.

\medskip
\textbf{Case 2.} Suppose $C= \{q_{1},\ldots,q_{k}\}$ with $k\leq n$ is a set of points not all lying in the same dependent hyperplane of $M$, %and
 such that the vectors ${\gamma_{q_{1}},\ldots,\gamma_{q_{k}}}$ form a circuit.
Note that 
%these vectors should not be dependent, as 
the points $\{q_{1},\ldots,q_{k}\}$ are independent in $M$. Therefore, our objective now is to demonstrate that we can apply an infinitesimal motion to the vectors of $\gamma$  to obtain $\tau\in V_{\mathcal{C}(M)}$ where the set $\{\tau_{q_{i}}:i\in [k]\}$ is independent. %to ensure that the vectors corresponding to $\{q_{1},\ldots,q_{k}\}$ become linearly independent while still maintaining the hypothesis of the proposition.
For each $i\in [k]$, we denote $l_{i_{1}},l_{i_{2}}$ as the dependent hyperplanes of $\mathcal{L}_{q_{i}}$. 

%\textcolor{red}{I swapped the cases $2.1$ and $2.2$.}

\medskip
{\bf Case 2.1.} Assume that all the dependent hyperplanes of the multiset $\{l_{i_{j}}:i\in [k],j\in [2]\}$ have a multiplicity of at least two. Let $N$ be the number of distinct dependent hyperplanes in $\{l_{i_{j}}:i\in [k],j\in [2]\}$. In particular, we have $2N\leq 2k \leq 2n$, implying $N\leq n$. Additionally, there %does
 do not exist dependent hyperplanes $l_{1}\neq l_{2}\in \mathcal{L}$ and  a subset $J\subset [n]$ of size $n-1$  such that $\{l_{j_{1}},l_{j_{2}}\}=\{l_{1},l_{2}\}$ for every $j\in J$. If that were the case, $l_{1}$ and $l_{2}$ would share at least $n-1$ points of $M$, which is not possible  for two distinct hyperplanes. Furthermore, since not all points $\{q_{1},\ldots,q_{k}\}$ belong to the same dependent hyperplane, every dependent hyperplane in the multiset $ \{l_{i_{j}}:i\in [k],j\in [2]\}$ has a multiplicity of at most $n-1$. Therefore, all the hypotheses of Lemma~\ref{lemm 2} are satisfied for the dependent hyperplanes in this multiset. Consequently, there %are 
 exist hyperplanes $\{\widetilde{H}_{l_{i_{j}}}:i\in [k], j\in \{1,2\}\}$ that represent an infinitesimal motion of the hyperplanes $\{H_{l_{i_{j}}}:i\in [k], j\in \{1,2\}\}$ such that the subspaces 
\begin{equation}\label{eq 2k} \widetilde{H}_{l_{1_{1}}}\cap \widetilde{H}_{l_{1_{2}}},\ldots,\widetilde{H}_{l_{k_{1}}}\cap \widetilde{H}_{l_{k_{2}}} \end{equation}
have the choice property.
To simplify notation,  we %will denote the hyperplanes 
set $\widetilde{H}_{l'}=H_{l'}$ %even 
for $l' \not \in %\mathcal{L}\backslash 
\{l_{1_{1}},\ldots,l_{k_{2}}\}$. For each $i\in [k]$, we know that $H_{l_{i_{1}}}\neq H_{l_{i_{2}}}$ by construction. Then, applying Lemma~\ref{move}, we select a vector $\gamma_{q_{i}}'$ from $\widetilde{H}_{l_{i_{1}}}\cap \widetilde{H}_{l_{i_{2}}}$, which represents an infinitesimal motion of $\gamma_{q_{i}}$. As the subspaces of~\eqref{eq 2k} have the choice property, we know by applying Lemma~\ref{choo} that we can choose  linearly independent vectors $\tau_{q_{i}}\in\widetilde{H}_{l_{i_{1}}}\cap \widetilde{H}_{l_{i_{2}}}$, representing infinitesimal motions of the vectors $\gamma_{q_{i}}'$ for $i\in [k]$ (and consequently of the vectors $\gamma_{q_{i}}$). %, while ensuring that $\{\tau_{q_{1}},\ldots,\tau_{q_{k}}\}$ are linearly independent.
%For each $i\in [k]$, we move the vector $\gamma_{q_{i}}$ to $\tau_{q_{i}}$. 
 For each $q \notin \{q_{1},\ldots,q_{k}\}$,   
we apply Lemma~\ref{move} to select a vector $\tau_{q} \in \widetilde{H}_{l_{1}}\cap \widetilde{H}_{l_{2}}$ which represents an infinitesimal motion of $\gamma_{q}$, where $\mathcal{L}_{q}=\{l_{1},l_{2}\}$.
%For each $q\notin \{q_{1},\ldots,q_{k}\}$, we have that $H_{l_{1}}\neq H_{l_{2}}$, where $\{l_{1},l_{2}\}=\mathcal{L}_{q}$. Then, applying Lemma~\ref{move}, we select a vector $\tau_{q}$ from $\widetilde{H}_{l_{1}}\cap \widetilde{H}_{l_{2}}$, which represents an infinitesimal motion of $\gamma_{q}$, and we move $\gamma_{q}$ to $\tau_{q}$.

%\\[0.7 ex]
We have $\tau_{l'}\subset \widetilde{H}_{l'}$ for every $l' \in \mathcal{L}$, which implies that $\tau \in V_{\mathcal{C}(M)}$. As the motion is infinitesimal, we can ensure that $\widetilde{H}_{l_{1}}\neq\widetilde{H}_{l_{2}}$ for $l_{1}\neq l_{2}$, and this implies that $c(M,\gamma)=0$. Since $\{\tau_{q_{1}},\ldots,\tau_{q_{k}}\}$ are independent and ${c(M,\tau)}=0$, we have found $\tau$ for this case.

\medskip
\textbf{Case 2.2.}  Assume there is a dependent hyperplane with multiplicity one in the multiset $ \{l_{i_{j}}:i\in [k],j\in [2]\}$, and without loss of generality, let this hyperplane be $l_{1_{1}}$. Since not all points $q_{i}$ belong to $l_{1_{2}}$, we can enact an infinitesimal motion on the hyperplane $H_{l_{1_{2}}}$ to ensure that  $\{\gamma_{q_{2}},\ldots,\gamma_{q_{k}}\}\not \subset H_{l_{1_{2}}}$. %$H_{l_{1_{2}}}$ does not contain $\{\gamma_{q_{2}},\ldots,\gamma_{q_{k}}\}$.
Hence, we may make this assumption. Thus, we have
\[ 
H_{l_{1_{2}}}\wedge \gamma_{q_{2}}\cdots \gamma_{q_{k}}\neq 0. \]
Then we can apply an infinitesimal motion to $H_{l_{1_{1}}}$ to obtain a hyperplane $\widetilde{H}_{l_{1_{1}}}$ such that 
\[ \widetilde{H}_{l_{1_{1}}}\wedge H_{l_{1_{2}}}\wedge \gamma_{q_{2}}\cdots \gamma_{q_{k}}\neq 0. 
\]
In particular, for this hyperplane, we have that $\widetilde{H}_{l_{1_{1}}}\cap H_{l_{1_{2}}}\not \subset \langle \gamma_{q_{2}},\ldots ,\gamma_{q_{k}} \rangle$, so we can take a vector 
\[ w\in \widetilde{H}_{l_{1_{1}}}\cap H_{l_{1_{2}}}\backslash \langle \gamma_{q_{2}},\ldots ,\gamma_{q_{k}} \rangle. \]
As $H_{l_{1_{1}}}\neq H_{l_{1_{2}}}$, applying Lemma~\ref{move}, we %know that there exists 
 select a vector $\gamma_{q_{1}}'\in \widetilde{H}_{l_{1_{1}}}\cap H_{l_{1_{2}}}$ %that is
 which represents an infinitesimal motion of $\gamma_{q_{1}}$.
Since $w\notin \langle \gamma_{q_{2}},\ldots ,\gamma_{q_{k}} \rangle$, we can choose an infinitesimal $\lambda$ such that 
\[
\tau_{q_{1}}=\gamma_{q_{1}}'+\lambda w\in \widetilde{H}_{l_{1_{1}}}\cap H_{l_{1_{2}}} \backslash \langle \gamma_{q_{2}},\ldots ,\gamma_{q_{k}} \rangle.
\]
%Then we move $\gamma_{q_{1}}$ to $\tau_{q_{1}}$, and 
 We set $\tau_{q_{i}}=\gamma_{q_{i}}$ for $2\leq i \leq k$, thus ensuring that $\{\tau_{q_{1}},\tau_{q_{2}},\ldots,\tau_{q_{k}}\}$ are linearly independent. %For ease of notation, we will now denote the hyperplanes $H_{l'}$ as $\widetilde{H}_{l'}$ even 
 To simplify notation, we also set $\widetilde{H}_{l'}=H_{l'}$ for $l'\neq l_{1_{1}}$. 
 For each $q \notin \{q_{1},\ldots,q_{k}\}$,   
we apply Lemma~\ref{move} to select a vector $\tau_{q} \in \widetilde{H}_{l_{1}}\cap \widetilde{H}_{l_{2}}$ which represents an infinitesimal motion of $\gamma_{q}$, where $\mathcal{L}_{q}=\{l_{1},l_{2}\}$.
%For every point $q\notin \{q_{1},\ldots,q_{k}\}$, we have that $H_{l_{1}}\neq H_{l_{2}}$, where $\{l_{1},l_{2}\}=\mathcal{L}_{q}$. Then, applying Lemma~\ref{move},~we select a vector $\tau_{q}\in \widetilde{H}_{l_{1}}\cap \widetilde{H}_{l_{2}}$ that is an infinitesimal motion of $\gamma_{q}$, and we move $\gamma_{q}$ to $\tau_{q}$.
As before, the new collection of vectors $\tau$ is obtained after applying an infinitesimal motion to the vectors of $\gamma$. Since $\tau_{l'}\subset \widetilde{H}_{l'}$ for every $l' \in \mathcal{L}$, we have that $\tau\in V_{\mathcal{C}(M)}$. As the motion is infinitesimal, we can ensure that $\widetilde{H}_{l_{1}}\neq\widetilde{H}_{l_{2}}$ for $l_{1}\neq l_{2}$, which implies ${c(M,\tau)}=0$. Since $\{\tau_{q_{1}},\ldots,\tau_{q_{k}}\}$ are independent and ${c(M,\tau)}=0$, we have found $\tau$ for this case.

\medskip
After iteratively applying this procedure to each of these cases, we eventually arrive, within a finite number of steps, at a collection of points $\kappa\in V_{\mathcal{C}(M)}$ such that $\{\kappa_{c}:c\in C\}$ is independent for every set of points $C$ belonging to one of these two types:
\begin{itemize}
\item $C$ consists of $n-1$ points lying within the same dependent hyperplane of $M$.
\item $C$ consists of at most $n$ points, not all lying in the same dependent hyperplane of $M$.
\end{itemize}
Thus, $\kappa \in \Gamma_{M}$, and since all motions were carried out infinitesimally, we conclude that $\gamma \in V_{M}$.
\end{proof}

\begin{remark}
In the proof of Proposition~\ref{coincide}, we assumed $\size{\mathcal{L}_{p}}=2$ for every $p\in C$. However, it is also possible that $\size{\mathcal{L}_{p}}=1$. In such case, we can introduce an additional hyperplane passing through the vector $\gamma_{p}$ without coinciding with any hyperplane $\{H_{l}:l\in \mathcal{L}\}$, and apply the same argument.
\end{remark}

\subsection{Liftability property}

%\textcolor{red}{I moved Definition~\ref{def Sm}, Lemma~\ref{lift n-2} and Lemma~\ref{lift otra} from Subection 4.1 to Subsection 4.2 since we were not using them in Subsection 4.1.}

 This subsection aims to prove Proposition~\ref{propo 2}, which is essential for establishing the main result of this section, Theorem~\ref{theo lift}. %To do that, we need to establish a series of lemmas and propositions.
We begin by introducing some definitions. Throughout this subsection, recall Definitions~\ref{pav} and~\ref{lif num}, and the notion of {\em submatroid of hyperplanes} from Definition~\ref{submatroid of hyperplanes}.

%\textcolor{red}{Here I changed the notation $\{\gamma_{p}:p\in N\}$ to $\{\gamma_{p}: p\in \mathcal{P}_{N}\}$ since $N$ is a matroid and not a set of points. I did the same in the rest of the file.}

\begin{definition}[Fully liftable]\normalfont
Let $M$ be an $n$-paving matroid and let $\gamma\in V_{\mathcal{C}(M)}$. We say that $\gamma$ is \textit{fully liftable} from a vector $q\in \mathbb{C}^{n}$ outside $\gamma$ if, for every subset $L$ containing at least two dependent hyperplanes from $R_{\gamma}$, the vectors $\{\gamma_{p}: p\in  \mathcal{P}_{M^{L}}\}$ can be lifted in a non-degenerate manner from $q$ to a set of vectors in $V_{\mathcal{C}(M^{L}})$.%, where $N$ is the submatroid of hyperplanes of $L$.
\end{definition}

\begin{remark}\label{lif}
Note that by Lemma~\ref{lifting 3}, if $\gamma$ is a collection of vectors in $V_{\mathcal{C}(M)}\cap V(I_{M}^{\lift})$, then $\gamma$ is \textit{fully liftable} from any vector $q\in \mathbb{C}^{n}$ outside $\gamma$.
\end{remark}

\begin{definition}\label{def Sm}
To any $n$-paving matroid $M$, and a dependent hyperplane $l \in \mathcal{L}$, we associate the sets: 
\[
S_{l}=\{p \in l : |\mathcal{L}_{p}| \geq 2\} \quad \text{and}\quad S_M=\{p \in \mathcal{P} : |\mathcal{L}_{p}| \geq 2\}.
\]
\end{definition}

% For the following lemmas recall Definition~\ref{pav}.}

\begin{lemma}\label{lift n-2}
Let $M$ be an $n$-paving matroid, and let $\gamma\in V_{\mathcal{C}(M)}$ be a \textit{regular} collection of vectors %in $V_{\mathcal{C}(M)}$ 
with $\rank \{\gamma_{p}: p \in S_{l}\} \leq n-2$ for every $l \in \mathcal{L}$. Suppose $q \in \mathbb{C}^{n}$ is a vector \textit{outside} $\gamma$. Then, %the vectors of 
$\gamma$ can be non-degenerately lifted %in a non-degenerate way 
from $q$ to a regular collection %of vectors 
$\widetilde{\gamma}\in V_{\mathcal{C}(M)}$ with ${c(M,\widetilde{\gamma})} = 0$. Moreover, $\widetilde\gamma$ can be chosen %to be infinitesimal.
 as an infinitesimal motion of $\gamma$.
\end{lemma}

\begin{proof}
 For each $l\in \mathcal{L}$, let $V_{S_{l}}$ the subspace spanned by $\{\gamma_{p}: p \in S_{l}\}$. Since $\dim(V_{S_{l}}) \leq n-2$, %for each $l \in \mathcal{L}$ 
we can choose a hyperplane $H_{l}$ of dimension $n-1$ that does not contain $q$, such that $V_{S_{l}} \subset H_{l}$, and $H_{l_{1}} \neq H_{l_{2}}$ for any $l_{1} \neq l_{2} \in \mathcal{L}$. For each point $p \in l$, we define $\widetilde{\gamma}_{p}$ as the projection of $\gamma_{p}$ onto the hyperplane $H_{l}$ from the vector $q$. This projection is well-defined because if $p \in l_{1}\cap l_{2}$, then $p \in V_{S_{l_{1}}}\cap V_{S_{l_{2}}}$, meaning that $p$ projects to itself in $H_{l_{1}}$ and $H_{l_{2}}$. Consequently $\widetilde{\gamma}_{l}=H_{l}$, ensuring all hyperplanes are distinct, which completes the proof. Moreover, $\widetilde\gamma$ can be chosen to be infinitesimal by selecting, for each $l \in \mathcal{L}$, the hyperplane $H_{l}$ as an infinitesimal motion of the hyperplane $\gamma_{l}$.
\end{proof}

\begin{lemma}\label{lift otra}
Let $M$ be an $n$-paving matroid and let $\gamma_{1}$ and $\gamma_{2}$ be two \textit{regular} collections of vectors, indexed by $\mathcal{P}$,  inside a hyperplane $H\subset \CC^{n}$. Suppose that  ${\gamma_{1}}_{p}={\gamma_{2}}_{p}$ for every $p\in S_{M}$ and that the vectors of $\gamma_{1}$ can be lifted in a non-degenerate way from a vector $q\notin H$   to a collection of vectors in $V_{\mathcal{C}(M)}$, then we can also find such a lifting from the point $q$ for the vectors of $\gamma_{2}$.
\end{lemma}

\begin{proof}
Let $\widetilde{\gamma_{1}}\in V_{\mathcal{C}(M)}$ be the lifting of the vectors of $\gamma_{1}$ from the point $q$, by hypothesis we know that not all the hyperplanes $\set{\widetilde{\gamma_{1}}_{l}:l\in \mathcal{L}}$ are the same. Now we define the collection of vectors $\widetilde{\gamma_{2}}$, where for the point $p\in l$ the vector $\widetilde{\gamma_{2}}_{p}$ is defined as the projection of ${\gamma_{2}}_{p}$  onto the hyperplane ${\widetilde{\gamma_{1}}}_{l}$  from the vector $q$. This is well-defined, because if $p\in l_{1}\cap l_{2}$ then $p\in S_{M}$ and ${\gamma_{1}}_{p}={\gamma_{2}}_{p}$, so the vector ${\gamma_{2}}_{p}$ %lifts
 is projected to  ${\widetilde{\gamma_{1}}}_{p}$. Then, we have $\widetilde{\gamma_{1}}_{l}=\widetilde{\gamma_{2}}_{l}$ for every $l\in \mathcal{L}$, implying that not all the vectors of $\widetilde{\gamma_{2}}$ lie on the same hyperplane. Therefore, $\widetilde{\gamma_{2}}$ is the desired lifting.
\end{proof}

\begin{proposition}\label{propo}
Let $M$ be an $n$-paving matroid. % with no points of degree greater than two. 
Suppose $\gamma\in V_{\mathcal{C}(M)}$ is a \textit{regular} collection of vectors that is \textit{fully liftable} from a point $q\in \mathbb{C}^{n}$ outside $\gamma$, and ${c(M,\gamma)}\neq 0$. Then, we can apply an infinitesimal motion to the vectors of $\gamma$ to obtain a regular collection of vectors $\sigma\in V_{\mathcal{C}(M)}$ that remains \textit{fully liftable} from $q$ and with ${c(M,\sigma)}<{c(M,\gamma)}$. Moreover, we can ensure that $\sigma_{p}\in \is{\gamma_{p},q}$ for every $p\notin S_{M}$.
\end{proposition}

\begin{proof}
Recall that since $\gamma$ is regular, $\dim(\gamma_{l})=n-1$ for every $l\in \mathcal{L}$, implying that each subspace $\gamma_{l}$ is a hyperplane in $\mathbb{C}^{n}$. We can establish a disjoint partition of $\mathcal{L}$ as 
$\mathcal{L}=\cup_{i=1}^{d}\mathcal{L}_{i}$,
where two dependent hyperplanes $l_{1}\neq l_{2}\in \mathcal{L}$ belong to the same subset of the partition if and only if $\gamma_{l_{1}}=\gamma_{l_{2}}$.

%If all subsets
 Observe that not all subsets $\mathcal{L}_{i}$ consist of a single element, %then we would have 
 since ${c(M,\gamma)}>0$. %, contradicting the hypothesis of the proposition. 
%Thus, suppose $\lvert\mathcal{L}_{i}\rvert>1$ for some $i$. %For this particular $i$, consider the submatroid of hyperplanes defined by $\mathcal{L}_{i}$, denoted as $M^{\mathcal{L}_{i}}$. 
Note that the vectors $\{\gamma_{p}:p\in  \mathcal{P}_{M^{\mathcal{L}_{i}}}\}$ lie within the same hyperplane in $\mathbb{C}^{n}$,  for each $i\in [d]$. Since $\gamma$ is fully liftable from $q$,  for each $i\in [d]$ there exists a lifting of the vectors $\{\gamma_{p}:p\in  \mathcal{P}_{M^{\mathcal{L}_{i}}}\}$ from $q$ to a collection of vectors $\gamma^{i}=\set{\gamma^{i}_{p}:p\in  \mathcal{P}_{M^{\mathcal{L}_{i}}}}\in V_{\mathcal{C}(M^{\mathcal{L}_{i}})}$,  such that not all vectors lie in the same hyperplane when $\size{\mathcal{L}_{i}}\geq 2$. By employing Lemma~\ref{lifting 3}, we can ensure that these liftings are infinitesimal.

\medskip
We now define the collection of vectors $\sigma=\{\sigma_{p}:p\in  \mathcal{P}_{M}\}$ as follows:
\begin{itemize}
\item If $\emptyset \subsetneq \mathcal{L}_{p}\subset \mathcal{L}_{i}$  for some $i\in [d]$, we set $\sigma_{p}=\gamma^{i}_{p}$.
\item If $p\in  \mathcal{P}_{M}$ belongs to the dependent hyperplanes $l_{1}$ and $l_2$, where $l_{1}\in \mathcal{L}_{i},l_{2} \in \mathcal{L}_{ j}$ and $i\neq j$, %and $\gamma_{l_{1}}\neq \gamma_{l_{2}}$, then 
then by applying Lemma~\ref{move}, there exists a vector in $\gamma^{i}_{l_{1}}\cap \gamma^{j}_{l_{2}}$ that represents an infinitesimal motion of $\gamma_{p}$. We define $\sigma_{p}$ as this vector.
\item  If $\mathcal{L}_{p}=\emptyset$, we set $\sigma_{p}=\gamma_{p}$.
%\item If $\mathcal{L}_{p}\cap \mathcal{L}_{i}=\emptyset$ we define $\sigma_{p}=\gamma_{p}$.
\end{itemize}

Note that by definition, the vectors in $\sigma$ represent an infinitesimal motion of the vectors in $\gamma$. Furthermore,  for every $i\in [d]$ and each %dependent hyperplane 
$l\in \mathcal{L}_{i}$, we have $\sigma_{l}=\gamma^{i}_{l}$. %and for every $l\notin \mathcal{L}_{i}$ we have $\sigma_{l}=\gamma_{l}$.
Thus, $\dim(\sigma_{l})\leq n-1$ for all $l\in \mathcal{L}$, implying $\sigma\in V_{\mathcal{C}(M)}$. We can ensure that $\sigma$ is \textit{regular}, since $\gamma$ was \textit{regular} and the motions are infinitesimal. Moreover,  for each $i$ with $\size{\mathcal{L}_{i}}\geq 2$, not all hyperplanes in $\{\sigma_{l}:l\in \mathcal{L}_{i}\}$ are the same.  This is because the vectors in $\gamma^{i}$ do not all lie on the same hyperplane, as the lifting of the vectors $\{\gamma_{p}:p\in  \mathcal{P}_{M^{\mathcal{L}_{i}}}\}$ is non-degenerate. This implies that the lifting number is strictly reduced. Additionally, if $p$ is a point not in $S_{M}$, then by the definition of $\sigma_{p}$, we have $\sigma_{p}\in \is{\gamma_{p},q}$. However, we still need to prove that $\sigma$ is fully liftable from $q$.

Consider a subset $L$ of $\mathcal{L}$ 
containing at least two dependent hyperplanes such that the vectors $\{\sigma_{p}: p %\text{ belongs to some } l\in L\}$
 \in \cup_{l\in L}l\}$ all lie in the same hyperplane. We have to prove that these vectors are liftable from the point $q$. %Let $N$ denote the submatroid of hyperplanes in $L$. 
Since all vectors $\{\gamma_{p}: p\in  \mathcal{P}_{M^{L}}\}$ were originally situated on the same hyperplane,  there exists $j\in [d]$ such that $L\subset \mathcal{L}_{j}$. As $\gamma$ is \textit{fully liftable} from $q$, the collection of vectors $\{\gamma_{p}: p\in  \mathcal{P}_{M^{L}}\}$ can be lifted in a non-degenerate manner from the point $q$. Furthermore, all vectors $\{\gamma^{j}_{p}: p\in  \mathcal{P}_{M^{L}}\}$ belong to the same hyperplane, %(where we set $\gamma^{j}_{p}=\gamma_{p}$ for $j\neq i$), 
and they are obtained as a lifting of the vectors $\{\gamma_{p}: p\in  \mathcal{P}_{M^{L}}\}$ from the point $q$. Hence, we conclude that they are also liftable in a non-degenerate manner from the point $q$.

Consider the collections of vectors $\{\gamma^{j}_{p}: p\in  \mathcal{P}_{M^{L}}\}$ and $\{\sigma_{p}: p\in  \mathcal{P}_{M^{L}}\}$. They span the same hyperplane, and the vectors in $\{\gamma^{j}_{p}: p\in  \mathcal{P}_{M^{L}}\}$ are liftable in a non-degenerate manner from the point $q$. Moreover, by the definition of $\sigma$, if $p\in S_{M^{L}}$ (a point of degree two within the submatroid $M^{L}$), then $\sigma_{p}=\gamma^{j}_{p}$. Additionally, both $\gamma$ and $\sigma$ are \textit{regular} collections. Therefore, all the hypotheses of Lemma~\ref{lift otra} are satisfied for these two collections of vectors. Consequently, Lemma~\ref{lift otra} asserts that the vectors $\{\sigma_{p}: p\in  \mathcal{P}_{M^{L}}\}$ are also liftable from the point $q$ in a non-degenerate manner, as desired.
\end{proof}

\begin{corollary}\label{coro}
Let $M$ be an $n$-paving matroid. % with no points of degree greater than two. 
Suppose $\gamma\in V_{\mathcal{C}(M)}$ is a \textit{regular} collection of vectors that is \textit{fully liftable} from the point $q\in \mathbb{C}^{n}$ outside $\gamma$. Then, we can apply an infinitesimal motion to the vectors of $\gamma$ to obtain a regular collection of vectors $\kappa\in V_{\mathcal{C}(M)}$ such that ${c(M,\kappa)}=0$ and $\kappa_{p}\in\is{\gamma_{p},q}$ for every point $p\notin S_{M}$.
\end{corollary}

\begin{proof}
The result follows by repeatedly applying Proposition \ref{propo} to the collection of vectors $\gamma$. After a finite number of steps, we arrive at the collection of vectors $\kappa$ with the desired properties. Since all motions involved are infinitesimal, $\kappa$ is an infinitesimal motion of $\gamma$, thus proving the result.
\end{proof}

%In the following proposition, 
We will now show that the assumption of $\gamma$ being regular can be removed from Proposition~\ref{propo}.

\begin{proposition}\label{propo 2}
Let $M$ be an $n$-paving matroid. % with no points of degree greater than two. 
Suppose $\gamma\in V_{\mathcal{C}(M)}$ is a collection of vectors that is fully liftable from the point $q\in \CC^{n}$ outside $\gamma$. Then, we can apply an infinitesimal motion to $\gamma$ to obtain a collection of vectors in $V_{\mathcal{C}(M)}$ with a lifting number equal to zero.
\end{proposition}

\begin{proof}
Consider the submatroid of hyperplanes of $R_{\gamma}$,  denote as $M^{R_{\gamma}}$. By the definition of $R_{\gamma}$, we know that the collection of vectors $\{\gamma_{p}: p\in  \mathcal{P}_{M^{R_{\gamma}}}\}$ is regular and fully liftable from the point $q$ (with respect to $M^{R_{\gamma}}$), by hypothesis. Then, applying Corollary~\ref{coro}, we know that there exists a collection of vectors $\kappa=\{\kappa_{p}: p\in  \mathcal{P}_{M^{R_{\gamma}}}\}$ that is an infinitesimal motion of $\restr{\gamma}{ \mathcal{P}_{M^{R_{\gamma}}}}$, with ${c(M,\kappa)}=0$ and $\kappa_{p}\in  \is{\gamma_{p},q}$ for every $p\notin S_{M^{R_{\gamma}}}$. Now, we define the collection of vectors $\tau=\{\tau_{p}: p\in  \mathcal{P}_{M}\}$ as follows:
\begin{itemize}
\item If $p\in  \mathcal{P}_{M^{R_{\gamma}}}$, we define $\tau_{p}=\kappa_{p}$.
\item If $p\notin  \mathcal{P}_{M^{R_{\gamma}}}$, we define $\tau_{p}=\gamma_{p}$.
\end{itemize}

By definition, if $l\in R_{\gamma}$, then $\tau_{l}=\kappa_{l}$. Now, let $l\notin R_{\gamma}$. For this dependent hyperplane, we know that $\dim(\gamma_{l})\leq n-2$. Since $\size{\mathcal{L}_{p}}\leq 2$, every point $p\in l$ does not belong to $S_{M^{R_{\gamma}}}$. Hence, $\tau_{p}\in \is{\gamma_{p},q}$ for every $p\in l$. Consequently, $\tau_{l}\subset \is{\gamma_{l},q}$, implying that $\dim(\tau_{l})\leq n-1$. From this, we deduce that $\tau \in V_{\mathcal{C}(M)}$.

\medskip

Note that %it is not possible to have 
 $\tau_{l_{1}}\neq \tau_{l_{2}}$ for any two dependent hyperplanes $l_{1}\in R_{\gamma}$ and $\ l_{2}\in \mathcal{L}\backslash R_{\gamma}$, since $q$ %cannot
does not lie in $\tau_{l_{1}}$, and $\tau_{l_{2}}$ must contain $q$ to achieve dimension $n-1$.  Moreover, due to the construction of $\tau$ and the fact that ${c(M,\kappa)}=0$, we have $\tau_{l_{1}}\neq \tau_{l_{2}}$ for any pair of $l_1,l
_2\in R_{\gamma}$ with $l_{1}\neq l_{2}$. However, it is still possible to have dependent hyperplanes $l_{1}\neq l_{2} \in R_{\tau}\backslash R_{\gamma}$ with $\tau_{l_{1}}=\tau_{l_{2}}$.

\medskip

We now denote the submatroid of hyperplanes of $R_{\tau}\backslash R_{\gamma}$ by $K$.
We fix a point $q_{1}$ outside $\tau$. 
Note that if a point $p\in  \mathcal{P}_{M}$ belongs to two dependent hyperplanes of $K$, then it was not moved during the liftings. This implies that $\tau_{p}=\gamma_{p}$ for every $p\in S_{K}$. Consequently, for every $l\in R_{\tau}\backslash R_{\gamma}$, we have:
$$
\rank \left\{\tau_{p}:p\in S_{l}\right\}=\rank \left\{\gamma_{p}:p\in S_{l}\right\}\leq \dim(\gamma_{l})\leq n-2,
$$
where $S_{l}$ is taken with respect to the matroid $K$ and not with respect to $M$.

Therefore, the  matroid $K$ and the collection of vectors $\set{\tau_{p}:p\in  \mathcal{P}_{K}}$ satisfy the hypotheses of Lemma~\ref{lift n-2} (since we also have $\dim(\tau_{l})=n-1$ for any $l\in R_{\tau}$). Consequently, there exists  an infinitesimal lifting of these vectors from the point $q_{1}$ to a collection of vectors $\widetilde{\tau}=\set{\widetilde{\tau}_{p}:p\in  \mathcal{P}_{K}}$ such that $\widetilde{\tau}_{l_{1}}\neq \widetilde{\tau}_{l_{2}}$ for every $l_{1}\neq l_{2}\in R_{\tau}\backslash R_{\gamma}$.
With this, we define a collection of vectors $\pi=\set{\pi_{p}:p\in  \mathcal{P}_{M}}$ as follows:
\begin{itemize}
\item If $p\notin  \mathcal{P}_{K}$ we define $\pi_{p}=\tau_{p}$.
\item If $p\in  \mathcal{P}_{K}$ and $p$ does not belong to any dependent hyperplane of $R_{\gamma}$ we define $\pi_{p}=\widetilde{\tau}_{p}$.
\item If $\mathcal{L}_{p}=\set{l_{1},l_{2}}$ with $l_{1}\in R_{\tau}\backslash R_{\gamma}$ and $l_{2}\in R_{\gamma}$ we know that $\tau_{l_{1}}\neq \tau_{l_{2}}$, then applying Lemma~\ref{move} we select a vector in $\widetilde{\tau}_{l_{1}}\cap \tau_{l_{2}}$ that is an infinitesimal motion of $\tau_{p}$ and we define $\pi_{p}$ as this vector. 
\end{itemize}
Observe that by definition, the vectors of $\pi$ are an infinitesimal motion of the vectors of $\tau$. For each $l\in \mathcal{L}$, the subspace $\pi_{l}$ is defined as follows:
\begin{itemize}
\item If $l\in R_{\gamma}$, then $\pi_{l}=\tau_{l}$.
\item If $l\in R_{\tau}\backslash R_{\gamma}$, then $\pi_{l}=\widetilde{\tau}_{l}$.
\item If $l\notin R_{\tau} $, then we have $\pi_{l}\subset \is{\tau_{l},q_{1}}$. Moreover, by the definition of $R_{\tau}$, we know that $\dim(\tau_{l})\leq n-2$, hence $\dim(\pi_{l})\leq n-1$.
\end{itemize}
As a consequence of this analysis, we conclude that $\pi \in V_{\mathcal{C}(M)}$. Note that:
\begin{itemize}
    \item For any $l_{1}\neq l_{2}\in R_{\gamma}$, we have  $\pi_{l_{1}}\neq \pi_{l_{2}}$.
    \item For any $l_{1}\neq l_{2}\in R_{\tau}\backslash R_{\gamma}$,  the definition of $\widetilde{\tau}$ ensures that $\widetilde{\pi}_{l_{1}}\neq \widetilde{\pi}_{l_{2}}$.
    \item For any $l_{1}\in R_{\gamma}$ and $l_{2}\in R_{\tau}\backslash R_{\gamma}$, we have $\pi_{l_{1}}\neq \pi_{l_{2}}$ because $\tau_{l_{1}}\neq \tau_{l_{2}}$.
\end{itemize}
Thus, we deduce that $\pi_{l_{1}}\neq \pi_{l_{2}}$ for any two dependent hyperplanes  $l_{1}\neq l_{2}\in R_{\tau}$.

\medskip

Similarly, for each point $p\in  \mathcal{P}_{M}$ belonging to two dependent hyperplanes of $R_{\pi}\backslash R_{\tau}$, we have $\pi_{p}=\tau_{p}$.  Consequently, %we can repeat the argument 
by repeating the above argument, we can obtain $\rho\in V_{\mathcal{C}(M)}$, which represents an infinitesimal motion of the vectors of $\pi$, and satisfies $\rho_{l_{1}}\neq \rho_{l_{2}}$ for any $l_{1}\neq l_{2}\in R_{\pi}$.
Continuing this procedure, after a finite number of steps, we obtain a collection of vectors $\gamma'\in V_{\mathcal{C}(M)}$ that satisfies ${\gamma'}_{l_{1}}\neq {\gamma'}_{l_{2}}$ for any $l_{1}\neq l_{2}\in R_{\gamma'}$, which completes the proof. 
\end{proof}

\subsection{Lifting polynomials}
Now, we %will
prove the main theorem of this section, providing a generating set,  up to radical, for $I_{M}$.

\begin{theorem}\label{theo lift}
Let $M$ be an $n$-paving matroid with no points of degree greater than two. Then: $$I_{M}=\sqrt{I_{\mathcal{C}(M)}+I_{M}^{\lift}}$$
\end{theorem}

\begin{proof}
We denote the ideal on the right-hand side by $I$. By Theorem~\ref{sub}, we have that $I\subset I_{M}$. To establish the other inclusion, we prove that $V(I)\subset V(I_{M})=\overline{\Gamma_{M}}$. 

Consider $\gamma$, a collection of vectors indexed on the points of $M$ and contained in $V(I)$, in particular, $\gamma\in V_{\mathcal{C}(M)}$.  We will prove that $\gamma\in V_{M}$. Applying Remark~\ref{lif}, and selecting a point $q\in \mathbb{C}^{n}$ outside $\gamma$, we know that $\gamma$ is fully liftable from the vector $q$. By Proposition~\ref{propo 2}, we can apply an infinitesimal motion to $\gamma$, yielding a collection of vectors $\tau \in V_{\mathcal{C}(M)}$ with ${c(M,\tau)}=0$. Finally, by applying Proposition~\ref{coincide}, we can infinitesimally perturb $\tau$ to obtain a collection of vectors in $\Gamma_{M}$.  Hence, $\gamma \in \closure{\Gamma_{M}}=V_{M}$, as desired.
\end{proof}

\begin{remark}\label{obs 4}
Note that Theorem~\ref{theo lift} provides a generating set for $I_{M}$,  up to radical, thereby generalizing Theorems~4.8 and~5.7 from \cite{Fatemeh3}. %obtained by taking the union of the finite set of generators of $I_{\mathcal{C}(M)}$ and the finite set of generators of $L_{q}(M)$ for every $q\in \mathbb{C}^{n}$.  In particular, this 
However, the computed generating set %of $I_{M}$
is not finite. Therefore, in the following sections, we focus on finding a finite generating set for $I_{M}$,  up to radical. Nevertheless, Theorem~\ref{theo lift} will remain essential in achieving this goal.Moreover, we would like to emphasize that we have neither empirical nor theoretical reasons to believe that the radical could be omitted in this case.
\end{remark}

\begin{proposition}\label{real 2}
Let $M$ be an $n$-paving matroid with no points of degree greater than two. % with no points of degree greater than two. 
Then $M$ is realizable in $\CC^{n}$.
\end{proposition}

\begin{proof}
The collection of zero vectors indexed  by $\mathcal{P}$ satisfies the hypothesis of Proposition~\ref{coincide}. Therefore, applying this proposition implies that there exists a realization of $M$.
\end{proof}

Using this proposition, we observe that as a particular case, every point-line configuration without points of degree greater than two is realizable. This has been proven in \textup{\cite[Proposition 5.7]{clarke2021matroid}} for forest-type point-line configurations. Thus, Proposition~\ref{real 2} extends the family of known realizable matroids. 

\subsection{Decomposing circuit varieties}

Now, we present a decomposition of the circuit variety. But before that, we %will 
prove a lemma. 

\begin{definition}
Recall the notion of $\overline{R_{\gamma}}$ from Definition~\ref{pav}.
A collection of vectors $\gamma \in V_{\mathcal{C}(M)}$ is called \textit{inflexible} if,  for each subset $R$ of the partition $\overline{R_{\gamma}}$ %$R\in \overline{R_{\gamma}}$ 
with $\size{R}>1$, the collection of vectors $\{\gamma_{p}:p\in  \mathcal{P}_{M^{R}}\}$ cannot be lifted in a non-degenerate way to a collection of vectors $\widetilde{\gamma}\in V_{\mathcal{C} (M^{R}})$ from a  vector outside $\gamma$. Otherwise, we say that $\gamma$ is {\em flexible}.
%, where $N$ denotes the submatroid of hyperplanes of $R$.
\end{definition}

 For the following lemma, recall Definition~\ref{lif num}.

\begin{lemma}\label{lem_never_liftable}
Let $M$ be an $n$-paving matroid % with no points of degree greater than two, 
and let $\gamma$ be a collection of vectors in $V_{\mathcal{C}(M)}$ that is flexible. Then, there exists a collection of vectors in $V_{\mathcal{C}(M)}$ that is an infinitesimal motion of $\gamma$ and has smaller lifting number. Furthermore, we can apply an infinitesimal motion to $\gamma$ to obtain an inflexible collection in $V_{\mathcal{C}(M)}$.
%Then, there exists a collection of vectors $\kappa\in V_{\mathcal{C}(M)}$ that is an infinitesimal motion of $\gamma$ and satisfies $M_{\kappa}<M_{\gamma}$. Moreover, we can apply an infinitesimal motion to $\gamma$ to obtain an \textcolor{blue}{inflexible} $\kappa\in V_{\mathcal{C}(M)}$.
%that is  \textcolor{blue}{inflexible}.
\end{lemma}

\begin{proof}
The first claim follows using the same argument as in Propositions \ref{propo} and \ref{propo 2}. The second claim follows by applying the first procedure repeatedly starting with $\gamma$.
\end{proof}

With this, we have the following proposition, in which we denote $V_{\gamma}$ for the matroid variety corresponding to the matroid defined by the vectors of $\gamma$. 

\begin{proposition}\label{dec}
Let $M$ be an $n$-paving matroid %with no points of degree greater than two 
and let $U(M)$ be the set of all inflexible $\gamma$'s for $M$. Then we have the following decomposition:
$$V_{\mathcal{C}(M)}=\bigcup_{\gamma\in U(M)}V_{\gamma}$$
\end{proposition}

\begin{proof}
The inclusion $\supset$ is clear. The other inclusion follows from Lemma~\ref{lem_never_liftable}. %Corollary~\ref{nev}.
\end{proof}

Now we %will 
apply this decomposition to two particular cases. 

\begin{proposition}\label{ico}
Let $M$ be an $n$-paving matroid. % with no points of degree greater than two. 
Suppose that any submatroid of hyperplanes of $M$ is liftable, then we have $V_{\mathcal{C}(M)}=V_{M}.$
\end{proposition}

\begin{proof}
In this case, the only inflexible collections of vectors $\gamma$ are those with ${c(M,\gamma)}=0$. However, by applying Proposition~\ref{coincide}, we know that any of these collections belong to $V_{M}$. Then, using Proposition~\ref{dec}, we obtain the desired result.
\end{proof}
\noindent This proposition generalizes \cite[Theorem~3.17]{Fatemeh3} from dimension 3 to higher dimensions.~See~Remark~\ref{different definition}. %  Refer to Remark~\ref{different definition} to avoid any confusion.}

\begin{example}
 Consider the point-line configuration depicted in Figure~\ref{new figure} (Left). 
As noted in Example~\ref{example 3 grid}, this configuration is liftable. Moreover, by Lemma~\ref{lifts}, all of its submatroids are liftable. Therefore, it satisfies the conditions of Proposition~\ref{ico}, implying that its circuit and matroid varieties coincide. 
Furthermore, this also holds for forest-type configurations, as discussed in Example~\ref{example 3 grid}. %Since any submatroid of a forest-type configuration is also a forest, we obtain by Proposition~\ref{ico} that the circuit and matroid varieties coincide for forest-type configurations.  
%every forest point-line configuration with no points of degree greater than two satisfies the conditions of Proposition~\ref{ico}.
\end{example}

\begin{lemma}\normalfont\label{less n-2}
Let $M$ be an $n$-paving matroid with no points of degree greater than two such that $\size{S_{l}}\leq n-2$ for every $l\in \mathcal{L}$. Then $V_{\mathcal{C}(M)}=V_{M}$. 
\end{lemma}
\begin{proof}
By Lemma~\ref{lift n-2}, every collection of vectors $\gamma \in V_{\mathcal{C}(M)}$ is fully liftable from any  vector outside it. Consequently, the only {inflexible} collections of vectors $\gamma$ are those with ${c(M,\gamma)}=0$. Then, applying the same argument as in Proposition~\ref{ico}, we obtain the desired result.
\end{proof}

\begin{example}
 Let $M$ be the $4$-paving matroid on $\mathcal{P}=\{p_{1},\ldots,p_{9}\}$ with the set of dependent hyperplanes \[\mathcal{L}=\{\{p_{1},p_{2},p_{3},p_{4}\},\{p_{4},p_{5},p_{6},p_{7}\},\{p_{1},p_{5},p_{8},p_{9}\}\}.\] Note that $S_{M}=\{p_{1},p_{4},p_{5}\}$, which implies that $\size{S_{l}}\leq 2$ for every $l\in \mathcal{L}$. Consequently, by applying Lemma~\ref{less n-2}, we find that $V_{\mathcal{C}(M)}=V_{M}$.
\end{example}

%\textcolor{red}{I changed our notation of the matroid variety of the uniform matroid $U_{n-1,\size{M}}$.}
\begin{proposition}\label{des 2}
Let $M$ be an $n$-paving matroid. % with no points of degree greater than two. 
Suppose that all the proper submatroids of hyperplanes of $M$ are liftable. Then we have 
\[ V_{\mathcal{C}(M)}=V_{M}\cup V_{U_{n-1,\size{M}}},\] 
%where $V_{0}$ is the matroid variety of the uniform matroid $U_{n-1,\size{M}}$. % of Definition~\ref{uniform 3}.
%of rank $n-1$ on $n$ elements. 
where $U_{n-1,\size{M}}$ is the uniform matroid. % from Definition~\ref{uniform 3}. 
Furthermore, $I_{M}$ equals to the radical of the ideal generated by $I_{\mathcal{C}(M)}$ and the $(\size{M}-n+1)$-minors of the liftability matrices $\mathcal{M}_{q}(M)$, with $q$ varying over $\CC^{n}$. 
\end{proposition}
\begin{proof}
%In this case, 
The only {inflexible} collections of vectors $\gamma\in V_{\mathcal{C}(M)}$ are those where ${c(M,\gamma)}=0$, or where all vectors belong to the same hyperplane.
By Proposition~\ref{coincide}, any collection of the first type is in $V_{M}$,  while any collection of the second type belongs to $V_{U_{n-1,\size{M}}}$. Consequently, the desired decomposition follows by Proposition~\ref{dec}.
The second statement is a direct consequence of Theorem~\ref{theo lift}, using the fact that every proper submatroid of hyperplanes is inherently liftable.
\end{proof}
Proposition~\ref{des 2} is a generalization of Corollary 3.20 and Proposition 3.22 in \cite{Fatemeh3},  from dimension three to higher dimensions. (See also Remark~\ref{different definition}). % Refer to Remark~\ref{different definition} to avoid any confusion.} %As highlighted in Lemma~\ref{lifts}, the matroid in Figure~\ref{fig:combined}~(Center) and Figure~\ref{fig:combined 2}~(Left) fall under this category.

\begin{example}\label{quad}
 By applying Lemma~\ref{lifts}, we find that all proper submatroids of the matroids $QS$ and $G_{3\times 4}$, depicted
in Figure~\ref{fig:combined}~(Center) and Figure~\ref{fig:combined 2}~(Left), respectively, are liftable. 
%Consider the $3$-paving matroid $QS$ in Figure~\ref{fig:combined}~(Center). All points in this matroid have degree two, and all proper submatroids are liftable, by Lemma~\ref{lifts}. By Lemma~\ref{lifts}, this also applies to the matroid $G_{3\times 4}$ in Figure~\ref{fig:combined 2}~(Left).
%Let $V_{0}$ and $V'_0$ be the matroid varieties corresponding to the uniform matroids of rank $2$, on $6$ and $12$ elements, respectively. 
Then, by Proposition~\ref{des 2}, we have that:\quad
$V_{\mathcal{C}(QS)}=V_{QS}\cup V_{U_{2,6}}\quad\text{and}\quad V_{\mathcal{C}(G_{3\times 4})}=V_{G_{3\times 4}}\cup V_{U_{2,12}}.$
\end{example}

\section{Computing graph polynomials
}\label{sec 5}
Here, we introduce a combinatorial approach for generating polynomials in $I_{M}$; see Theorem~\ref{thm ci}. Applying this theorem, we proceed in Section~\ref{rel 6} to derive a finite generating set for $I_{M}$,  up to radical.  Throughout this section, we use Notation~\ref{initial notation}.

%\subsection{Polynomials arising from graphs} 

%\textcolor{red}{I made some changes in this section to make Theorem~\ref{thm ci} more explicit.}

% We start by fixing our notation.}

\begin{notation}\label{notation graph G}
 Let $G$ be a loopless directed graph, with vertices $V(G) = \{p_i : i \in [n]\}$ and directed edges $E(G)$. Each edge $(p_{i},p_{j})\in E(G)$ is assigned a weight $\alpha_{i,j}$, an indeterminate.
The associated $n\times n$ matrix $M(G)$  %of indeterminates
 is defined as
\begin{equation*}%\label{matrix 17}
M(G)_{ij}=
\begin{cases}
1, \qquad \text{if $i=j,$}\\ 
-\alpha_{i,j}, \ \text{if $(p_{i},p_{j})\in E(G),$}\\
0,  \qquad  \text{otherwise.}
\end{cases}
\end{equation*}
%We interpret the edge $(p_i, p_j)$ as directed \textit{from} $p_i$ \textit{to} $p_j$. %We refer to $\alpha_{i,j}$ as the \textit{weight} of the edge $(p_i, p_j)$.
A \textit{cycle}  of $G$ is a sequence of vertices $w = (v_0, v_1, \dots, v_s)$ such that $(v_{i-1}, v_i) \in E(G)$ is an edge of $G$ for all $i \in [s]$, and  $v_0, \dots, v_{s-1}$ are distinct and $v_s = v_0$. The  edges of $w$ are denoted by $E(w) = \{(v_{i-1}, v_i) : i \in [s] \}$. We define 
\[
P(w) = \prod_{(p_i, p_j) \in E(W)} \alpha_{i,j},
\]
which is the product of weights of edges in $w$. A pair of cycles $\mathcal K$ and $\mathcal L$ are \textit{disjoint} if they do not share any vertices.
We denote by 
\[
C(G) = \{\{\mathcal K_1, \dots, \mathcal K_m\} : m \ge 1, \, \mathcal K_i \text{ and } \mathcal K_j \text{ are disjoint cycles for all } i,j \in [m]  \text{ with } i \neq j\}
\]
the set of all non-empty collections of disjoint cycles. 
For each $K\in C(G)$, we define 
$$P(K)=\prod_{\mathcal{K}\in K}P(\mathcal{K}).$$
\end{notation}

\begin{lemma}\label{lemma: first way}
 With the notation above, $\det(M(G))$ has the following expression:
\begin{equation}
\label{eq r}
   \det(M(G)) = 1 +\sum_{K \in C(G)} (-1)^{|K|} P (K).
\end{equation}
\end{lemma}
\begin{proof}
 Observe that $M(G)=\Id_{n}-A(G)$, where $A(G)$ denotes the weighted adjacency matrix of the graph $G$. From this, we deduce that $\det(M(G))=p_{A(G)}(1)$, where $p_{A(G)}$ is the characteristic polynomial of $A(G)$. By evaluating $p_{A(G)}$ at $1$, using the expression of this polynomial provided in  \textup{\cite[\S2]{cavers2012skew}}, we obtain the desired expression for $\det(M(G))$.
\end{proof}

\begin{theorem}\label{thm: first way}
 Let $p_1, \dots , p_n$ be a collection of non-zero %linearly dependent 
vectors in $\CC^r$. Assume that for each $i \in [n]$, there exists a linear dependency of the form:
\begin{equation}\label{9}
p_i =
\sum_{\substack{j \in [n]\\ j \neq i}} \beta_{i,j} p_j,
\end{equation}
where $\beta_{i,j} \in \CC$. %are scalar coefficients.
Let $G$ be the directed graph with vertices $V(G)=\{p_{i}:i\in [n]\}$ and directed edges $E(G)=\{(p_{i},p_{j}):\beta_{i,j}\neq 0\}$. Then, %$\det(M(G))$ 
the polynomial in~\eqref{eq r} vanishes when the variables $\alpha_{i,j}$ are evaluated as $\beta_{i,j}$ $\pare{\text{where}\ (p_{i},p_{j})\in E(G)}$.
\end{theorem}

\begin{proof}
 Let $M$ be the $n\times n$ matrix obtained from $M(G)$ by substituting each variable $\alpha_{i,j}$ with $\beta_{i,j}$, where $(p_{i},p_{j})\in E(G)$ . By Lemma~\ref{lemma: first way}, proving the statement reduces to showing that $\det(M)=0$.
%We construct 
Consider the matrix $P$ of size $n\times r$, where the rows %of $P$ 
correspond to the vectors ${p_{1},\ldots,p_{n}}$. 
Using Equation~\eqref{9} we deduce that $MP=0$, which implies that $M$ is not invertible (since $P\neq 0$) and its determinant is zero.
\end{proof}

%\textcolor{red}{I removed the definition of ''forest type'' since the obtained graph is not a forest and it could also create confusion with the family of forest-type configurations. I also moved this example from Section~\ref{sec 7} to Section~\ref{sec 5}.}

\begin{example}%[Forest type matroid]
 With the notation of Theorem~\ref{thm: first way}, suppose there is a vertex $p_{1}\in V(G)$ such that $p_{1}$ is contained in all cycles of $G$. Consequently, any two cycles of $G$ are not disjoint, as they both contain $p_{1}$. Thus, by Theorem~\ref{thm: first way}, we obtain:
\begin{equation*}
1 = \sum_{K \ \text{cycle}} P(K),
\end{equation*}
where these products are taken over the coefficients $\beta_{i,j}$. A particular instance of this occurs when the graph $G$ is a cycle, in which case we have $1 = P(G)$.
\end{example}

\begin{example}
 Let $\{p_{1},p_{2},p_{3},p_4,p_5\}\subset \CC^{r}$ be a set of non-zero vectors. Suppose that 
\begin{equation*}
p_1=\beta_{1,2}p_2+\beta_{1,3}p_3,\quad p_2=\beta_{2,4}p_4,\quad p_3=\beta_{3,5}p_5,\quad 
p_4=\beta_{4,1}p_1+\beta_{4,5}p_5,\quad \text{and} \quad   p_5=\beta_{5,1}p_1+\beta_{5,4}p_4,
\end{equation*}
with $\beta_{i,j}\in \CC$. Let $G$ be the graph constructed from these linear dependencies, as above. We have $V(G)=\{p_{1},p_{2},p_{3},p_4,p_5\}$ and \[E(G)=\{(p_1,p_2),(p_1,p_3),(p_2,p_4),(p_3,p_5),(p_4,p_1),(p_4,p_5),(p_5,p_1),(p_5,p_4)\}.\]
The cycles of $G$ are $(p_1,p_2,p_4),(p_1,p_2,p_4,p_5),(p_1,p_3,p_5),(p_1,p_3,p_5,p_4)$. By~Theorem~\ref{thm: first way},~we obtain:
\[ \beta_{1,2}\beta_{2,4}\beta_{4,1}+\beta_{1,3}\beta_{3,5}\beta_{5,1}+\beta_{1,2}\beta_{2,4}\beta_{4,5}\beta_{5,1}+\beta_{1,3}\beta_{3,5}\beta_{5,4}\beta_{4,1}=1.\]
\end{example}

Theorem~\ref{thm: first way} provides a method to construct polynomials that vanish in the coefficients of the linear dependencies of a given collection of vectors. To generate polynomials inside $I_{M}$, we need one extra step, the following remark, that connects Theorem~\ref{thm: first way} with matroids and their realization spaces.

\begin{remark} \label{rem 5}
Let $M$ be a matroid of rank $n$,  $\set{p_{1},\ldots,p_{m}}\in \mathcal{C}(M)$ a circuit of $M$  with $m\leq n$, and $\gamma \in \Gamma_{M}$ a realization of $M$ in $\CC^{n}$. We know that there exists coefficients $\alpha_{j},2\leq j\leq m$ such that
\begin{equation}\label{suma}
\gamma_{p_{1}}=\sum_{j=2}^{m}\alpha_{j}\gamma_{p_{j}}
\end{equation}
Then for any collection of $n-m+1$ vectors  $\set{q_{1},\ldots,q_{n-m+1}}\subset \CC^{n}$  and any $2\leq i\leq m$, we have %the equality 
\begin{eqnarray*}
\corch{\gamma_{p_{1}},\ldots, \gamma_{p_{i-1}},\hat{\gamma}_{p_{i}},\ldots,\gamma_{p_{m}},q_{1},\ldots,q_{n-m+1}} & = & 
\sum_{j=2}^{m}\alpha_{j}\corch{\gamma_{p_{j}},\gamma_{p_{2}},\ldots, \gamma_{p_{i-1}},\hat{\gamma}_{p_{i}},\ldots,\gamma_{p_{m}},q_{1},\ldots,q_{n-m+1}}\\ & = & 
\alpha_{i} (-1)^{i}\corch{\hat{\gamma}_{p_{1}},\ldots, \gamma_{p_{i-1}},\gamma_{p_{i}},\ldots,\gamma_{p_{m}},q_{1},\ldots,q_{n-m+1}}
\end{eqnarray*}
The second equality arises from expanding $\gamma_{p_{1}}$ using~\eqref{suma}, and leads to the following expressing of $\alpha_{i}$
 \begin{equation}
 \alpha_{i}= (-1)^{i}\frac{\corch{\gamma_{p_{1}},\ldots, \gamma_{p_{i-1}},\hat{\gamma_{p_{i}}},\ldots,\gamma_{p_{m}},q_{1},\ldots,q_{n-m+1}}}{ \corch{\hat{\gamma_{p_{1}}},\ldots, \gamma_{p_{i-1}},\gamma_{p_{i}},\ldots,\gamma_{p_{m}},q_{1},\ldots,q_{n-m+1}}}.\label{7}
 \end{equation}
  Note that Equation~\eqref{7} is obtained from evaluating the quotient of polynomials
\begin{equation}\label{29}
 (-1)^{i}\frac{\corch{p_{1},\ldots, p_{i-1},\hat{p_{i}},\ldots,p_{m},q_{1},\ldots,q_{n-m+1}}}{\corch{\hat{p_{1}},\ldots, p_{i-1},p_{i},\ldots,p_{m},q_{1},\ldots,q_{n-m+1}}}
\end{equation}
at the vectors of $\gamma$, i.e.~substituting the vector of $n$ indeterminates associated with each $p\in \mathcal{P}$ by the vector $\gamma_{p}$.
 If $m=n$, then the above quotient is 
\begin{equation}\label{8}
 \alpha_{i}= (-1)^{i}\frac{\corch{\gamma_{p_{1}},\ldots, \gamma_{p_{i-1}},\hat{\gamma_{p_{i}}},\ldots,\gamma_{p_{n}},q}}{\corch{\hat{\gamma_{p_{1}}},\ldots, \gamma_{p_{i-1}},\gamma_{p_{i}},\ldots,\gamma_{p_{n}},q}},
\end{equation}
for any vector $q\in \CC^{n}$. This is always the case if $M$ is an $n$-paving matroid.  Note that Equation~\eqref{8} is obtained from evaluating the quotient of polynomials
\begin{equation}\label{19}
 (-1)^{i}\frac{\corch{p_{1},\ldots, p_{i-1},\hat{p_{i}},\ldots,p_{n},q}}{\corch{\hat{p_{1}},\ldots, p_{i-1},p_{i},\ldots,p_{n},q}}
\end{equation}
at the vectors of $\gamma$. %, i.e.~substituting the vector of $n$ indeterminates associated with each $p\in \mathcal{P}$ by the vector $\gamma_{p}$}.
Note that these equations hold for every choice of vectors $\{q_{1},\ldots,q_{n-m+1}\} \subset \mathbb{C}^{n}$. We will commonly refer to these vectors as the \textit{extra vectors}.
\end{remark}

\noindent Remark~\ref{rem 5} shows how to express the coefficients of the linear dependencies of the vectors $\{\gamma_{p} : p \in \mathcal{P}\}$ as a quotient of brackets for any $\gamma \in \Gamma_{M}$. Before stating our main result, we establish our notation.

\begin{notation}\label{notacion grafo G}  Let $M$ be an $n$-paving matroid. %, and let $\gamma \in \Gamma_{M}$ be a realization of $M$ in $\mathbb{C}^{n}$. 
Consider subsets $J$ and $P=\{p_{1},\ldots,p_{k}\}$ of $\mathcal{P}$ such that $P\cap \overline{J}=\emptyset$. Let $C=\{c_{1},\ldots,c_{k}\}$ be a set of circuits of size $n$ 
included in $P\cup J$, with $p_{i}\in c_{i}$ for all $i\in [k]$. We represent the elements of $c_{i}$ as $$c_{i}=\{p_{i_{1}},\ldots,p_{i_{a_{i}}},%q_{i_{1}},\ldots,q_{i_{n-a_{i}}}\}$$
r_{i,1},\ldots,r_{i,n-a_{i}}\},$$ where $p_{i_{1}}=p_{i}$ and $%r_{i_{1}},\ldots,q_{i_{n-a_{i}}}\}
\{r_{i,1},\ldots,r_{i,n-a_{i}}\}\subset J$. %For each $i\in [k]$, the set of vectors $\{\gamma_{p}:p\in c_{i}\}$ forms a circuit. Hence, there exist coefficients $\{\alpha_{i,j}:j\in [k],j\neq i\}$ such that
%\begin{equation}\label{ci}
%\gamma_{p_{i}}=\sum_{\substack{j \in [k]\\ j \neq i}}\alpha_{i,j}\gamma_{p_{j}}+\text{some combination of vectors in }  \{\gamma_{p}:p\in J\}},
%\end{equation}
%where $\alpha_{i,j}=0$ if $p_{j}\notin c_{i}$.
%We represent the elements of $c_{i}$ as $$c_{i}=\{p_{i_{1}},\ldots,p_{i_{a_{i}}},%q_{i_{1}},\ldots,q_{i_{n-a_{i}}}\}$$
% r_{i,1},\ldots,r_{i,n-a_{i}}\}}$$ where $p_{i_{1}}=p_{i}$ and $%r_{i_{1}},\ldots,q_{i_{n-a_{i}}}\}
% \{r_{i,1},\ldots,r_{i,n-a_{i}}}\}\subset J$.
We construct the directed weighted graph $G$, where the vertex set is $V(G)=P$ and the edge set is $E(G)=\{(p_{i},p_{j}):p_{j}\in c_{i}\}$.
%:\alpha_{i,j}\neq 0\}$
Each edge $(p_{i},p_{j})\in E(G)$ is assigned a weight $\alpha_{i,j}$, an indeterminate.
From Lemma~\ref{lemma: first way}, $\det(M(G))$ has the following expression:
\begin{equation}\label{matrix 18}
1 +\sum_{K \in C(G)} (-1)^{|K|} P (K)\in \CC[\alpha_{i,j}]_{(p_{i},p_{j})\in E(G)}.
\end{equation}
\end{notation}
%\vspace{1cm}

\begin{theorem}\label{thm ci}
% Let $N$ be the matrix obtained from $M(G)$ by replacing each variable $\alpha_{i,j}$ $\pare{\text{where}\ p_{j}\in c_{i}}$ with the polynomial
%\[\pare{-1}^{s}\corch{p_{i},p_{i_{2}},\ldots, \hat{p_{j}},r_{i,1},\ldots,r_{i,n-a_{i}},q_{p_{i}}},\]
%where $p_{j}=p_{i_{s}}$. Then, $\det(N)$ belongs to $I_{M}$ for any selection of \textit{extra vectors} $\set{q_{p}:p\in P}\subset \mathbb{C}^{n}$.}
 The polynomial derived from~\eqref{matrix 18} %$\det(M(G))$ by 
by substituting the variables $\alpha_{i,j}$ $\pare{\text{where}\ p_{j}\in c_{i}}$ 
 with the bracket quotients %from \eqref{19},
\[\frac{\pare{-1}^{s-1}\corch{p_{i},p_{i_{2}},\ldots, \hat{p_{j}},r_{i,1},\ldots,r_{i,n-a_{i}},q_{p_{i}}}}{\corch{\hat{p_{i}},p_{i_{2}},\ldots, p_{j},r_{i,1},\ldots,r_{i,n-a_{i}},q_{p_{i}}}},\] 
where $p_{j}=p_{i_{s}}\in c_{i}$, and subsequently multiplying by the denominators, belongs to $I_{M}$ for any choice of \textit{extra vectors} $\set{q_{p}:p\in P}\subset \mathbb{C}^{n}$.
%applying Theorem~\ref{thm: first way},  % using any extra vectors, 
%and subsequently multiplying by the denominators, we obtain a polynomial inside $I_{M}$.
\end{theorem}

\begin{proof}
 Let $\gamma\in \Gamma_{M}$ be any realization of $M$ in $\CC^{n}$. For each $i\in [k]$, the set of vectors $\{\gamma_{p}:p\in c_{i}\}$ forms a circuit. Hence, there exist coefficients $\{\beta_{i,j}:j\in [k],j\neq i\}$ such that
\begin{equation}\label{ci}
\gamma_{p_{i}}=\sum_{\substack{i\neq j \in [k] %\\ j \neq i
}}\beta_{i,j}\gamma_{p_{j}}+\text{some combination of vectors in }  \{\gamma_{p}:p\in J\}.
\end{equation}
Consider the subspace $J'=\text{span}\{\gamma_{p}: p \in J\}$. Since $P \cap \overline{J} = \emptyset$, the vectors $\{\gamma_{p}: p \in P\}$ do not belong to $J'$. Equivalently, none of the vectors $\{\gamma_{p}: p \in P\}$ is the zero vector in the vector space $V=\mathbb{C}^{n}/J'$.
Then, in $V$, Equation~\eqref{ci} can be expressed as:
\begin{equation}
    \gamma_{p_{i}} = \sum_{\substack{i\neq j \in [k] %\\ j \neq i
    }}\beta_{i,j}\gamma_{p_{j}}.
\end{equation} 
Consequently, by applying Theorem~\ref{thm: first way} in the vector space $V$ to the vectors $\{\gamma_{p}:p\in P\}$, we obtain that $\det(M(G))$ vanishes when each variable $\alpha_{i,j}$ $\pare{\text{where}\ p_{j}\in c_{i}}$ is evaluated as $\beta_{i,j}$.
By substituting the coefficients $\beta_{i,j}$ with their expressions as quotients of brackets in~\eqref{8}, and subsequently clearing the denominators, we obtain that the polynomial of the statement vanishes at the vectors of $\gamma$. Note that Lemma~\ref{lemma: first way} is also being applied. Since $\gamma$ is an arbitrary element of $\Gamma_{M}$, it follows that this polynomial belongs to $I_{M}$.
\end{proof}

\begin{remark}\label{for any matroid}
\begin{itemize}
    \item[(i)]  Observe that one may be concerned about the potential vanishing of the denominators in Equations~\eqref{7} and~\eqref{8}. However, since Theorem~\ref{thm ci} involves multiplying by the denominators, this does not yield any problem.

    \item[(ii)]  Observe that Theorem~\ref{thm ci} applies to any matroid by replacing the variables $\alpha_{i,j}$ $\pare{\text{where}\ p_{j}\in c_{i}}$ 
 with the bracket quotients from \eqref{29}, for any choice of \textit{extra vectors}. However, for clarity, we have stated the theorem only in the context of paving matroids.
\end{itemize}
\end{remark}

\noindent See Section~\ref{sec 7} for examples of applying Theorem~\ref{thm ci} to generate polynomials within the matroid ideal. % $I_{M}$.

%\vspace{-5mm}

\begin{definition}[Graph ideal]\label{cl}
Let $M$ be an $n$-paving matroid. The polynomials constructed in Theorem~\ref{thm ci} are called the graph polynomials of $M$. We define $I_{M}^{\gr}$ as the ideal generated by all graph polynomials, where any choice of sets $J$, $P$, and $C$ satisfying the hypotheses of the theorem, along with any selection of additional vectors $\set{q_{p}:p\in P}\subset \mathbb{C}^{n}$, are considered.
\end{definition}

The next remark is fundamental for finding a finite generating set,  up to radical, for $I_{M}$. 

\begin{remark}\label{free}
Let $M$ be an $n$-paving matroid. We note that for any choice of sets $J$, $P$, and $C$ satisfying the hypothesis of Theorem~\ref{thm ci}, the polynomial constructed has the property that each additional vector $\{q_{p}:p\in P\}\subset \CC^{n}$ appears exactly once within each %term
 monomial on the brackets.  Therefore, using the multilinearity of the determinant, a finite generating set of $I_{M}^{\gr}$  can be obtained by choosing the extra vectors $\set{q_{p}:p\in P}\subset \CC^{n}$ as vectors of the canonical basis $E$ of $\mathbb{C}^{n}$.

%if we choose the extra vectors  $\set{q_{p}:p\in P}\subset \CC^{n}$ as elements of the canonical basis $E=\set{e_{1},\ldots,e_{n}}$ of $\CC^{n}$, when constructing the polynomials in Theorem~\ref{thm ci}, then by the multilinearity of the determinant we see that these polynomials generate the ideal $I_{M}^{\gr}$. 
%Thus, a finite generating set of $I_{M}^{\gr}$, up to radical, can be obtained by choosing the extra vectors from the canonical basis $E$ of $\mathbb{C}^{n}$.
\end{remark}

\section{Relation between lifting polynomials and graph polynomials %constructed in the first and second way.
}\label{rel 6}
In this section, focusing on $n$-paving matroids, we describe the relationship between the lifting polynomials introduced in Section~\ref{constructing} and the graph polynomials established in Section~\ref{sec 5}. Specifically, we will demonstrate the inclusion $I_{M}^{\lift}\subset I_{M}^{\gr}$, as shown in Theorem~\ref{inc}. Using this ideal containment, we proceed to present a finite %generating set for $I_{M}$,
set of defining equations for $V_{M}$, where $M$ represents an $n$-paving matroid devoid of vertices of degree exceeding two;  refer to Corollary~\ref{fin} and Remark~\ref{fin rel}. Throughout this section, we use the notation from Section~\ref{sec 5}, and Notation~\ref{initial notation}.
%\\[0.7 ex]

%\textcolor{red}{I changed the notation of $q_{i_{j}}\subset J$ to $r_{i,j}$ since it could have been confusing because we denoted the extra vectors with $q$. The subscripts $q_{i_{j}}$ were also confusing since the points $p_{i_{j}}$ have the same subscripts.}

\medskip
We begin by establishing a lemma crucial for the proof of Theorem~\ref{inc}. In this lemma, we adopt the same context as in  Notation~\ref{notacion grafo G}, %Theorem~\ref{thm ci}, 
where $M$ represents an $n$-paving matroid, $P=\{p_{1},p_{2},\ldots,p_{k}\}$ denotes a set of points lying outside the closure of $J$, and $C=\{c_{1},\ldots,c_{k}\}$ denotes a set of circuits  of size $n$ contained within $P\cup J$, with $p_{i}\in c_{i}$ for each $i\in [k]$. We represent the elements of $c_{i}$ as $$c_{i}=\{p_{i_{1}},\ldots,p_{i_{a_{i}}},%q_{i_{1}},\ldots,q_{i_{n-a_{i}}}\}$$
 r_{i,1},\ldots,r_{i,n-a_{i}}\}$$ where $p_{i_{1}}=p_{i}$ and $%r_{i_{1}},\ldots,q_{i_{n-a_{i}}}\}
 \{r_{i,1},\ldots,r_{i,n-a_{i}}\}\subset J$. We also assume that the circuits within $C$ are distinct.

\begin{lemma}\label{det}
Let $q\in\mathbb{C}^{n}$. Consider the $k\times k$ submatrix $N$ of $\mathcal{M}_{q}(M)$ with the columns corresponding to the points of $P$ and the rows corresponding to the circuits of $C$. Then the determinant of this submatrix coincides with the polynomial generated in Theorem~\ref{thm ci}  for $P,J$, and $C$,
with all extra vectors being $q$.
\end{lemma}

\begin{proof}
The $i^{\rm th}$ diagonal entry of $N$  (corresponding to $c_{i}$ and $p_{i}$), for every $i$, is given by 
$$N_{ii}=\corch{\hat{p_{i}},p_{i_{2}},\ldots,p_{i_{a_{i}}}, r_{i,1},\ldots,r_{i,n-a_{i}}},q.$$ 
For every $i\in [k]$, we divide all the entries of the $i^{\text{th}}$ %column
 row of $N$  (corresponding to $c_{i}$) by this %number
 polynomial, and we denote the resulting matrix as $L$. We easily obtain
\begin{equation}\label{mult}
\det(N)=\prod_{i=1}^{k}\corch{\hat{p_{i}},p_{i_{2}},\ldots,p_{i_{a_{i}}}, r_{i,1},\ldots,r_{i,n-a_{i}}},q\det(L)
\end{equation}
%Using Equation~\eqref{8}, 
We can express the entries of $L$ as follows:
$$L_{ij}=\begin{cases}
1 & \text{if } i=j,\\
0 & \text{if } p_{j}\notin c_{i}, \\
\frac{\pare{-1}^{s-1}\corch{p_{i},p_{i_{2}},\ldots, \hat{p_{j}},r_{i,1},\ldots,r_{i,n-a_{i}},q}}{\corch{\hat{p_{i}},p_{i_{2}},\ldots, p_{j},r_{i,1},\ldots,r_{i,n-a_{i}},q}}%-\alpha_{i,j} 
& \text{if } p_{j}=p_{i_{s}}\in c_{i},
\end{cases}$$
%where the number $\alpha_{i,j}$ represents the weight of the edge $(p_{i},p_{j})\in E(G)$ in the graph $G$, as defined above  in} Theorem~\ref{thm ci}. 
 Note that $\det(L)$ is obtained from $\det(M(G))$ by replacing the variables $\alpha_{i,j}$ $\pare{\text{where}\ p_{j}\in c_{i}}$ 
 with the bracket quotients from \eqref{19}, with all extra vectors being $q$. %Thus, $\det(L)$ shares the same expression as Equation~\eqref{eq r} in Theorem~\ref{thm ci}. 
Finally, multiplying $\det(L)$ by the %expression
product of brackets in Equation~\eqref{mult} corresponds to multiplying by the denominators of the expressions in Equation~\eqref{19}. Therefore, $\det(N)$ matches the polynomial in Theorem~\ref{thm ci}.
\end{proof}

\begin{remark}\label{is zero}
Examining the proof of Lemma~\ref{det}, we infer that if two circuits of $C$ were identical, the determinant of the submatrix would be zero. This implies that the polynomial constructed in Theorem~\ref{thm ci} using $C$ would also evaluate to zero.
\end{remark}

Using Lemma~\ref{det}, we have the following Corollary.

\begin{corollary}\label{dependent}
Let $P=\{p_{1},\ldots,p_{k}\}$ and $J$ be subsets of points such that $P\cap \overline{J}=\emptyset$. For any $q\in \mathbb{C}^{n}$, all the minors of the $k\times k$ submatrices of $\mathcal{M}_{q}(M)$ with columns corresponding to the points of $P$ and rows corresponding to $k$ circuits of size $n$ in $J\cup P$ lie in the ideal  $I_{M}^{\gr}$.
\end{corollary}

\begin{proof}
Let $C$ be a collection of $k$ circuits of size $n$  within $J\cup P$. Let $A$ be the corresponding $k\times k$ submatrix of $\mathcal{M}_{q}(M)$ on the rows indexed by $C$. 
To prove that $\det(A)\in  I_{M}^{\gr}$, we consider two cases:

\medskip{\bf Case 1.} Suppose it is possible to choose for each point $p\in P$ a circuit $c_{p}\in C$ such that $p\in c_{p}$ and $c_{p}\neq c_{q}$ for $p\neq q$. In this case, by Lemma~\ref{det}, $\det(A)$ is one of the polynomials generated in Theorem~\ref{thm ci} for $P,J,C$,  with all extra vectors set equal to $q$. Hence, we conclude that $\det(A)\in  I_{M}^{\gr}$.

\medskip{\bf Case 2.} Suppose it is not possible to select for each point of $P$ a distinct circuit from $C$ containing it. Then we can apply the Hall Marriage Theorem \cite[Theorem~1]{hall1948distinct}, which asserts that there exist $r$ points $\{p_{1},p_{2},\ldots,p_{r}\}\subset P$ that belong to at most $r-1$ circuits of $C$, denoted as $\{c_{1},\ldots,c_{r-1}\}$. Consequently, in the $r$ columns of $A$ corresponding to $\{p_{1},p_{2},\ldots, p_{r}\}$, the only non-zero entries occur in the $r-1$ rows corresponding to the circuits $\{c_{1},\ldots,c_{r-1}\}$. As a result, $\det(A)$ is zero, implying it is contained in $ I_{M}^{\gr}$. 
\end{proof}

Recall the notions $I_{M}^{\lift}$ and $I_{M}^{\gr}$ from Definitions~\ref{cl 2} and~\ref{cl}, the ideals generated by the lifting and graph polynomials, respectively.

\begin{theorem}\label{inc}
Let $M$ be an $n$-paving matroid. Then $I_{M}^{\lift}\subset I_{M}^{\gr}.$
\end{theorem}

\begin{proof}
Let $N$ be any submatroid of $M$ of rank $n$. We fix  an arbitrary set $J$ of points in  $\mathcal{P}_{N}$ with $|J|=n-1$ and consider the submatrix $A$ obtained by removing from $\mathcal{M}_{q}(N)$ the columns corresponding to $J$. We will prove that all $(|N|-(n-1))$-minors of $A$ are contained in $I_{M}^{\gr}$.  Since $N$ and $J$ are arbitrary, this will imply the result.

We apply Corollary~\ref{dependent} for the $n$-paving matroid $N$, taking the sets $\overline{J}$ and $P= \mathcal{P}_{N}\setminus \overline{J}$ (with closure taken in $N$), to conclude that the $(\size{N}-\size{\closure{J}})$-minors of the submatrix $B$, obtained by removing the columns corresponding to $\overline{J}$ from the matrix $\mathcal{M}_{q}(N)$, lie within $I_{M}^{\gr}$. Since $B$ is a submatrix of $A$, it follows that the $(|N|-(n-1))$-minors of $A$ belong to the ideal generated by the $(|N|-|\overline{J}|)$-minors of the matrix $B$. This implies the desired result.
\end{proof}

As an immediate consequence of Theorem~\ref{inc}, we obtain the following corollary. 

\begin{corollary}\label{fin}
Let $M$ be an $n$-paving matroid with no points of degree greater than two. Then: 
\[
I_{M}=\sqrt{I_{\mathcal{C}(M)}+I_{M}^{\gr}}.
\]
\end{corollary}

\begin{remark}\label{fin rel}
In particular, Corollary~\ref{fin} gives rise to a finite generating set for $I_{M}$,  up to radical, by selecting the sets $J,P,C$ that satisfy the hypothesis of Theorem~\ref{thm ci} and the additional vectors $\{q_{p}: p\in P\}$ as elements of the canonical basis $E=\{e_{1},\ldots,e_{n}\}$ of $\mathbb{C}^{n}$, as described in Remark~\ref{free}. We then take the union of these polynomials with the finite set of generators of $I_{\mathcal{C}(M)}$.
\end{remark}

%The following example is an application of Proposition~\ref{des 2} to the quadrilateral set. 

\begin{example}\label{quad2}
 Consider the point-line configuration $QS$ depicted in Figure~\ref{fig:combined}~(Center). By applying Proposition~\ref{des 2}, the ideal $I_{QS}$ is the radical of the ideal generated by $I_{\mathcal{C}(QS)}$ and %and $I_{QS}^{\lift}$. Here, $I_{QS}^{\lift}$ represents the ideal generated by 
the $4\times 4$ minors of the $4\times 6$ liftability matrices $\mathcal{M}_{q}(QS)$, where $q$ varies over $\CC^{3}$. Using Corollary~\ref{fin}, we obtain a finite generating set for the matroid ideal $I_{QS}$, up to radical.
%One can observe, using Lemma~\ref{lifts},
%As it was pointed in Example~\ref{quad}, this argument also applies to the $3\times 4$ grid %in Example~\ref{34}, 
%depicted in Figure~\ref{fig:combined 2}~(Left). 
\end{example}

\begin{example}\label{quad3}
 Consider the point-line configuration $G_{3\times 4}$ shown in Figure~\ref{fig:combined 2}~(Left). By applying Proposition~\ref{des 2}, the ideal $I_{G_{3\times 4}}$ is the radical of the ideal generated by $I_{\mathcal{C}(G_{3\times 4})}$ and %and $I_{QS}^{\lift}$. Here, $I_{QS}^{\lift}$ represents the ideal generated by 
the $10\times 10$ minors of the $16\times 12$ liftability matrices $\mathcal{M}_{q}(G_{3\times 4})$, where $q$ varies over $\CC^{3}$. Using Corollary~\ref{fin}, we obtain a finite generating set for the matroid ideal $I_{G_{3\times 4}}$, up to radical.
%One can observe, using Lemma~\ref{lifts},
%As it was pointed in Example~\ref{quad}, this argument also applies to the $3\times 4$ grid %in Example~\ref{34}, 
%depicted in Figure~\ref{fig:combined 2}~(Left). 
\end{example}

\section{Examples and applications}\label{sec 7}
We now apply Theorem~\ref{thm ci} and Corollary~\ref{fin} to specific matroids and compute polynomials within their ideals.   Throughout this section, we use Notations~\ref{notacion grafo G} and~\ref{initial notation}.

%We %willdemonstrate that for the Pascal configuration discussed in Example~\ref{conic}, the ideal $I_{M}^{\lift}$ 
%does not contain is not contained in the ideal generated by the polynomials constructed in \cite[Theorem~3.0.2]{sidman2021geometric} using the Grassmann-Cayley algebra. See Example~\ref{conic} and Remark~\ref{GC polynomials}.
\medskip
\noindent
We show that the polynomials constructed using our methods are not necessarily included among those from the Grassmann-Cayley approach. Specifically, we demonstrate that for the Pascal configuration discussed in Example~\ref{conic}, the ideal $I_{M}^{\lift}$ is not contained in the ideal generated by the polynomials from \cite[Theorem~3.0.2]{sidman2021geometric} using the Grassmann-Cayley algebra. See Example~\ref{conic} and Remark~\ref{GC polynomials}.

Furthermore, we %will
construct polynomials in $I_{CS}$ using Theorem~\ref{thm ci}, where $I_{CS}$ represents the ideal of the matroid, constructed in \cite{sidman2021geometric}, derived from $d+4$ points on a rational normal curve of degree $d$.

%\textcolor{red}{Examples~\ref{cycle},\ref{conic} and~\ref{ex n lines} lines were not consistent with our notation. Following Notation~\ref{notacion grafo G}, the $i^{th}$ circuit of $C$ should contain $p_{i}$. It is also convenient to put $p_{i}$ as the first element of $c_{i}$. I made these changes in these examples.}
%\textcolor{red}{I changed the variables appearing on the polynomials of Examples~\ref{cycle},\ref{ex n lines},\ref{ext} and Corollary~\ref{menelaus} to denote the extra vectors with the letter $q$.}

\begin{example}[{Cycle}]\label{cycle}
 Consider an $n$-paving matroid $M$, and let $J$, $P=\{p_{1},\ldots,p_{k}\}$, and $C=\{c_{1},\ldots,c_{k}\}$ be the sets used to construct $G$, with $p_{i}\in c_{i}$ for $i\in [k]$, as described in %Theorem~\ref{thm ci}.
Notation~\ref{notacion grafo G}. Suppose that $G$ is a cycle of size $k$. Without loss of generality, we may assume that %assume, without loss of generality, that 
$E(G)=\{(p_{i},p_{i+1}): i\in[k]\}$.
%$J$, $P$, and $C$ be as defined in 
%Assume that its weighted graph $G$ constructed with $J$, $P$, and $C$, forms a cycle of size $k$. Let $J$, $P=\{p_{1},\ldots,p_{k}\}$, and $C=\{c_{1},\ldots,c_{k}\}$ be the sets used to construct $G$, with $p_{i}\in c_{i}$ for $i\in [k]$. 
For $i\in [k]$, denote the elements of the circuit $c_{i}$ as:
$$c_{i}=\set{p_{i},p_{i+1}, r_{i,1},r_{i,2},\ldots,r_{i,n-2}},$$ 
with the points $r_{i,j}$ being inside $J$. %By applying the fact that the product of all the numbers associated with the edges of $G$ equals $1$, along with Equation~\eqref{7}, 
 Applying Theorem~\ref{thm ci}, we deduce that the polynomial
$$\prod_{i=1}^{k}\corch{p_{i}, r_{i,1},r_{i,2},\ldots,r_{i,n-2},q_{i}}-\prod_{i=1}^{k}\corch{p_{i+1}, r_{i,1},r_{i,2},\ldots,r_{i,n-2},q_{i}}$$
belongs to $I_{M}$ for any choice of  extra vectors $q_{i}\in \mathbb{C}^{n}$. A special instance arises when $n=3$ and $k=3$, which corresponds to the quadrilateral %case
 set $QS$, %as
depicted in Figure~\ref{fig:combined}~(Center). %{fig_quad}.
%Consider this specific 
 For the matroid $QS$, we choose $J=\{p_{1},p_{5},p_{6}\}$, $P= \{p_{2},p_{4},p_{3}\}$, and the set of circuits $C=\{\{p_{2},p_{4},p_{6}\}, \{p_{4},p_{3},p_{5}\},\{p_{3},p_{2},p_{1}\}\}$. Applying Theorem~\ref{thm ci}, we find that the polynomial
\[
\corch{p_{2},p_{6},q_{1}}\corch{p_{4},p_{5},q_{2}}\corch{p_{3},p_{1},q_{3}}-\corch{p_{4},p_{6},q_{1}}\corch{p_{3},p_{5},q_{2}}\corch{p_{2},p_{1},q_{3}}
\]
belongs to %the ideal 
$I_{QS}$ for any  choice of extra vectors 
$\{q_{1},q_{2},q_{3}\}\subset \CC^{3}$. Indeed, this polynomial coincides with the one constructed in \cite{Fatemeh3} for the quadrilateral set.
\end{example}

Expanding on Example~\ref{cycle}, we present the configuration consisting of a set of lines and all their intersection points,  which we denote by $L_{n}$.

%\begin{definition}\label{conf ln}

%\end{definition}

\begin{example}\label{ex n lines}
%Expanding on this observation, consider $\mathcal{L}$ as a set of $n$ lines, where no three lines are concurrent. From this set, we construct the $3$-paving matroid $M$, wherein $\mathcal{L}$ serves as the set of lines, and all possible intersections between two lines of $\mathcal{L}$ constitute the points. Note that the quadrilateral set emerges as a special case when $\lvert\mathcal{L}\rvert=4$.
 Consider $\mathcal{L}$ a collection of $n$ lines in $\mathbb{P}^{2}$, where no three lines are concurrent. Let $\mathcal{P}$ be the set of their $\textstyle \binom{n}{2}$ intersection points. Let $L_{n}$ denote the point-line configuration with the points $\mathcal{P}$ and the lines $\mathcal{L}$,
where three points in $\mathcal{P}$ are dependent if and only if they lie on the same line of $\mathcal{L}$.

 Choose a line $l$ in $\mathcal{L}$, and let the $n-1$ 
points on this line be $r_{1},r_{2},\ldots,r_{n-1}$.
Additionally, label $l_{1},l_{2},\ldots,l_{n-1}$ as the lines  distinct from $l$, containing the points %passing through 
$r_{1},r_{2},\ldots,r_{n-1}$, respectively.
We denote $\sigma=(c_{1},c_{2},\ldots,c_{n-1})$ as a single cycle permutation of $\{1,2,\ldots,n-1\}$ (having only one cycle of size $n-1$ in its decomposition). Then, for each $i\in [n-1]$, we denote  $p_{i}$ for the intersection of the lines $l_{c_{i}}$ and $l_{c_{i+1}}$, with $l_{c_{n}}=l_{c_{1}}$. Notably, for every $i\in [n-1]$, the points $p_{i-1}$ and $p_{i}$ align with  $r_{c_{i}}$.
We define 
\[
P=\set{p_{1},p_{2},\ldots,p_{n-1}},\quad  J=\set{ r_{1},\ldots,r_{n-1}},\quad \text{and} \quad  C=\set{ \set{p_{i},p_{i-1}, r_{c_{i}}}}:i\in [n-1].
\]
 The associated graph $G$ is a cycle on the vertices $\{p_{i}:i\in [n-1]\}$. By applying Theorem~\ref{thm ci} for $P$, $J$, and $C$, we derive that the polynomial
\[
\prod_{i=1}^{n-1}\corch{p_{i}, r_{c_{i}},q_{i}}-\prod_{i=1}^{n-1}\corch{p_{i-1}, r_{c_{i}},q_{i}}
\]
belongs to %$I_{M}$
 $I_{L_{n}}$ for any choice of extra vectors $q_{i}\in \mathbb{C}^{3}$.  Note that the same procedure can be applied to any subset of lines of $\mathcal{L}$.
Consequently,  this approach enables us to construct a collection of polynomials within  $I_{L_n}$. %we can generate multiple polynomials within %$I_{M}$
% $I_{L_{n}}$}
%for this configuration.
First, we can select any subset of $m$ lines,  fix a line among them, and subsequently a single cycle permutation of $\mathbb{S}_{ m-1}$. For each of these selections, we %can
construct a polynomial within  $I_{L_{n}}$. 
\end{example}

We now present a proof of Menelaus' Theorem for $n$-gons, using the construction of graph polynomials. We refer to \cite{shephard1995ceva} and the references therein for further details.
\begin{corollary}[\textup{\cite[Theorem~1]{shephard1995ceva}}]\label{menelaus} Let $v_{1},\ldots,v_{n}$ be distinct non-zero vectors in $\CC^{3}$. For each $i\in [n]$, let $w_{i}=\alpha_{i}v_{i}+\beta_{i}v_{i+1}$ be a vector in the $\text{span}\{{v_{i},v_{i+1}}\}$. Suppose that the points $\{w_{i}:i\in [n]\}$ lie on the same line in the projective plane, and that this line does not contain any of the other points. Then: 
\[\prod_{i\in [n]}\alpha_{i}=(-1)^{n}\prod_{i\in [n]}\beta_{i}.\]
\end{corollary}

\begin{proof}
Let $M$ be the point-line configuration whose set of points and lines are  $\mathcal{P}=\{v_{i},w_{j}:i,j\in [n]\}$ and $\mathcal{L}=\{\{v_{i},v_{i+1},w_{i}\},\{w_{1},\ldots,w_{n}\}:i\in [n]\}$, respectively. We select 
\[P=\{v_{1},\ldots,v_{n}\},\quad J=\{w_{1},\ldots,w_{n}\},\quad\text{and}\quad C=\{\{v_{i},v_{i+1},w_{i}\}:i\in [n]\}.
\]
 The associated graph $G$ is a cycle on the vertices $\{v_{i}:i\in [n]\}$. By applying Theorem~\ref{thm ci} for $P$, $J$, and $C$, for any choice of vectors $q_{i}\in \CC^{3}$, we obtain: 
\[\prod_{i=1}^{n}\corch{v_{i},w_{i}, q_{i}}=\prod_{i=1}^{n}\corch{v_{i+1},w_{i}, q_{i}}.
\]
By substituting $w_{i}$ with $\alpha_{i}v_{i}+\beta_{i}v_{i+1}$ within each bracket and using the multilinearity of the determinant, we obtain:
\[\prod_{i=1}^{n}\corch{v_{i},v_{i+1}, q_{i}}\beta_{i}=\prod_{i=1}^{n}\corch{v_{i},v_{i+1}, q_{i}}(-\alpha_{i}).\]
Finally, by selecting the vectors $q_{i}$ in such a way that none of the brackets vanish, and dividing both sides by $\textstyle \prod_{i=1}^{n-1}\corch{v_{i}, v_{i+1},q_{i}}$, we obtain the desired equality.
\end{proof}

\begin{example}[{Pascal matroid}]\label{conic}
Consider the %the conic, often referred to as the 
Pascal matroid, denoted as $M_{\text{Pascal}}$, consisting of $9$ points and $7$ lines, as depicted in Figure~\ref{fig:combined}~(Right), where the lines represent circuits of size $3$ within the matroid. We denote the ideal of this matroid as $I_{\text{Pascal}}$. We generate %extra 
polynomials within $I_{\text{Pascal}}$ using Theorem~\ref{thm ci}.
To do this, we select 
\begin{equation*}
\begin{aligned}
& \qquad \qquad \qquad J=\set{  {p_7},  {p_8},  {p_9}}, \ P= \set{  {p_1},  {p_6},  {p_5},  {p_4},  {p_3},  {p_2}},\\
&   C=\set{ \set{  {p_1},  {p_6},  {p_9}},\set{  {p_6},  {p_5},  {p_8}},\set{  {p_5},  {p_4},  {p_7}},\set{  {p_4},  {p_3},  {p_9}},\set{  {p_3},  {p_2},  {p_8}},\set{  {p_2},  {p_1},  {p_7}}}.
\end{aligned}
\end{equation*}
This establishes that the directed graph with vertices $\{  {p_1},  {p_6},  {p_5},  {p_4},  {p_3},  {p_2}\}$ forms a cycle. Applying Theorem~\ref{thm ci} for $J$, $P$, and $C$, we conclude that the polynomial
\begin{equation}
\begin{aligned}\label{eq:prom_notin}
&\left[  {p_1},  {p_9},q_{1}\right]\left[  {p_6},  {p_8},q_{2}\right]\left[  {p_5},  {p_7},q_{3}\right]\left[  {p_4},  {p_9},q_{4}\right]\left[  {p_3},  {p_8},q_{5}\right]\left[  {p_2},  {p_7},q_{6}\right]\
\\ &- \left[  {p_6},  {p_9},q_{1}\right]\left[  {p_5},  {p_8},q_{2}\right]\left[  {p_4},  {p_7},q_{3}\right]\left[  {p_3},  {p_9},q_{4}\right]\left[  {p_2},  {p_8},q_{5}\right]\left[  {p_1},  {p_7},q_{6}\right]
\end{aligned}
\end{equation}
belongs to $I_{\text{Pascal}}$ for any choice of vectors $q_{i}\in \mathbb{C}^{3}$.  Recall that in Equation~\eqref{eq:prom_notin} we are using Notation~\ref{initial notation}. Therefore, we can interpret this expression as a polynomial in $27$ variables, where each point $p_{1},\ldots,p_{9}$ corresponds to three indeterminates, and the vectors $q_{i}$ are elements of $\CC^{3}$ 
that contribute to the coefficients in the expansion of the bracket polynomials.
\end{example}

\begin{remark}\label{GC polynomials}
 In \cite[Theorem~3.0.2]{sidman2021geometric}, using the Grassmann-Cayley algebra, certain polynomials within the matroid ideal of $M_{\text{Pascal}}$ were constructed. However, as noted in Remark~\ref{relation between varieties}, the definition of the matroid ideal given in \cite{sidman2021geometric} differs from ours. Nevertheless, the polynomials constructed in \cite{sidman2021geometric} lead to polynomials within $I_{\text{Pascal}}$, as discussed in Remark~\ref{relation between varieties}. From now on, when we refer to the polynomials constructed in \cite{sidman2021geometric} for $M_{\text{Pascal}}$, we assume them to be inside $I_{\text{Pascal}}$.
%In this case, we have certain polynomials within $I_{\tex. Howevert{Pascal}}$ constructed in \cite[Theorem~3.0.2]{sidman2021geometric} by the Grassmann-Cayley algebra.
\end{remark}

In Proposition~\ref{notin}, we demonstrate that the polynomials constructed in Example~\ref{conic} are not encompassed within the ideal generated by the polynomials constructed in \cite[Theorem~3.0.2]{sidman2021geometric} using the Grassmann-Cayley algebra. For brevity, we denote this ideal as $GC_{\text{Pascal}}$.

\begin{proposition}\label{notin}
 Let $q\in \CC^{3}$ be any non-zero vector. The polynomial in~\eqref{eq:prom_notin}  obtained by setting all vectors $q_{i}$ equal to $q$
does not belong to $GC_{\text{Pascal}}$.
\end{proposition}

\begin{proof}
We will first prove the statement for $q=e_{3}=(0,0,1)$. We denote by $P$ the polynomial of  Equation~\eqref{eq:prom_notin} where 
 all vectors $q_{i}$ are set equal to $e_{3}$. To demonstrate that $P\notin GC_{\text{Pascal}}$, we prove that $V(GC_{\text{Pascal}})\not\subset V(P)$.
 For each $i\in [9]$, %let the vector $p_{i}$ 
we consider a vector $p_{i}$ of the form $(x_{i},1,0)$, where $x_{i}\in \mathbb{C}$ are values that we will now define. Denote by $\gamma$ the collection of vectors $\{p_{i}:i\in [9]\}$. Notably, $[p_{i},p_{j},e_{3}]=x_{i}-x_{j}$ for every $i,j\in [9]$.
All vectors in $\gamma$ are in the subspace $x_{3}=0$, implying that $[p_{i},p_{j},p_{k}]=0$ for every $i,j,k \in [9]$. Consequently, $\gamma\in V(GC_{\text{Pascal}})$ since the generators of $GC_{\text{Pascal}}$ involve polynomials in these brackets.
\smallskip 

On the other hand, we have that $P(\gamma)=0$ if and only if 
\[
(x_{1}-x_{9})(x_{6}-x_{8})(x_{5}-x_{7})(x_{4}-x_{9})(x_{3}-x_{8})(x_{2}-x_{7})=
(x_{6}-x_{9})(x_{5}-x_{8})(x_{4}-x_{7})(x_{3}-x_{9})(x_{2}-x_{8})(x_{1}-x_{7}).     
\]              
Therefore, by selecting the numbers $\{x_{i}:i\in [9]\}$ such that this equality does not hold, we ensure that $\gamma\in V(GC_{\text{Pascal}})$ while $\gamma \notin V(P)$. Thus, we have proved that $P\notin GC_{\text{Pascal}}$.

Now, consider an arbitrary  non-zero vector $q\in \mathbb{C}^{3}$, and let $Q$ denote  the polynomial of Equation~\eqref{eq:prom_notin} where 
all vectors $q_{i}$ are set equal to $q$.
%$Q$ denote the associated polynomial %from the proposition 
%of this vector. 
%Suppose
Consider $\gamma \in V(GC_{\text{Pascal}})\backslash V(P)$, and let $T$ be a projective transformation such that $T(q)=e_{3}$. Then,
$$Q(T(\gamma))=\det(T)^{6}P(\gamma)\neq 0,$$
where $T(\gamma)$ is obtained by applying $T$ to all the vectors of $\gamma$. Furthermore, since $\gamma\in V(GC_{\text{Pascal}})$ and $V(GC_{\text{Pascal}})$ is closed under projective transformations (due to all generators of $GC_{\text{Pascal}}$ being homogeneous polynomials in the brackets), we conclude that $T(\gamma)\in V(GC_{\text{Pascal}})$. Therefore, we have:
$$T(\gamma)\in V(GC_{\text{Pascal}})\backslash V(Q),$$ 
which completes the proof.
\end{proof}

%\textcolor{red}{I removed the remark that was here since it was quite informal and not necessary.}
\begin{comment}
\begin{remark}
    Note that we have demonstrated this proposition for a specific matroid. However, this argument can be generalized to any matroid. We can apply the same reasoning to deduce that the polynomials constructed in Theorem~\ref{thm ci} are not included in the ideal generated by the polynomials constructed using the Grassmann-Cayley algebra.
\end{remark}
\end{comment}

We can extend Example~\ref{conic} to higher dimensions. In \cite{caminata2021pascal}, Caminata and Schaffler established a necessary and sufficient condition for $d+4$ points in $\mathbb{P}^{d}$ to lie on a rational normal curve of degree $d$. We will leverage this result to introduce our next example. Note that any collection of $d+1$ points on a rational normal curve of degree $d$ are linearly independent.

\begin{theorem}[\textup{\cite[Corollary~5.2]{caminata2021pascal}}] \label{teo cur}
Let $P=\set{p_{1},\ldots,p_{d+4}}$ be $d+4$ points in $\mathbb{P}^{d}$, not all lying in the same hyperplane. Then, the points of $P$ lie on a quasi-Veronese curve if and only if the following equality holds in the Grassmann-Cayley algebra for every $I=\{i_{1}<\cdots<i_{6}\}$ and $I^{c}=\{j_{1}<\cdots<j_{d-2}\}$: 
\begin{equation}\label{eq gc}
(p_{i_{1}}p_{i_{2}}\wedge p_{i_{4}}p_{i_{5}}H_{I}) \vee (p_{i_{2}}p_{i_{3}}\wedge p_{i_{5}}p_{i_{6}}H_{I}) \vee (p_{i_{3}}p_{i_{4}}\wedge p_{i_{6}}p_{i_{1}}H_{I}) \vee H_{I} = 0,
\end{equation}
where $H_{I}=j_{1}\cdots j_{d-2}.$
\end{theorem}

Note that when $d=2$, Equation~\eqref{eq gc} yields the quartic from Example~\ref{conic}. In \cite{sidman2021geometric}, the authors introduced a matroid using Theorem~\ref{teo cur}. We provide the details in the following definition.

\begin{definition}[\cite{sidman2021geometric}]
Suppose $P=\set{p_{1},\ldots,p_{d+4}}\subset \mathbb{P}^{d}$ lies on a rational normal curve of degree $d$. Using the notation of Theorem~\ref{teo cur}, consider the set 
\[
\mathcal{Q}=\bigcup_{I=\set{i_{1}<\cdots <i_{6}}} \set{
(p_{i_{1}}p_{i_{2}} \wedge p_{i_{4}}p_{i_{5}} H_{I}),
(p_{i_{2}}p_{i_{3}} \wedge p_{i_{5}}p_{i_{6}} H_{I}),
(p_{i_{3}}p_{i_{4}} \wedge p_{i_{6}}p_{i_{1}} H_{I})
}.
\]
Then we define $A_{CS}$ as the matroid defined by the $d+4+\size{\mathcal{Q}}$ vectors of $P \cup \mathcal{Q}$, and denote the associated matroid ideal by $I_{CS}$.
\end{definition}

In \cite[Theorem~4.2.2]{sidman2021geometric}, non-trivial quartics in $I_{CS}$ were constructed using the Grassmann-Cayley algebra. In Example~\ref{ext}, we construct polynomials in $I_{CS}$ using Theorem~\ref{thm ci}.  Note that, as discussed in Remark~\ref{for any matroid}(ii), the matroid $A_{CS}$ is not paving.

\begin{example}\label{ext}
Consider the matroid $A_{CS}$ and take any $I=\set{i_{1}<\cdots <i_{6}}$. Consider the points: % $x_{1},x_{2},x_{3}$ denote the points:
\[
x_1=(p_{i_{1}}p_{i_{2}}\wedge p_{i_{4}}p_{i_{5}}H_{I}),
\quad x_2=(p_{i_{2}}p_{i_{3}}\wedge p_{i_{5}}p_{i_{6}}H_{I}),\quad x_3=(p_{i_{3}}p_{i_{4}}\wedge p_{i_{6}}p_{i_{1}}H_{I}).
\]
Note that since any subset of $d+1$ points of $P$ is linearly independent, the points  $\{p_{i_1},\ldots,p_{i_6}\}$ lie outside the hyperplane  $\is{H_{I},x_{1},x_{2},x_{3}}$  (this is a hyperplane due to Theorem~\ref{teo cur}). Now we select 
$$J=\set{p_{j_{1}},\ldots,p_{j_{d-2}}}\cup \set{x_{1},x_{2},x_{3}},\quad P=\set{p_{i_{1}},\ldots,p_{i_{6}}},$$
where $\set{j_{1},\ldots,j_{d-2}}=I^{c}$, and the set of circuits $C$ as 
\begin{align*}
&C=\{\set{p_{i_{1}},p_{i_{2}},x_{1}},\set{p_{i_{2}},p_{i_{3}},x_{2}},\set{p_{i_{3}},p_{i_{4}},x_{3}},\set{p_{i_{4}},p_{i_{5}},p_{j_{1}},\ldots,p_{j_{d-2}}, x_{1}},\\
& \quad \qquad\set{p_{i_{5}},p_{i_{6}},p_{j_{1}},\ldots,p_{j_{d-2}}, x_{2}},\{p_{i_{6}},p_{i_{1}},p_{j_{1}},\ldots,p_{j_{d-2}}, x_{3}\}\}.
\end{align*}
Note that $P\cap \overline{J}=\emptyset$. Thus, the directed graph in $\{p_{i_{1}},\ldots,p_{i_{6}}\}$ forms a cycle. By applying Theorem~\ref{thm ci} and Remark~\ref{rem 5} with $J,P,C$, we deduce that the polynomial 
\begin{align*}&\qquad \left[p_{i_{1}},x_{1}, q_{1}^{1},\ldots,q_{d-1}^{1}\right]\corch{p_{i_{2}},x_{2}, q_{1}^{2},\ldots,q_{d-1}^{2}}\left[p_{i_{3}},x_{3}, q_{1}^{3},\ldots,q_{d-1}^{3}\right]\cdot \\
&\corch{p_{i_{4}},x_{1},p_{j_{1}},\ldots,p_{j_{d-2}},q_{1}}\corch{p_{i_{5}},x_{2},p_{j_{1}},\ldots,p_{j_{d-2}},q_{2}}\corch{p_{i_{6}},x_{3},p_{j_{1}},\ldots,p_{j_{d-2}},q_{3}}-\\[0.7 ex] 
& \qquad \left[p_{i_{2}},x_{1}, q_{1}^{1},\ldots,q_{d-1}^{1}\right]\corch{p_{i_{3}},x_{2}, q_{1}^{2},\ldots,q_{d-1}^{2}}\corch{p_{i_{4}},x_{3}, q_{1}^{3},\ldots,q_{d-1}^{3}}\cdot \\
&\corch{p_{i_{5}},x_{1},p_{j_{1}},\ldots,p_{j_{d-2}},q_{1}}\corch{p_{i_{6}},x_{2},p_{j_{1}},\ldots,p_{j_{d-2}},q_{2}}\corch{p_{i_{1}},x_{3},p_{j_{1}},\ldots,p_{j_{d-2}},q_{3}}
\end{align*}
belongs to $I_{CS}$ for any election of the extra vectors in these brackets.
\end{example}

%\begin{definition}\label{grid n times k}
% Consider the $n\times k$ matrix  
%\[\mathcal{Y}=\begin{pmatrix}
%p_{11} & p_{12} & \cdots & p_{1k} \\
%p_{21} & p_{22} & \cdots & p_{2k} \\
%\vdots & \vdots &\ddots &\vdots\\
%p_{n1} & p_{n2} & \cdots & p_{nk}
%\end{pmatrix}\]
%where $k\geq 2n-2$. For each $i\in [n]$ and $j\in [k]$, we denote the rows and columns of $\mathcal{Y}$ by 
%\[
%R_{i}=\{p_{i1},p_{i2},\ldots,p_{ik}\},\quad\text{ and }\quad
%C_{j}=\{p_{1j},\ldots,p_{nj}\}.
%\]
%We define $G_{n\times k}$ as the $n$-paving matroid with the set of points $\mathcal{P}=\{p_{i,j} : (i,j)\in [n]\times [k]\}$ and the set of dependent hyperplanes \[\mathcal{L}=\{R_{1},\ldots,R_{n}\}\cup \{C_{1},\ldots,C_{k-(n-2)}\}\cup \{\bigcup_{j\geq k-n+3}C_{j}\},\]
%where the last dependent hyperplane is formed by the last $n-2$ columns of $\mathcal{Y}$. Using Lemma~\ref{sub h}, it is clear that this defines an $n$-paving matroid.}
%\end{definition}

\begin{example}[$n\times k$ grid]\label{grid}
 Consider the $n\times k$ matrix $\mathcal{Y}=(p_{ij})$, 
%\[\mathcal{Y}=\begin{pmatrix}
%p_{11} & p_{12} & \cdots & p_{1k} \\
%p_{21} & p_{22} & \cdots & p_{2k} \\
%\vdots & \vdots &\ddots &\vdots\\
%p_{n1} & p_{n2} & \cdots & p_{nk}
%\end{pmatrix}\]
where $k\geq 2n-2$. For each $i\in [n]$ and $j\in [k]$, we denote the rows and columns of $\mathcal{Y}$ by 
\[
R_{i}=\{p_{i1},p_{i2},\ldots,p_{ik}\},\quad\text{ and }\quad
C_{j}=\{p_{1j},\ldots,p_{nj}\}.
\]
We define $G_{n\times k}$ as the $n$-paving matroid with the set of points $\mathcal{P}=\{p_{i,j} : (i,j)\in [n]\times [k]\}$ and the set of dependent hyperplanes \[\mathcal{L}=\{R_{1},\ldots,R_{n}\}\cup \{C_{1},\ldots,C_{k-(n-2)}\}\cup \{\bigcup_{j\geq k-n+3}C_{j}\},\]
where the last dependent hyperplane is formed by the last $n-2$ columns of $\mathcal{Y}$. Using Lemma~\ref{sub h}, it is clear that this defines an $n$-paving matroid.
%Consider a grid of size $n\times k$ consisting of $n$ horizontal lines and $k$ vertical lines, where $k\geq 2n-2$. We construct the $n$-paving matroid $M$ using this grid, where the points are the $nk$ points of intersection between horizontal and vertical lines, and the dependent hyperplanes consist of all the horizontal and vertical lines, except for the last $n-2$ vertical lines forming a dependent hyperplane of $M$. It is evident from this construction that $M$ satisfies the axioms of a matroid. We label the points of $M$ as $\{p_{i,j} : (i,j)\in [n]\times [k]\}$.
% Consider the $n$-paving matroid $G_{n\times k}$ from Definition~\ref{grid n times k}. 
 Now, we define
\[P=\{p_{ij}: j\leq k-n+2\},\quad \text{and} \quad J=\{p_{ij}: j\geq k-n+3\}=\bigcup_{j\geq k-n+3}C_{j}.\]
Observe that $P\cap \closure{J}=P\cap J=\emptyset$.
%We choose $J$ as the dependent hyperplane formed by the last $n-2$ vertical lines of the grid (corresponding to points $p_{ij}$ with $j\geq k-n+3$), and define the set $P$ as $\mathcal{P}\setminus J$.
To apply Theorem~\ref{thm ci}, we must choose a circuit for each point in $P$. We do this by selecting a subset $\mathcal{D}\subset P$ such that $\mathcal{D}$ contains at most one point from each column and at least one point from each row (this requires $k\geq 2n-2$ for $\mathcal{D}$ to exist). For each $p_{ij}\in \mathcal{D}$, we choose its circuit as the column $C_{j}$. For each $p_{ij}\notin \mathcal{D}$, we select a point $q\in \mathcal{D}$ in the row $R_{i}$, and define its circuit as $\{p_{ij},q,p_{i,k-n+3},\ldots,p_{i,k}\}$. This construction yields a polynomial within $I_{G_{n\times k}}$ by applying Theorem~\ref{thm ci}.
\end{example}

%\begin{remark}
 Note that the condition in Example~\ref{grid}, requiring  $\mathcal{D}$ to contain at most one point from each column, is essential to avoid obtaining zero polynomials, as discussed in Remark~\ref{is zero}.
%\end{remark}

\begin{example}[{$3\times 4$ grid}]\label{34}
 We follow the procedure of Example~\ref{grid} for the point-line configuration $G_{3\times 4}$ 
depicted in Figure~\ref{fig:combined 2}~(Left). %, defining a $3$-paving matroid. 
 We choose 
\[J=\set{p_{14}, p_{24}, p_{34}},\quad \text{and} \quad  P=\set{p_{ij}:j\in [3]}.\]
Then, we select a subset $\mathcal{D}\subset P$ (using the same notation as in the previous example) consisting of three points, with no two in the same row or column. Using $\mathcal{D}$, we select one circuit for each point in %$\mathcal{P}\backslash J$
 $P$, as previously described. Additionally, for each $j\in [3]$, we select the same  extra vector in $\CC^{3}$ associated with the two points in $\{p_{j1},p_{j2},p_{j3}\}\backslash \mathcal{D}$. This results in six extra vectors. Regardless of the choice of $\mathcal{D}$, we obtain the polynomial:
\[\sum_{\sigma \in \mathbb{S}_{3}}\sign\pare{\sigma}\corch{p_{1\sigma\pare{1}},p_{14},q_{1}}\corch{p_{2\sigma\pare{2}},p_{24},q_{2}} [p_{3\sigma\pare{3}},p_{34},q_{3}]
[r_{\sigma,1,1},r_{\sigma,1,2},q_{4}]
[r_{\sigma,2,1},r_{\sigma,2,2},q_{5}]
[r_{\sigma,3,1},r_{\sigma,3,2},q_{6}]
\]
where the $q_{i}$'s are arbitrary vectors in $\CC^{3}$ (the extra vectors), and  $r_{\sigma,k,1}$ and $r_{\sigma,k,2}$ refer to the other two points in the column $C_{k}$ besides $p_{\sigma^{-1}(k) k}$. %This polynomial matches the one constructed in \cite{Fatemeh3}, where 
It is shown in \cite[Theorem~5.7]{Fatemeh3} that these polynomials, combined with $I_{\mathcal{C}(G_{3\times 4})}$, generate the matroid ideal $I_{G_{3\times 4}}$.
\end{example}

\begin{figure}[h]
  \begin{subfigure}[b]{0.3\textwidth}
        \centering

\tikzset{every picture/.style={line width=0.75pt}} %set default line width to 0.75pt        

\begin{tikzpicture}[x=0.75pt,y=0.75pt,yscale=-1,xscale=1]

%Straight Lines [id:da5615116120148407] 
\draw    (102.2,130.67) -- (103.35,223) ;
%Straight Lines [id:da6886760205213638] 
\draw    (125.26,130.67) -- (126.42,223) ;
%Straight Lines [id:da6503057091410785] 
\draw    (171.4,130.67) -- (172.55,223) ;
%Straight Lines [id:da5004472388161385] 
\draw    (149.49,130.67) -- (150.64,223) ;
%Straight Lines [id:da877507178332296] 
\draw    (80.67,153.37) -- (194.85,153.37) ;
%Straight Lines [id:da6288376983503297] 
\draw    (80.67,176.45) -- (194.85,176.45) ;
%Straight Lines [id:da26927951775166814] 
\draw    (80.67,199.53) -- (194.85,199.53) ;
%Shape: Ellipse [id:dp7202893110777695] 
\draw  [fill={rgb, 255:red, 173; green, 216; blue, 230}, fill opacity=1  ,fill opacity=1 ] (99.51,153.68) .. controls (99.51,152.02) and (100.85,150.67) .. (102.51,150.67) .. controls (104.17,150.67) and (105.52,152.02) .. (105.52,153.68) .. controls (105.52,155.34) and (104.17,156.69) .. (102.51,156.69) .. controls (100.85,156.69) and (99.51,155.34) .. (99.51,153.68) -- cycle ;
%Shape: Ellipse [id:dp9300279095239248] 
\draw  [fill={rgb, 255:red, 173; green, 216; blue, 230}, fill opacity=1  ,fill opacity=1 ] (168.71,153.68) .. controls (168.71,152.02) and (170.06,150.67) .. (171.72,150.67) .. controls (173.38,150.67) and (174.72,152.02) .. (174.72,153.68) .. controls (174.72,155.34) and (173.38,156.69) .. (171.72,156.69) .. controls (170.06,156.69) and (168.71,155.34) .. (168.71,153.68) -- cycle ;
%Shape: Ellipse [id:dp024737387751922357] 
\draw  [fill={rgb, 255:red, 173; green, 216; blue, 230}, fill opacity=1  ,fill opacity=1 ] (168.52,176.83) .. controls (168.52,175.17) and (169.86,173.82) .. (171.52,173.82) .. controls (173.19,173.82) and (174.53,175.17) .. (174.53,176.83) .. controls (174.53,178.5) and (173.19,179.84) .. (171.52,179.84) .. controls (169.86,179.84) and (168.52,178.5) .. (168.52,176.83) -- cycle ;
%Shape: Ellipse [id:dp908053713226831] 
\draw  [fill={rgb, 255:red, 173; green, 216; blue, 230}, fill opacity=1  ,fill opacity=1 ] (169.86,199.85) .. controls (169.86,198.19) and (171.21,196.84) .. (172.87,196.84) .. controls (174.53,196.84) and (175.88,198.19) .. (175.88,199.85) .. controls (175.88,201.51) and (174.53,202.86) .. (172.87,202.86) .. controls (171.21,202.86) and (169.86,201.51) .. (169.86,199.85) -- cycle ;
%Shape: Ellipse [id:dp9471850592660367] 
\draw  [fill={rgb, 255:red, 173; green, 216; blue, 230}, fill opacity=1  ,fill opacity=1 ] (146.79,153.68) .. controls (146.79,152.02) and (148.14,150.67) .. (149.8,150.67) .. controls (151.46,150.67) and (152.81,152.02) .. (152.81,153.68) .. controls (152.81,155.34) and (151.46,156.69) .. (149.8,156.69) .. controls (148.14,156.69) and (146.79,155.34) .. (146.79,153.68) -- cycle ;
%Shape: Ellipse [id:dp6951181906473445] 
\draw  [fill={rgb, 255:red, 173; green, 216; blue, 230}, fill opacity=1  ,fill opacity=1 ] (147.05,176.82) .. controls (147.05,175.16) and (148.4,173.81) .. (150.06,173.81) .. controls (151.72,173.81) and (153.07,175.16) .. (153.07,176.82) .. controls (153.07,178.49) and (151.72,179.83) .. (150.06,179.83) .. controls (148.4,179.83) and (147.05,178.49) .. (147.05,176.82) -- cycle ;
%Shape: Ellipse [id:dp6494721549827923] 
\draw  [fill={rgb, 255:red, 173; green, 216; blue, 230}, fill opacity=1  ,fill opacity=1 ] (146.79,199.85) .. controls (146.79,198.19) and (148.14,196.84) .. (149.8,196.84) .. controls (151.46,196.84) and (152.81,198.19) .. (152.81,199.85) .. controls (152.81,201.51) and (151.46,202.86) .. (149.8,202.86) .. controls (148.14,202.86) and (146.79,201.51) .. (146.79,199.85) -- cycle ;
%Shape: Ellipse [id:dp39578946749516974] 
\draw  [fill={rgb, 255:red, 173; green, 216; blue, 230}, fill opacity=1  ,fill opacity=1 ] (122.57,153.68) .. controls (122.57,152.02) and (123.92,150.67) .. (125.58,150.67) .. controls (127.24,150.67) and (128.59,152.02) .. (128.59,153.68) .. controls (128.59,155.34) and (127.24,156.69) .. (125.58,156.69) .. controls (123.92,156.69) and (122.57,155.34) .. (122.57,153.68) -- cycle ;
%Shape: Ellipse [id:dp4266097208967894] 
\draw  [fill={rgb, 255:red, 173; green, 216; blue, 230}, fill opacity=1  ,fill opacity=1 ] (122.83,176.68) .. controls (122.83,175.02) and (124.18,173.67) .. (125.84,173.67) .. controls (127.5,173.67) and (128.85,175.02) .. (128.85,176.68) .. controls (128.85,178.34) and (127.5,179.69) .. (125.84,179.69) .. controls (124.18,179.69) and (122.83,178.34) .. (122.83,176.68) -- cycle ;
%Shape: Ellipse [id:dp27679712757650365] 
\draw  [fill={rgb, 255:red, 173; green, 216; blue, 230}, fill opacity=1  ,fill opacity=1 ] (122.57,198.69) .. controls (122.57,197.03) and (123.92,195.68) .. (125.58,195.68) .. controls (127.24,195.68) and (128.59,197.03) .. (128.59,198.69) .. controls (128.59,200.36) and (127.24,201.7) .. (125.58,201.7) .. controls (123.92,201.7) and (122.57,200.36) .. (122.57,198.69) -- cycle ;
%Shape: Ellipse [id:dp9672193306058067] 
\draw  [fill={rgb, 255:red, 173; green, 216; blue, 230}, fill opacity=1  ,fill opacity=1 ] (100.66,199.85) .. controls (100.66,198.19) and (102.01,196.84) .. (103.67,196.84) .. controls (105.33,196.84) and (106.67,198.19) .. (106.67,199.85) .. controls (106.67,201.51) and (105.33,202.86) .. (103.67,202.86) .. controls (102.01,202.86) and (100.66,201.51) .. (100.66,199.85) -- cycle ;
%Shape: Ellipse [id:dp8437734462763792] 
\draw  [fill={rgb, 255:red, 173; green, 216; blue, 230}, fill opacity=1  ,fill opacity=1 ] (99.77,176.83) .. controls (99.77,175.17) and (101.11,173.82) .. (102.77,173.82) .. controls (104.43,173.82) and (105.78,175.17) .. (105.78,176.83) .. controls (105.78,178.5) and (104.43,179.84) .. (102.77,179.84) .. controls (101.11,179.84) and (99.77,178.5) .. (99.77,176.83) -- cycle ;

% Text Node
\draw (85,136.98) node [anchor=north west][inner sep=0.75pt]   [align=left] {{\tiny $p_{11}$}};
% Text Node
\draw (116.52-8,160.68) node [anchor=north west][inner sep=0.75pt]   [align=left] {{\tiny $p_{22}$}};
% Text Node
\draw (140.91-8,161.45) node [anchor=north west][inner sep=0.75pt]   [align=left] {{\tiny $p_{23}$}};
% Text Node
\draw (157.27-3,160.76) node [anchor=north west][inner sep=0.75pt]   [align=left] {{\tiny $p_{24}$}};
% Text Node
\draw (140.91-8,183.53) node [anchor=north west][inner sep=0.75pt]   [align=left] {{\tiny $p_{33}$}};
% Text Node
\draw (157.27-3,183.15) node [anchor=north west][inner sep=0.75pt]   [align=left] {{\tiny $p_{34}$}};
% Text Node
\draw (117.65-8.5,183.15) node [anchor=north west][inner sep=0.75pt]   [align=left] {{\tiny $p_{32}$}};
% Text Node
\draw (157.27-3,137.75) node [anchor=north west][inner sep=0.75pt]   [align=left] {{\tiny $p_{14}$}};
% Text Node
\draw (140.72-8,136.98) node [anchor=north west][inner sep=0.75pt]   [align=left] {{\tiny $p_{13}$}};
% Text Node
\draw (116.5-8,136.98) node [anchor=north west][inner sep=0.75pt]   [align=left] {{\tiny $p_{12}$}};
% Text Node
\draw (93.43-8,183.15) node [anchor=north west][inner sep=0.75pt]   [align=left] {{\tiny $p_{31}$}};
% Text Node
\draw (92.28-8,160.06) node [anchor=north west][inner sep=0.75pt]   [align=left] {{\tiny $p_{21}$}};
% Text Node
%\draw (114,226) node [anchor=north west][inner sep=0.75pt]   [align=left] {{\footnotesize 3x4 Grid}};

\end{tikzpicture}
        \label{fig:quadrilateral}
    \end{subfigure}
\hspace{-3cm}  
\begin{tikzpicture}%[line cap=round,line join=round,>=triangle 45,x=1cm,y=1cm]
\clip(-8.684444444444446,-4.0388888888889) rectangle (6.008888888888889,-1.49222222222222344);
\draw [line width=0.75pt, color=black] (-4.24,-3.73)-- (-1.6088888888888893,-3.7344444444444527);
\draw [line width=0.75pt, color=black] (-3.0687046559256035,-1.8320123232163457)-- (-1.6088888888888893,-3.7344444444444527);
\draw [line width=0.75pt, color=black] (-3.015298016155543,-3.7320687533510926)-- (-1.697766319984402,-1.8343280974324667);
\draw [line width=0.75pt, color=black] (-4.26,-1.83)-- (-3.015298016155543,-3.7320687533510926);
\draw [line width=0.75pt, color=black] (-4.24,-3.73)-- (-3.0687046559256035,-1.8320123232163457);
\draw [line width=0.75pt, color=black] (-4.26,-1.83)-- (-1.697766319984402,-1.8343280974324667);
\draw [line width=0.75pt, color=black] (0.94732799070567,-1.8387961621464783)-- (2.2632632791006335,-3.7409852420255283);
\draw [line width=0.75pt, color=black] (2.2632632791006335,-3.7409852420255283)-- (3.307889010724878,-1.8427835963019235);
\draw [line width=0.75pt, color=black] (2.2632632791006335,-3.7409852420255283)-- (2.1872224466953636,-1.8408905784572756);
\draw [line width=0.75pt, color=black] (1.5990046895995609,-2.7807971185609333)-- (2.790064098030882,-2.783729307015079);
\draw [line width=0.75pt, color=black] (0.94732799070567,-1.8387961621464783)-- (3.307889010724878,-1.8427835963019235);
\begin{scriptsize}
\draw [fill={rgb, 255:red, 173; green, 216; blue, 230}, fill opacity=1] (-4.24,-3.73) circle (2pt);
\draw[color=black] (-4.251111111111112,-3.9477777777777863) node {${p_3}$};
\draw [fill={rgb, 255:red, 173; green, 216; blue, 230}, fill opacity=1] (-1.6088888888888893,-3.7344444444444527) circle (2pt);
\draw[color=black] (-1.6111111111111115,-3.9366666666666753) node {${p_2}$};
\draw [fill={rgb, 255:red, 173; green, 216; blue, 230}, fill opacity=1] (-4.26,-1.83) circle (2pt);
\draw[color=black] (-4.231111111111112,-1.5966666666666698) node {${p_8}$};
\draw [fill={rgb, 255:red, 173; green, 216; blue, 230}, fill opacity=1] (-1.697766319984402,-1.8343280974324667) circle (2pt);
\draw[color=black] (-1.6444444444444448,-1.6144444444444477) node {${p_6}$};
\draw [fill={rgb, 255:red, 173; green, 216; blue, 230}, fill opacity=1] (-3.0687046559256035,-1.8320123232163457) circle (2pt);
\draw[color=black] (-3.004444444444445,-1.5966666666666698) node {${p_7}$};
\draw [fill={rgb, 255:red, 173; green, 216; blue, 230}, fill opacity=1] (-3.015298016155543,-3.7320687533510926) circle (2pt);
\draw[color=black] (-2.8977777777777782,-3.9388888888888974) node {${p_1}$};
\draw [fill={rgb, 255:red, 173; green, 216; blue, 230}, fill opacity=1] (0.94732799070567,-1.8387961621464783) circle (2pt);
\draw[color=black] (0.7733333333333334,-1.632222222222226) node {${p_7}$};
\draw [fill={rgb, 255:red, 173; green, 216; blue, 230}, fill opacity=1] (2.1872224466953636,-1.8408905784572756) circle (2pt);
\draw[color=black] (2.1,-1.6144444444444477) node {${p_6}$};
\draw [fill={rgb, 255:red, 173; green, 216; blue, 230}, fill opacity=1] (3.307889010724878,-1.8427835963019235) circle (2pt);
\draw[color=black] (3.231111111111112,-1.592222222222226) node {${p_5}$};
\draw [fill={rgb, 255:red, 173; green, 216; blue, 230}, fill opacity=1] (-3.6462544213195143,-2.767884170562566) circle (2pt);
\draw[color=black] (-3.922222222222223,-2.7188888888888946) node {${p_4}$};
\draw [fill={rgb, 255:red, 173; green, 216; blue, 230}, fill opacity=1] (-2.348118764187554,-2.771079962778145) circle (2pt);
\draw[color=black] (-2.103333333333334,-2.7366666666666725) node {${p_5}$};
\draw [fill={rgb, 255:red, 173; green, 216; blue, 230}, fill opacity=1] (2.2632632791006335,-3.7409852420255283) circle (2pt);
\draw[color=black] (2.3,-3.97) node {${p_1}$};
\draw [fill={rgb, 255:red, 173; green, 216; blue, 230}, fill opacity=1] (1.5990046895995609,-2.7807971185609333) circle (2pt);
\draw[color=black] (1.2933333333333336,-2.7633333333333393) node {${p_4}$};
\draw [fill={rgb, 255:red, 173; green, 216; blue, 230}, fill opacity=1] (2.2248986974008673,-2.7823379646095874) circle (2pt);
\draw[color=black] (2.48,-2.5566666666666724) node {${p_3}$};
\draw [fill={rgb, 255:red, 173; green, 216; blue, 230}, fill opacity=1] (2.790064098030882,-2.783729307015079) circle (2pt);
\draw[color=black] (3.124444444444445,-2.69) node {${p_2}$};
\end{scriptsize}
\end{tikzpicture}
  \caption{(Left) $3\times 4$ grid; (Center) Matroid in Example~\ref{Ex:8 nodes}; (Right) Matroid in Example~\ref{Ex:7 nodes}.}    \label{fig:combined 2}
\end{figure}

\begin{example}\label{Ex:8 nodes}
Consider the point-line configuration $M$ depicted in Figure~\ref{fig:combined 2}~(Center). %{fig_quad}.
For  
\begin{equation*}
\begin{aligned}& \qquad \qquad J=\{ {p_6}, {p_7}, {p_8}\}, P=\{ {p_1}, {p_2}, {p_3}, {p_4}, {p_5}\},\ \text{and}\\
&C=\{\{ {p_1}, {p_2}, {p_3}\},\{ {p_2}, {p_5}, {p_7}\},\{ {p_3}, {p_4}, {p_7}\},\{ {p_4}, {p_1}, {p_8}\},\{ {p_5}, {p_1}, {p_6}\}\},
\end{aligned}
\end{equation*}
the associated graph $G$ has the vertex set $V(G)=P$, the edge set \[E(G)=\{ ({p_1}, {p_2}), ({p_1}, {p_3}), ({p_2}, {p_5}), ({p_3}, {p_4}), ({p_4}, {p_1}), ({p_5}, {p_1})\},\] and the cycles $\{ ({p_1}, {p_2}, {p_5}), ({p_1}, {p_3}, {p_4})\}$.
By Theorem~\ref{thm ci}, the polynomial 
\begin{equation*}
\begin{gathered}
\corch{ {p_1}, {p_3},q_{1}}\corch{ {p_2}, {p_7},q_{2}}\corch{ {p_5}, {p_6},q_{3}}\corch{ {p_4}, {p_7},q_{4}}\corch{ {p_1}, {p_8},q_{5}}-\corch{ {p_1}, {p_2},q_{1}}\corch{ {p_3}, {p_7},q_{4}}\corch{ {p_4}, {p_8},q_{5}}\corch{ {p_5}, {p_7},q_{2}}\corch{ {p_1}, {p_6},q_{3}}\\ -\corch{ {p_2}, {p_3},q_{1}}\corch{ {p_5}, {p_7},q_{2}}\corch{ {p_1}, {p_6},q_{3}}\corch{ {p_4}, {p_7},q_{4}}\corch{ {p_1}, {p_8},q_{5}},
\end{gathered}
\end{equation*}
is in the ideal $I_{M}$, for any collection of vectors $\{q_{1},q_{2},q_{3},q_{4},q_{5}\}$ in $\CC^{3}$.  
\end{example}

\begin{example}\label{Ex:7 nodes}
Let $M$ be the point-line configuration depicted in Figure~\ref{fig:combined 2}~(Right). %{fig_quad}.
We choose $J=\{ {p_5}, {p_6}, {p_7}\}$, $P=\{ {p_1}, {p_2}, {p_3}, {p_4}\}$, and $C=\{\{ {p_1}, {p_3}, {p_6}\},\{ {p_2}, {p_1}, {p_5}\},\{ {p_3}, {p_4}, {p_2}\},\{ {p_4}, {p_1}, {p_7}\}\}.$
The associated graph on $V(G)=P$ %$G$ has the vertex set $V(G)=P$, 
has the edge set $E(G)=\{ ({p_1}, {p_3}), ({p_2}, {p_1}), ({p_3}, {p_2}), ({p_3}, {p_4}), ({p_4}, {p_1})\}$ and the cycles 
$\{ ({p_3}, {p_2}, {p_1}), ({p_3}, {p_4}, {p_1})\}$.
By Theorem~\ref{thm ci}, for any collection of vectors 
$\{q_{1},q_{2},q_{3},q_{4}\}$ in $\CC^{3}$, the following polynomial is in $I_M$.
\begin{equation*}
\begin{gathered}
\corch{ {p_3}, {p_4},q_{3}}\corch{ {p_2}, {p_5},q_{2}}\corch{ {p_1}, {p_6},q_{1}}\corch{ {p_1}, {p_7},q_{4}}-\corch{ {p_3}, {p_2},q_{3}}\corch{ {p_4}, {p_7},q_{4}}\corch{ {p_1}, {p_6},q_{1}}\corch{ {p_1}, {p_5},q_{2}}\\
 -\corch{ {p_3}, {p_6},q_{1}}\corch{ {p_2}, {p_4},q_{3}}\corch{ {p_1}, {p_5},q_{2}}\corch{ {p_1}, {p_7},q_{4}}.
\end{gathered}
\end{equation*} 
\end{example}

\begin{example}
Let $M$ be the $4$-paving matroid on $\mathcal{P}=\{p_1,\ldots,p_9\}$ with the set of dependent hyperplanes \[\mathcal{L}=\{\{p_1,p_2,p_3,p_4\},\{p_2,p_5,p_6,p_7\},\{p_3,p_5,p_8,p_9\},\{p_4,p_5,p_6,p_8\},\{p_5,p_1,p_7,p_9\},\{p_6,p_7,p_8,p_9\}\}.\]
Using Lemma~\ref{sub h}, we deduce that $M$ is indeed a $4$-paving matroid. We choose the following sets: %$J=\{6,7,8,9\}$, $P=\{1,2,3,4,5\}$, and the set of circuits 
\begin{equation*}
\begin{gathered}
J=\{p_6,p_7,p_8,p_9\},\ P=\{p_1,p_2,p_3,p_4,p_5\},\ \text{and}\\
C=\{\{p_1,p_2,p_3,p_4\},\{p_2,p_5,p_6,p_7\},\{p_3,p_5,p_8,p_9\},\{p_4,p_5,p_6,p_8\},\{p_5,p_1,p_7,p_9\}\}.
\end{gathered}
\end{equation*} 

The associated graph has the edge set \[E(G)=\{(p_1,p_2),(p_1,p_3),(p_1,p_4),(p_2,p_5),(p_3,p_5),(p_4,p_5),(p_5,p_1)\},\]
and the cycles $\{(p_1,p_2,p_5),(p_1,p_3,p_5),(p_1,p_4,p_5)\}.$
By Theorem~\ref{thm ci}, the following polynomial is in $I_M$ for any collection of vectors 
$\{q_{1},q_{2},q_{3},q_{4},q_{5}\}$ in $\CC^{4}$. 
%for $J$, $P$, and $C$, we find that the polynomial
\begin{equation*}
\begin{aligned}
&[p_1,p_3,p_4,q_{1}][p_2,p_6,p_7,q_{2}]
[p_5,p_8,p_9,q_{3}]
[p_5,p_6,p_8,q_{4}][p_5,p_7,p_9,q_{5}]-\\
&[p_1,p_2,p_4,q_{1}][p_5,p_6,p_7,q_{2}][p_3,p_8,p_9,q_{3}][p_5,p_6,p_8,q_{4}][p_5,p_7,p_9,q_{5}]+
\\ 
&[p_1,p_2,p_3,q_{1}][p_5,p_6,p_7,q_{2}][p_5,p_8,p_9,q_{3}][p_4,p_6,p_8,q_{4}][p_5,p_7,p_9,q_{5}]-\\
&[p_2,p_3,p_4,q_{1}][p_5,p_6,p_7,q_{2}][p_5,p_8,p_9,q_{3}][p_5,p_6,p_8,q_{4}][p_1,p_7,p_9,q_{5}].
\end{aligned}
\end{equation*}
\end{example}

We conclude this paper with a brief remark on potential applications of our results.

\begin{remark}
The lifting construction and graph polynomials introduced in Definitions~\ref{cl 2} and~\ref{cl} may be used to compute complete generating sets for matroid ideals in various families. For instance, this approach has been applied in \cite{liwski2025algebraic} to the Pascal, Pappus, and cactus matroids. These problems are also studied from an algorithmic point of view; see \cite{liwski2025efficient, liwski2025minimal}.
\end{remark}

%\bibliographystyle{abbrv}
%\bibliography{Citation.bib}
\printbibliography

\medskip
\noindent 
\footnotesize{\textbf{Authors' addresses}
\noindent
Department of Mathematics, KU Leuven, Celestijnenlaan 200B, B-3001 Leuven, Belgium
\\ E-mail address: {\tt emiliano.liwski@kuleuven.be}

\medskip  \noindent
Department of Computer Science, KU Leuven, Celestijnenlaan 200A, B-3001 Leuven, Belgium\\ 
   Department of Mathematics, KU Leuven, Celestijnenlaan 200B, B-3001 Leuven, Belgium\\ 
   %UiT – The Arctic University of Norway, 9037 Troms\o, Norway\\
   E-mail address: {\tt fatemeh.mohammadi@kuleuven.be}
\end{document}